\documentclass[10pt]{amsart}

\usepackage[utf8]{inputenc} 
\usepackage[dvipsnames, table]{xcolor} 
\usepackage{color} 
\usepackage{tikz} 
\usepackage{tikz-cd} 
\usepackage{amssymb, amsmath, mathtools, amsfonts, amsthm} 
\usepackage[cm]{fullpage} 
\usepackage{enumitem} 
\usepackage{multicol} 
\usepackage{hyperref} 
\usepackage{cleveref} 
\definecolor{header_color}{RGB}{255,255,240}
\definecolor{mylinkcolor}{RGB}{0,0,255}
\definecolor{mycitecolor}{RGB}{169,169,169}
\definecolor{myurlcolor}{RGB}{255,20,147} \hypersetup{colorlinks=true,
  urlcolor=myurlcolor, citecolor=mycitecolor, linkcolor=mylinkcolor,
  linktoc=page, breaklinks=true} \usepackage{url} 
\usepackage[numbers]{natbib} 
\usepackage{setspace}
\usepackage{indentfirst} 
\usepackage{longtable} 
\usepackage{colonequals} 
\usepackage{caption} 
\usepackage{subcaption} 
\usepackage{float}

\newcounter{counter}[subsection] 
\renewcommand{\thecounter}{\arabic{section}.\arabic{subsection}.\arabic{counter}}
\newtheorem{theorem}[counter]{Theorem}

\newtheorem{lemma}[counter]{Lemma}
\newtheorem{coro}[counter]{Corollary}

\newtheorem{prop}[counter]{Proposition}
 \theoremstyle{definition}
\newtheorem{defn}[counter]{Definition}
\newtheorem{example}[counter]{Example}

\theoremstyle{remark} 
\newtheorem*{notation}{Notation} \newtheorem{remark}[counter]{Remark}
\newtheorem{remarks}[counter]{Remarks}

\newcommand{\github}{\textsc{GitHub} }
\newcommand{\lmfdb}{\textsc{LMFDB} }

\newcommand{\sage}{\textsc{SageMath} }

\newcommand{\RR}{\mathbf{R}} 
\newcommand{\ZZ}{\mathbf{Z}} 
\newcommand{\QQ}{\mathbf{Q}} 
\newcommand{\CC}{\mathbf{C}} 
\newcommand{\FF}{\mathbf{F}} 
\newcommand{\Qbar}{\overline{\QQ}} 
\newcommand{\md}{\text{ mod }} 
\newcommand{\inv}{^{-1}} 
\newcommand{\unit}{^{\times}} 
\newcommand{\dd}{\,\mathrm{d}} 
\newcommand{\paren}[1]{\left( #1 \right)} 
\newcommand{\brk}[1]{\left\lbrace #1 \right\rbrace} 
\newcommand{\cdef}[1]{\textsf{#1}} 

\newcommand{\ol}{\overline}
\newcommand{\mbf}[1]{\mathbf{#1}} 
\newcommand{\msf}[1]{\mathsf{#1}} 
\newcommand{\mand}{\text{ and }} 
\newcommand{\mfor}{\text{ for }} 

\newcommand{\SL}{\msf{SL}} 
\newcommand{\UU}{\msf{U}} 
\newcommand{\USp}{\msf{USp}} 
\newcommand{\SF}{\msf{SF}} 

\let\originalleft\left \let\originalright\right
\renewcommand{\left}{\mathopen{}\mathclose\bgroup\originalleft}
  \renewcommand{\right}{\aftergroup\egroup\originalright}

\DeclareMathOperator{\ord}{ord} 
\DeclareMathOperator{\Aut}{Aut} 
\DeclareMathOperator{\Gal}{Gal} 
\DeclareMathOperator{\Tr}{Tr} 
\DeclareMathOperator{\rank}{rk} 
\DeclareMathOperator{\lcm}{lcm} 
\DeclareMathOperator{\diag}{diag} 
\DeclareMathOperator{\sheafHom}{\mathscr{H}\text{\kern -1pt
    {om}}} 

\newcommand{\smallmat}[4]{\bigl(\begin{smallmatrix}#1&#2\\#3&#4\end{smallmatrix}\bigr)}

\title{Frobenius distributions of low dimensional\\ abelian varieties over
  finite fields} 

\author{Santiago Arango-Piñeros} \address{Department of Mathematics, Emory
  University, Atlanta, GA 30322, USA} \email{santiago.arango.pineros@gmail.com}
\urladdr{\url{https://sarangop1728.github.io/}}

\author{Deewang Bhamidipati} \address{Mathematics Department, University of
  California, Santa Cruz, CA 95064, USA} \email{dbhamidi@ucsc.edu}
\urladdr{\url{https://bdeewang.com/}}

\author{Soumya Sankar} \address{Mathematical Institute, Utrecht University,
  Hans Freudenthal building, Budapest 6, 3584 CD Utrecht, The Netherlands}
\email{s.sankar@uu.nl}
\urladdr{\url{https://sites.google.com/site/soumya3sankar/}}

\begin{document}
\begin{abstract}
  Given a $g$-dimensional abelian variety $A$ over a finite field $\FF_q$, the
  Weil conjectures imply that the normalized Frobenius eigenvalues generate a
  multiplicative group of rank at most $g$. The Pontryagin dual of this group
  is a compact abelian Lie group that controls the distribution of high powers
  of the Frobenius endomorphism. This group, which we call the Serre--Frobenius
  group, encodes the possible multiplicative relations between the Frobenius
  eigenvalues. In this article, we classify all possible Serre--Frobenius
  groups that occur for $g \le 3$. We also give a partial classification for
  simple ordinary abelian varieties of prime dimension $g\geq3$.
\end{abstract}

\maketitle

\setcounter{tocdepth}{1}

\sloppy

\section{Introduction}
\label{sec:intro}

Let $E$ be an elliptic curve over a finite field $\FF_q$ of characteristic
$p>0$. The zeros $\alpha_1, \overline{\alpha}_1$ of the characteristic
polynomial of Frobenius acting on the Tate module of $E$ are complex numbers of
absolute value $\sqrt{q}$. Consider $u_1 \colonequals \alpha_1/\sqrt{q}$ and
$\overline{u}_1$ the normalized zeros in the unit circle $\UU(1)$. The curve
$E$ is \emph{ordinary} if and only if $u_1$ is not a root of unity, and in this
case, the sequence $(u_1^r)_{r=1}^\infty$ is equidistributed in $\UU(1)$.
Further, the normalized Frobenius traces
$x_r \colonequals u_1^r + \overline{u}_1^{r}$ are equidistributed on the
interval $[-2,2]$ with respect to the pushforward of the probability Haar
measure on $\UU(1)$ via $u \mapsto u + \overline{u}$, namely
\begin{equation} \label{eq:lambda1} \lambda_1(x) \colonequals
  \frac{\mathrm{d}x}{\pi\sqrt{4-x^2}}, \end{equation} where $\dd x$ is the
restriction of the Lebesgue measure to $[-2,2]$ (see \cite[Proposition
2.2]{fite2015equidistribution}).

In contrast, if $E$ is supersingular, the sequence $(u_1^r)_{r=1}^{\infty}$
generates a finite cyclic subgroup of order $m$, $C_m \subset \UU(1)$. In this
case, the normalized Frobenius traces are equidistributed with respect to the
pushforward of the uniform measure on $C_m$.

This dichotomy branches out in an interesting way for abelian varieties of
higher dimension $g > 1$: potential non-trivial multiplicative relations
between the Frobenius eigenvalues
$\alpha_1, \overline{\alpha}_1, \dots, \alpha_g, \overline{\alpha}_g $ increase
the complexity of the problem of classifying the distribution of normalized
traces of high powers of Frobenius,
\begin{equation}
  x_r \colonequals (\alpha_1^r + \overline{\alpha}_1^r + \cdots + \alpha_g^r + \overline{\alpha}_g^r)/q^{r/2} \in [-2g, 2g], \mfor r \geq 1.
  \label{eq:normalized-traces}
\end{equation}
In analogy with the case of elliptic curves, we identify a compact abelian
subgroup of $\USp_{2g}(\CC)$ controlling the distribution of Sequence
(\ref{eq:normalized-traces}) via pushforward of the Haar measure. In this
article, we provide a complete classification of the conjugacy class of this
subgroup, which we call the \emph{Serre--Frobenius group}, for abelian
varieties of dimension up to 3. We do this by classifying the possible
multiplicative relations between the Frobenius eigenvalues. This classification
provides a description of all the possible distributions of Frobenius traces in
these cases (see \Cref{cor:equidist}). We also provide a partial classification
for simple ordinary abelian varieties of odd prime dimension.

\begin{defn}[Serre--Frobenius group]
  \label{def:SF-group}
  Let $A$ be an abelian variety of dimension $g$ over $\FF_q$. Let
  $\alpha_1, \alpha_2 \ldots, \alpha_g, \overline{\alpha}_1,
  \overline{\alpha}_2 \ldots \overline{\alpha}_g $ denote the eigenvalues of
  Frobenius. Let $u_i = \alpha_i/\sqrt{q}$ denote the normalized Frobenius
  eigenvalues. The \cdef{Serre--Frobenius group of $A$}, denoted by $\SF(A)$,
  is the closure of the subgroup of $\USp_{2g}(\CC)$ generated by the diagonal
  matrix $\mathrm{diag}(u_1, \dots, u_g, \ol{u}_1, \dots, \ol{u}_g)$. This
  group is well defined up to relabelling of the eigenvalues of Frobenius (see
  \Cref{rmk:relabelling}).
\end{defn}

The classification of Serre--Frobenius groups relies crucially on the relation
between the Serre--Frobenius group and the multiplicative subgroup of
$U_A \subset \CC\unit$ generated by the normalized eigenvalues
$u_1, \dots, u_g$. Indeed, the isomorphism class of the former is determined by
the Pontryagin dual of the latter (see \Cref{lemma:equiv-char-of-SFgroup}). The
rank of the group $U_A$ is called the \cdef{angle rank} of the abelian variety
and the order of the torsion subgroup is called the \cdef{angle torsion order}.
The relation between $\SF(A)$ and the group generated by the normalized
eigenvalues gives us the following structure theorem.

\begin{theorem}
  \label{mainthm:structure-thm}
  Let $A$ be an abelian variety defined over $\FF_q$. Then
  \[\SF(A) \cong \UU(1)^\delta \times C_m,\]
  where $\delta = \delta_A$ is the angle rank and $m = m_A$ is the angle
  torsion order. Furthermore, the connected component of the identity is
  $\SF(A)^\circ = \SF(A\times_{\FF_q} \FF_{q^m})$.
\end{theorem}

By definition, the Serre--Frobenius group carries the data of the embedding
into the \(\USp_{2g}(\CC)\), which in turn is captured by the relations among
the Frobenius eigenvalues. While in general these relations can be hard to pin
down (see for instance, \cite[Theorem 3.25]{dupuy2022angle}), in our cases, we
are able to write them down explicitly and use them to deduce the angle torsion
order. In particular, we classify the Serre--Frobenius groups of abelian
varieties of dimension $g \le 3$.

\begin{theorem}[Elliptic curves]
  \label{mainthm:elliptic-curves}
  Let $E$ be an elliptic curve defined over $\FF_q$. Then
  \begin{enumerate}
  \item $E$ is ordinary if and only if $\SF(E) = \UU(1)$.
  \item $E$ is supersingular if and only if
    $\SF(E) \in \brk{C_1, C_2, C_3, C_4, C_6, C_8, C_{12}}$.
  \end{enumerate}
  Here, $C_m$ is the subgroup of $\USp_2(\CC) = \SL_2(\CC)$ generated by
  $\smallmat{\zeta_m}{0}{0}{{\zeta}^{-1}_m}$ for $\zeta_m$ a primitive $m$-th
  root, and
  $\UU(1) = \brk{\smallmat{u}{0}{0}{\ol{u}} : u \in \CC^\times, |u| = 1}$.
  Moreover, each one of these groups is realized for some prime power $q$.
\end{theorem}

We note that the classification of supersingular Serre--Frobenius groups of
elliptic curves follows from Deuring \cite{deuring1941typen} and Waterhouse's
\cite{waterhouse1969abelian} classification of Frobenius traces (see also
\cite[Section 14.6]{oortAbvar} and \cite[Theorem 2.6.1]{serre2020rational}).

\begin{theorem}[Abelian surfaces]
  \label{mainthm:surfaces}
  Let $S$ be an abelian surface over $\FF_q$. Then $S$ has Serre--Frobenius
  group according to \Cref{fig:proof-2}. In particular, the possible options
  for the connected component of the identity $\SF(S)^{\circ}$, and the size of
  the cyclic component group $\SF(S)/\SF(S)^\circ \cong C_m$ are given below.
  Moreover, each one of these groups is realized for some prime power $q$.
\end{theorem}
\begin{table}[h]
  \setlength{\arrayrulewidth}{0.3mm} \setlength{\tabcolsep}{5pt}
  \renewcommand{\arraystretch}{1.2}
  \begin{tabular}{|c|c|}
    \hline
    \rowcolor{header_color} 
    $\SF(S)^\circ$ & $ m$ \\ \hline
    $1$            &          1,2,3,4,5,6,8,10,12,24                 \\ \hline
    $\UU(1)$       &       1,2,3,4,6,8,12            \\ \hline
    $\UU(1)^2$     &          1                \\ \hline
  \end{tabular}
\end{table}

\begin{theorem}[Abelian threefolds]
  \label{mainthm:threefolds}
  Let $X$ be an abelian threefold over $\FF_q$. Then, $X$ has Serre--Frobenius
  group according to \Cref{fig:proof-3}. In particular, the possible options
  for the connected component of the identity, $\SF(X)^{\circ}$, and the size
  of the cyclic component group $\SF(X)/\SF(X)^\circ \cong C_m$ are given
  below. Moreover, each one of these groups is realized for some prime power
  $q$.
\end{theorem}
\begin{table}[h]
  \setlength{\arrayrulewidth}{0.3mm} \setlength{\tabcolsep}{5pt}
  \renewcommand{\arraystretch}{1.2}
  \begin{tabular}{|c|c|}
    \hline
    \rowcolor{header_color} 
    $\SF(X)^\circ$ & $m$ \\ \hline
    $1$            &       1,2,3,4,5,6,7,8,9,10,12,14,15,18,20,24,28,30,36 \\ \hline
    $\UU(1)$       &   1,2,3,4,5,6,7,8,10,12,24                            \\ \hline
    $\UU(1)^2$     &       1,2,3,4,6,8,12,24                               \\ \hline
    $\UU(1)^3$     &       1                                               \\ \hline
  \end{tabular}
\end{table}

If $g$ is an odd prime, we have the following classification for simple
ordinary abelian varieties; in the following theorem, we say that an abelian
variety $A$ \cdef{splits} over a field extension $\FF_{q^m}$ if $A$ is
isogenous over $\FF_{q^m}$ to a product of proper abelian subvarieties.

\begin{theorem}[Prime dimension]
  \label{mainthm:odd-prime}
  Let $A$ be a simple ordinary abelian variety defined over $\FF_q$ of
  \cdef{prime} dimension $g > 2$. Then, exactly one of the following conditions
  holds.
  \begin{enumerate}
  \item $A$ is absolutely simple.
  \item $A$ splits over a degree $g$ extension of $\FF_q$ as a power of an
    elliptic curve, and $\SF(A) \cong \UU(1)\times C_g$.
  \item $A$ splits over a degree $2g + 1$ extension of $\FF_q$ as a power of an
    elliptic curve, and $\SF(A) \cong \UU(1)\times C_{2g+1}$. This case only
    occurs if \(2g+1\) is also a prime, i.e., if \(g\) is a Sophie Germain
    prime.
  \end{enumerate}
\end{theorem}

\subsection{Application to distributions of Frobenius traces}
\label{subsec:application-to-traces}

Our results can be applied to understanding the distribution of Frobenius
traces of an abelian variety over $\FF_q$ as we range over finite extensions of
the base field. Indeed, for each integer $r \ge 1$, we may rewrite Equation
(\ref{eq:normalized-traces}) as
\[
  x_r = u_1^r +\overline{u}_1^r + \cdots + u_g^r +\overline{u}_g^r \in [-2g,
  2g]
\]
denote the \cdef{normalized Frobenius trace} of the base change of an abelian
variety $A$ to $\FF_{q^r}$.

In \cite{ahmadi2010shparlinski}, the authors study Jacobians of smooth
projective genus $g$ curves with maximal angle rank\footnote{In their notation,
  this is the condition that the Frobenius angles are linearly independent
  modulo 1.} and show that the sequence $(x_r/2g)_{r = 1}^\infty$ is
equidistributed on $[-1,1]$ with respect to an explicit measure. The
Serre--Frobenius group enables us to remove the assumption of maximal angle
rank.

\begin{coro}
  \label{cor:equidist}
  Let $A$ be a $g$-dimensional abelian variety defined over $\FF_q$. Then, the
  sequence $(x_r)_{r=1}^\infty$ of normalized traces of Frobenius is
  equidistributed in $[-2g,2g]$ with respect to the pushforward of the Haar
  measure on $\SF(A) \subseteq \USp_{2g}(\CC)$ via the trace
  \begin{equation}
    \label{eq:pushforward}
    \SF(A) \subseteq \USp_{2g}(\CC) \to [-2g,2g], \quad M \mapsto \Tr(M).
  \end{equation}
\end{coro}

The classification of the Serre--Frobenius groups in our theorems can be used
to distinguish between the different Frobenius trace distributions occurring in
each dimension.

\begin{example}
  \label{example:S-iso-ExE}
  Let $S$ be a simple abelian surface over $\FF_q$ with Frobenius eigenvalues
  $R_S = \brk{\alpha_1,\alpha_2, \overline{\alpha}_1, \overline{\alpha}_2}$ and
  suppose that $S_{(2)} \colonequals S \times_{\FF_q} \FF_{q^2}$ is isogenous
  to $E^2$ for some ordinary elliptic curve $E/\FF_{q^2}$. In this case,
  $\brk{\alpha_1^2, \overline{\alpha}_1^2} = R_E = \brk{\alpha_2^2,
    \overline{\alpha}_2^2}$. Normalizing, and possibly after re-indexing, we
  see that either \(u_2 = u_1\) or \(u_2 = -u_1\). Since \(S\) is simple and
  ordinary, the characteristic polynomial of Frobenius of $S$ is irreducible
  (see \Cref{rmk:simple-ord-irreducible}), and we must have \(u_2 = -u_1\). The
  Serre--Frobenius groups of $S$ and $S_{(2)}$ are calculated as follows.
  \begin{align*}
    \SF(S) &= \overline{\brk{
             \begin{bmatrix}
               u_1^r & & &  \\
                     & (-u_1)^r & &  \\
                     & & \ol{u}_1^r &  \\
                     & & & (-\ol{u}_1)^r
             \end{bmatrix}
             : r \in \ZZ }} 
             = \brk{
             \begin{bmatrix}
               u & & &  \\
                 & -u & &  \\
                 & & \ol{u} &  \\
                 & & & -\ol{u}
             \end{bmatrix}
             : u \in \UU(1) }, \\
    \SF(S_{(2)}) &= \overline{\brk{
             \begin{bmatrix}
               u_1^{2r} & & &  \\
                        & (-u_1)^{2r} & &  \\
                        & & \ol{u}_1^{2r} &  \\
                        & & & (-\ol{u}_1)^{2r}
             \end{bmatrix}
                   : r \in \ZZ }} 
                   = \brk{
             \begin{bmatrix}
               u & & &  \\
                 & u & &  \\
                 & & \ol{u} &  \\
                 & & & \ol{u}
             \end{bmatrix}
                   : u \in \UU(1) }.
  \end{align*}
  The sequence of normalized traces, henceforth referred to as the
  $a_1$-sequence, is given by $x_r(S) = 0$ when $r$ is odd, and
  $x_r(S) = 2u_1^r + 2\overline{u}_1^r$ when $r$ is even. Extending the base
  field to $\FF_{q^2}$ yields the sequence of normalized traces
  $x_r(S_{(2)}) = x_{2r}(S) = 2x_r(E)$. The data of the embedding
  $\SF(S) \subseteq \USp_{4}(\CC)$ precisely captures the (non-trivial)
  multiplicative relations between the Frobenius eigenvalues.
  
  The sequence of normalized traces $x_r(S)$ is equidistributed with respect to
  the pushforward of the Haar measure under the trace map
  $\SF(S) \subseteq \USp_4(\CC) \to [-4,4]$ given by
  $\diag(z_1, z_2, \ol{z}_1, \ol{z}_2) \mapsto z_1 + z_2 + \overline{z}_1 +
  \overline{z}_2$, and similarly for \(S_{(2)}\). These can be computed
  explicitly for $S$ and $S_{(2)}$ as
  \begin{equation}
    \tfrac{1}{2}\delta_0 + \frac{\dd x}{2\pi\sqrt{16-x^2}} \quad \mand \quad  \frac{\dd x}{\pi\sqrt{16-x^2}},
  \end{equation}
  where $\dd x$ is the restriction of the Haar measure to $[-4,4]$, and
  $\delta_0$ is the Dirac measure supported at $0$.

  For instance, choose the surface $S$ to be in the isogeny class with \lmfdb
  \cite{lmfdb} label\footnote{Recall the
    \href{https://www.lmfdb.org/Variety/Abelian/Fq/Labels}{labelling
      convention} for isogeny classes of abelian varieties over finite fields
    in the \lmfdb is \texttt{g.q.iso} where \texttt{g} is the dimension,
    \texttt{q} is the cardinality of the base field, and \texttt{iso} specifies
    the isogeny class by writing the coefficients of the Frobenius polynomial
    in base 26.}
  \href{https://www.lmfdb.org/Variety/Abelian/Fq/2/5/a_ab}{\texttt{2.5.a\_ab}}
  and Weil polynomial $P(T) = T^4 - T^2 + 25$. This isogeny class is ordinary
  and simple, but not geometrically simple. Indeed, $S_{(2)}$ is in the isogeny
  class $\texttt{1.25.ab}^2 =$
  \href{https://www.lmfdb.org/Variety/Abelian/Fq/2/25/ac_bz}{\texttt{2.25.ac\_bz}}
  corresponding to the square of an ordinary elliptic curve. The corresponding
  $a_1$-histograms describing the frequency of the sequence
  $(x_r)_{r=1}^\infty$ are depicted in \Cref{fig:example}. Each graph
  represents a histogram of $16^6 = 16777216$ samples placed into $4^6 = 4096$
  buckets partitioning the interval $[-2g,2g]$. The vertical axis has been
  suitably scaled, with the height of the uniform distribution, $1/4g$,
  indicated by a gray line.

\begin{figure}[H]
  \centering
  \begin{subfigure}{0.49\textwidth}
    \centering \includegraphics[scale=0.35]{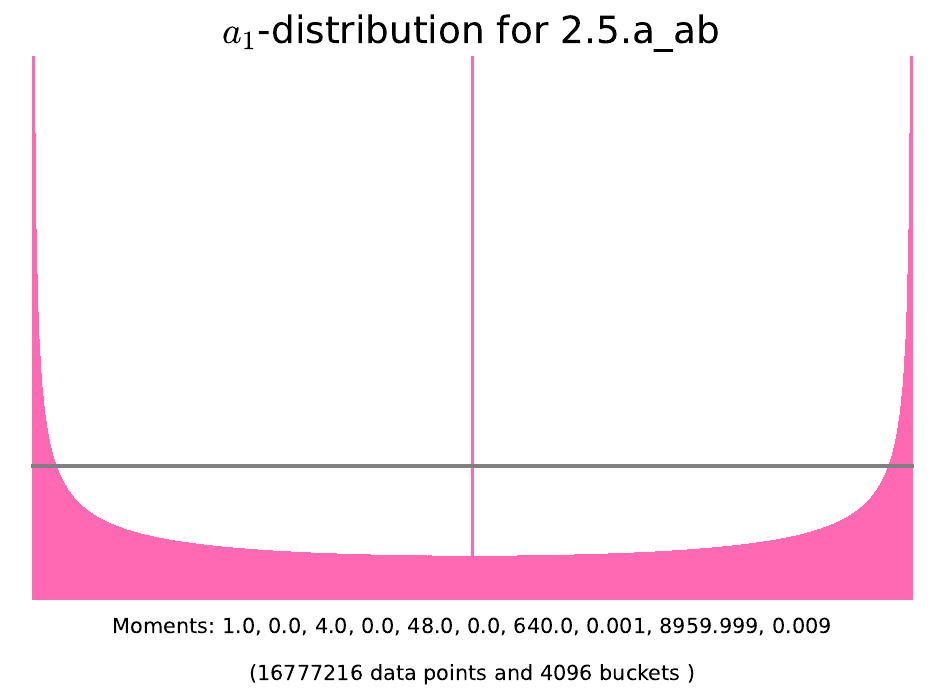}
  \end{subfigure}
  \begin{subfigure}{0.49\textwidth}
    \centering \includegraphics[scale=0.35]{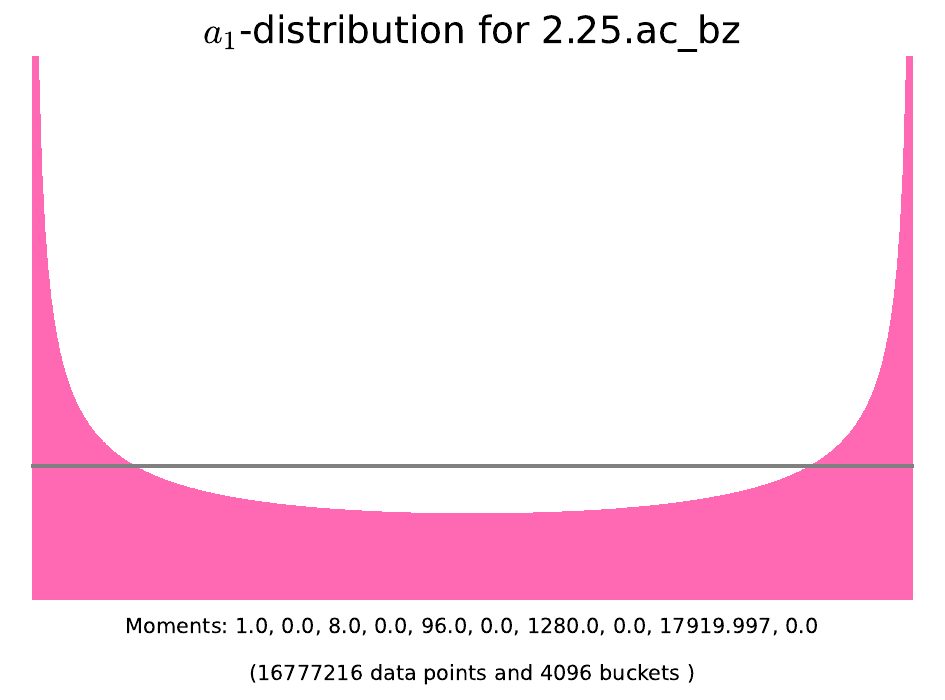}
  \end{subfigure}
  \caption{$a_1$-histograms for
    \href{https://www.lmfdb.org/Variety/Abelian/Fq/2/5/a_ab}{\texttt{2.5.a\_ab}}
    and
    \href{https://www.lmfdb.org/Variety/Abelian/Fq/2/25/ac_bz}{\texttt{2.25.ac\_bz}}.}
  \label{fig:example}
\end{figure}
\end{example}

\subsection{Relation to other work} The reason for adopting the name
``Serre--Frobenius group'' is that the Lie group $\SF(A)$ is closely
related to Serre's \cdef{Frobenius torus} \cite{serre_oeuvres}, as
explained in \Cref{rmk:frob-torus}.

\subsubsection{Angle rank}
\label{subsubsec:angle-rank-survey}

In this article, we study multiplicative relations between Frobenius
eigenvalues, a subject studied extensively by Zarhin \cite{
  Zarhin1991K3, Zarhin1990K3, lenstra1993, Zarhin1994-tate-nonsimple,
  Zarhin2015EigenFrob}. Our classification relies heavily on being
able to understand multiplicative relations in low dimension, and we
use results of Zarhin in completing parts of it. The number of
multiplicative relations is quantified by the angle rank, an invariant
studied in \cite{dupuy2022angle}, \cite{dupuy2022LMFDB} for absolutely
simple abelian varieties by elucidating its interactions with the
Galois group and Newton polygon of the Frobenius polynomial. We study
the angle rank as a stepping stone to classifying the full
Serre--Frobenius group. While our perspective differs from that in
\cite{dupuy2022angle}, the same theme is continued here: the
Serre--Frobenius groups depend heavily on the Galois group of the
Frobenius polynomial. It is worth noting that here that the results
about the angle rank in the non-absolutely simple case cannot be
pieced together by knowing the results in the absolutely simple cases
(for instance, see Zywina's exposition of Shioda's example
\cite[Remark 1.16]{zywina-monodromy}).

\subsubsection{Sato--Tate groups}
\label{subsec:Sato-Tate-survey}

The Sato--Tate group of an abelian variety defined over a number field
controls the distribution of the Frobenius of the reduction modulo
prime ideals, and it is defined via its $\ell$-adic Galois
representation (see \cite[Section 3.2]{sutherland2013sato}). The
Serre--Frobenius group can also be defined via $\ell$-adic
representations in an analogous way: it is conjugate to a maximal
compact subgroup of the image of Galois representation
$\rho_{A,\ell} \colon \Gal(\overline{\FF}_q/\FF_q) \to \Aut(V_\ell
A)\otimes \CC$, where $V_{\ell}A$ is the $\ell$-adic Tate vector
space. Therefore it is natural to expect that the Sato--Tate and the
Serre--Frobenius group are related to each other. The following
observations support this claim:
\begin{itemize}[leftmargin=*]
\item Assuming standard conjectures, the connected component of the
  identity of the Sato--Tate group can be recovered from knowing the
  Frobenius polynomial at two suitably chosen primes (\cite[Theorem
  1.6]{zywina-monodromy}).
\item Several abelian Sato--Tate groups (see
  \cite{fite_kedlaya_rotger_sutherland_2012, fite2023satotate}) appear
  as Serre--Frobenius groups of abelian varieties over finite
  fields. The ones with maximal angle rank are as below.
  \begin{itemize}
  \item $\UU(1)$ is the Sato--Tate group of an elliptic curve with
    complex multiplication over any number field that contains the CM
    field (see
    \href{https://www.lmfdb.org/SatoTateGroup/1.2.B.1.1a}{\texttt{1.2.B.1.1a}}). It
    is also the Serre--Frobenius group of any ordinary elliptic curve
    (see \Cref{fig:ordinary-ec}), and the $a_1$-moments coincide.
  \item $\UU(1)^2$ is the Sato--Tate group of weight 1 and degree 4
    (see
    \href{https://www.lmfdb.org/SatoTateGroup/1.4.D.1.1a}{\texttt{1.4.D.1.1a}}). It
    is also the Serre--Frobenius group of an abelian surface with
    maximal angle rank (see \Cref{fig:simple-ord-S}), and the
    $a_1$-moments coincide.
  \item $\UU(1)^3$ is the Sato--Tate group of weight 1 and degree 6
    (see
    \href{https://www.lmfdb.org/SatoTateGroup/1.6.H.1.1a}{\texttt{1.6.H.1.1a}}). It
    is also the Serre-Frobenius group of abelian threefolds with
    maximal angle rank (see \Cref{fig:hist-simple-ord-X}), and the
    $a_1$-moments coincide.
  \end{itemize}
\end{itemize}

\subsection{Outline}
In \Cref{sec:frob-mult-groups}, we give some background on abelian
varieties over finite fields, expand on the definition of the
Serre--Frobenius group, and describe how it controls the distribution
of traces of high powers of Frobenius.  In \Cref{sec:generalities}, we
prove some preliminary results on the geometric isogeny types of
abelian varieties of dimension $g \le 3$ and $g$ odd prime.  We also
recall some results about Weil polynomials of supersingular abelian
varieties, and Zarhin's notion of neatness. In
\Cref{sec:simple-ord-odd-prime}, we discuss the classification in the
case of simple ordinary abelian varieties of odd prime dimension. In
\Cref{sec:elliptic-curves}, \Cref{sec:surfaces}, and
\Cref{sec:threefolds}, we give a complete classification of the
Serre--Frobenius group for dimensions 1, 2, and 3 respectively.  A
list of tables containing different pieces of the classification
follows this section.

\subsection{Notation} \label{subsec:notation} Throughout this paper,
$A$ will denote a $g$-dimensional abelian variety over a finite field
$\FF_q$ of characteristic $p$. The polynomial
$P_A(T) = \sum_{i=0}^{2g} a_iT^{2g - i}$ will denote the
characteristic polynomial of the $q$-Frobenius endomorphism $\pi_A$
acting on the Tate module of $A$, and $h_A(T)$ its minimal
polynomial. The set of roots of $P_A(T)$ is denoted by $R_A$. We
usually write
$\alpha_1,\overline{\alpha}_1 \dots, \alpha_g, \overline{\alpha}_g \in
R_A$ for the Frobenius eigenvalues, where
$\overline{\alpha}_i = q/\alpha_i$. In the case that $P_A(T)$ is a
power of $h_A(T)$, we will denote this power by $e_A$ (See
\ref{subsec:background-av}). The subscript $(\cdot)_{(r)}$ will denote
the base change of any object or map to $\FF_{q^r}$. The group $U_A$
will denote the multiplicative group generated by the normalized
eigenvalues of Frobenius, $\delta_A$ its rank and $m_A$ the order of
its torsion subgroup. The group $\Gamma_A$ will denote the
multiplicative group generated by
$\{ \alpha_1, \alpha_2 \ldots \alpha_{g}, q\}$. In
\Cref{sec:surfaces}, $S$ will be used to denote an abelian surface,
while in \Cref{sec:threefolds}, $X$ will be used to denote a
threefold.

\newpage
\tableofcontents
\listoftables \listoffigures

\newpage
\section{Frobenius multiplicative groups}
\label{sec:frob-mult-groups}

In this section we introduce the Serre--Frobenius group of $A$ and
explain how it is related to Serre's theory of Frobenius tori
\cite{serre_oeuvres}. We do this from the perspective of the theory of
algebraic groups of multiplicative type, as in \cite[Chapter
12]{milne2017algebraic}. We start by recalling some facts about
abelian varieties over finite fields.

\subsection{Background on Abelian varieties over finite fields}
\label{subsec:background-av}

Fix $A$ a $g$ dimensional abelian variety over $\FF_q$. A
\cdef{$q$-Weil number} is an algebraic integer $\alpha$ such that
$|\phi(\alpha)| = \sqrt{q}$ for every embedding
$\phi\colon \QQ(\alpha) \to \CC$. Let $P_A(T)$ denote the
characteristic polynomial of the Frobenius endomorphism acting on the
$\ell$-adic Tate module of $A$. The polynomial $P_A(T)$ is monic of
degree $2g$, and Weil \cite{weil_numberofsols} showed that its roots
are $q$-Weil numbers; we denote the set of roots of $P_A(T)$ by
$R_A \colonequals \{ \alpha_1, \alpha_2 \ldots, \alpha_g,
\alpha_{g+1}, \dots, \alpha_{2g}\}$ with
$\ol{\alpha}_j \colonequals \alpha_{g+j} = q/\alpha_j$ for
$j \in \brk{1, \dots, g}$. The seminal work of Honda
\cite{honda1968isogeny} and Tate \cite{tate1966endomorphisms,
  tate1971classes} classifies the isogeny decomposition type of $A$ in
terms of the factorization of $P_A(T)$. In particular, if $A$ is
simple, we have that $P_A(T) = h_A(T)^{e_A}$ where $h_A(T)$ is the
\cdef{minimal polynomial} of the Frobenius endomorphism and $e_A$ is
the degree, i.e., the square root of the dimension of the central
simple algebra
$\mathrm{End}^0(A) \colonequals \mathrm{End}(A)\otimes\QQ$ over its
center. The Honda--Tate theorem gives a bijective correspondence
between isogeny classes of simple abelian varieties over $\FF_q$ and
conjugacy classes of $q$-Weil numbers, sending the isogeny class
determined by $A$ to the set of roots $R_A$. Further, the isogeny
decomposition $A \sim A_1 \times A_2 \ldots \times A_k$ can be read
from the factorization $P_A(T) = \prod_{i=1}^{k} P_{A_i}(T)$.

Writing $P_A(T) = \sum_{i=0}^{2g} a_i T^{2g-i}$, the \cdef{$q$-Newton
  polygon} of $A$ is the lower convex hull of the set of points
$\{(i, \nu(a_i)) \in \RR^2 : a_i \neq 0\}$ where $\nu$ is the $p$-adic
valuation normalized so that $\nu(q)=1$. The Newton polygon is isogeny
invariant. Define the \cdef{$p$-rank} of $A$ as the number of slope
$0$ segments of the Newton polygon. An abelian variety is called
\cdef{ordinary} if it has maximal $p$-rank, i.e., its $p$-rank is
equal to $g$. It is called \cdef{almost ordinary} if it has \(p\)-rank
\(g-1\); equivalently, the set of slopes of its Newton polygon is
$\brk{0, 1/2, 1}$ and the slope $1/2$ has length $2$.  An abelian
variety is called \cdef{supersingular} if all the slopes of the Newton
polygon are equal to $1/2$. The field
$L = L_A \colonequals \QQ(\alpha_1, \dots, \alpha_g)$ is the splitting
field of the Frobenius polynomial. By definition, the Galois group
$\Gal(L/\QQ)$ acts on the roots $R_A$ by permuting them.

\begin{remark}
  \label{rmk:simple-ord-irreducible}
  When an abelian variety \(A\) over \(\FF_q\) is simple and ordinary,
  then \(P_A(T)\) is irreducible and its endomorphism algebra is a
  field (\cite[Theorem 7.2]{waterhouse1969abelian}).
\end{remark}

\begin{notation}
  Whenever $A$ is fixed or clear from context, we will omit the
  subscript corresponding to it from the notation described above. In
  particular, we will use $P(T), h(T)$ and $e$ instead of
  $P_A(T), h_A(T)$ and $e_A$.
\end{notation}

\subsection{Angle groups}
\label{sec:angle_groups} Denote by $\Gamma \colonequals \Gamma_A$ the
multiplicative subgroup of $\CC\unit$ generated by the set of
Frobenius eigenvalues $R_A$, and let
$\Gamma_{(r)}\colonequals \Gamma_{A_{(r)}}$ for every $r \geq
1$. Since $\alpha \mapsto q/\alpha$ is a permutation of $R_A$, the set
$\brk{\alpha_1, \dots, \alpha_g, q}$ is a set of generators for
$\Gamma$; that is, every $\gamma \in \Gamma$ can be written as
\begin{equation}
  \label{eq:mult-relations}
  \gamma = q^k \prod_{j=1}^g \alpha_j^{k_j}
\end{equation}
for some $(k, k_1, \dots, k_{g}) \in \ZZ^{g+1}$.

Since $\Gamma$ is a subgroup of $\Qbar\unit$, it is naturally a
$\Gal(\overline{\QQ}/\QQ)$-module. However, this perspective is not
necessary for our applications. This group is denoted as $\Phi_A$ in
\cite{zywina-monodromy}.
  
\begin{defn}
  \label{def:angle-group}
  We define the \cdef{angle group} of $A$ to be $U \colonequals U_A$,
  the multiplicative subgroup of $\UU(1)$ generated by the unitarized
  eigenvalues
  $\{ u_j \colonequals \alpha_j/\sqrt{q}: j = 1, \dots, g$\}. When $A$
  is fixed, for every $r \geq 1$ we abbreviate
  $U_{(r)} \colonequals U_{A_{(r)}}$.
\end{defn}

\begin{defn}
  \label{defn:angle-rank}
  The \cdef{angle rank} of an abelian variety $A/\FF_q$ is the rank of
  the finitely generated abelian group $U_A$. It is denoted by
  $\delta_A \colonequals \rank U_A$. The \cdef{angle torsion order}
  $m_A$ is the order of the torsion subgroup of $U_A$, so that
  $U_A \cong \ZZ^{\delta_A}\oplus \ZZ/m_A\ZZ$.
\end{defn}
The angle rank $\delta$ is by definition an integer between $0$ and
$g$. When $\delta = g$, there are no multiplicative relations among
the normalized eigenvalues. In other words, there are no additional
relations among the generators of $\Gamma_A$ apart from the ones
imposed by the Weil conjectures. If $A$ is absolutely simple, the
maximal angle rank condition also implies that the Tate conjecture
holds for all powers of $A$ (see Remark 1.3 in
\cite{dupuy2022angle}). On the other extreme, $\delta = 0$ if and only
if $A$ is supersingular (See Example 5.1 \cite{dupuy2022LMFDB}).

\begin{remark}
  \label{rmk:delta-invariant-base-extension}
  The angle rank is invariant under base extension:
  $\delta(A) = \delta(A_{(r)})$ for every $r \geq 1$. Indeed, any
  multiplicative relation between $\brk{u_1^r, \dots, u_g^r}$ is a
  multiplicative relation between $\brk{u_1, \dots, u_g}$. We have
  that
  $U_A/\mathrm{Tors}(U_A) \cong
  U_{A_{(r)}}/\mathrm{Tors}(U_{A_{(r)}})$ for every positive integer
  $r$. In particular, $U_A/\mathrm{Tors}(U_A) \cong U_{A_{(m)}}$ where
  $m = m_A$ is the angle torsion order of $A$.
\end{remark}

\begin{example}[Extension and restriction of scalars]
  Let $A/\FF_q$ be an abelian variety with Frobenius polynomial
  $P_A(T) = \prod (T-\alpha_i) \in \CC[T]$ and angle group
  $U_A = \langle u_1, \dots, u_g\rangle$. Then, the extension of
  scalars $A_{(r)}$ has Frobenius polynomial
  $P_{(r)}(T) = \prod(T-\alpha_i^r)$ and angle group
  $U_{A_{(r)}} = \langle u_1^r, \dots, u_g^r \rangle \subset U_A$. On
  the other hand, if $B/\FF_{q^r}$ is an abelian variety for some
  $r \geq 1$, and $A/\FF_q$ is the Weil restriction of $B$ to $\FF_q$,
  then $P_A(T) = P_B(T^r)$ and
  $U_A = \langle U_B, \zeta_r\rangle \supset U_B$. See
  \cite{claus_weil-restriction}.
\end{example}

\subsection{The Serre--Frobenius group}
\label{sec:q-Serre--Frobenius-group}
For every locally compact abelian group $G$, denote by $\widehat{G}$
its \cdef{Pontryagin dual}; this is the topological group of
continuous group homomorphisms $G \to \UU(1)$. It is well known that
$G \mapsto \widehat{G}$ gives an anti-equivalence of categories from
the category of locally compact abelian groups to itself. Moreover,
this equivalence preserves exact sequences, and every such $G$ is
canonically isomorphic to its double dual via the evaluation
isomorphism. See \cite{pontrjagin} for the original reference and
\cite{morris} for a gentle introduction.

Recall that we defined the Serre--Frobenius group of $A$ as the
topological group generated by the matrix
$\diag(u_1,\dots,u_g,\ol{u}_1,\dots,\ol{u}_g)$ inside of
$\USp_{2g}(\CC)$ (see \Cref{def:SF-group}).

\begin{remark}
  \label{rmk:relabelling}
  The group $\UU(1)^g$ embeds into $\USp_{2g}(\CC)$ as a maximal torus
  via $\mathbf{z} \mapsto \mathrm{diag}(\mathbf{z}, \ol{\mathbf{z}})$,
  so a different choice of indexing of the Frobenius eigenvalues
  yields a conjugate subgroup $g\inv\SF(A)g$, where $g$ is an element
  of the Weyl group $N_{\USp_{2g}(\CC)}(\UU(1)^g)/\UU(1)^g$. This Weyl
  group is isomorphic to the group $S_g^\pm$ of signed permutation
  matrices; the generic Galois group of a complex multiplication
  polynomial of degree $2g$.
\end{remark}

\begin{notation}
  In light of \Cref{rmk:relabelling}, we identify $\UU(1)^g$ with the
  group
  \begin{equation*}
    \begin{bmatrix}
      \diag(z_1,\dots,z_g)  &          \\
      & \diag(\ol{z}_1,\dots,\ol{z}_g)
    \end{bmatrix}
    \subset \USp_{2g}(\CC),
  \end{equation*}
  and the vector $\mathbf{u} \colonequals (u_1,\cdots,u_g)$ with the
  matrix $\diag(\mathbf{u},\overline{\mathbf{u}})$. The embedding of
  $\SF(A)$ into $\USp_{2g}(\CC)$ is completely determined by the
  topological generator of $(u_1, \dots, u_g)$, a vector of normalized
  Frobenius eigenvalues. We will represent the embedding of $\SF(A)$
  into $\USp_{2g}(\CC)$ (up to conjugation) by giving the topological
  generator of the Serre--Frobenius group.
\end{notation}

The following lemma will help us identify the isomorphism class of
$\SF(A)$ as a compact abelian group.

\begin{lemma}
  \label{lemma:equiv-char-of-SFgroup}
  The Serre--Frobenius group of an abelian variety $A$ has character
  group $U_A$. In particular, $\SF(A) \cong \widehat{U}_A$ canonically
  via the evaluation isomorphism.
\end{lemma}

\begin{proof}
  We have an injection $U_A \to \widehat{\SF(A)}$ given by mapping
  $\gamma$ to the character $\phi_\gamma$ that maps the topological
  generator $\mathbf{u}$ to $\gamma$. To see that this map is
  surjective, observe that by the exactness of Pontryagin duality, the
  inclusion $\SF(A) \hookrightarrow \UU(1)^g$ induces a surjection
  $\ZZ^g = \widehat{\UU(1)}^g \to \widehat{\SF(A)}$. Explicitly, this
  tells us that every character of $\SF(A)$ is given by
  $\phi(z_1, \dots, z_g) = z_1^{m_1}\cdots z_g^{m_g}$ for some
  $(m_1, \dots, m_g) \in \ZZ^g$. By continuity, every character $\phi$
  of $\SF(A)$ is completely determined by $\phi(\mathbf{u})$. In
  particular, we have that
  $\phi(\mathbf{u}) = u_1^{m_1} \cdots u_g^{m_g} \in U_A$.
\end{proof}

The following theorem should be compared to \cite[Theorem
3.12]{sutherland2013sato}
\begin{theorem}[\Cref{mainthm:structure-thm}]
  Let $A$ be an abelian variety defined over $\FF_q$. Then
  \[\SF(A) \cong \UU(1)^\delta \times C_m,\]
  where $\delta = \delta_A$ is the angle rank and $m = m_A$ is the
  angle torsion order. Furthermore, the connected component of the
  identity is
  \[\SF(A)^\circ = \SF(A_{(m)}).\]
\end{theorem}

\begin{proof}
  Since every finite subgroup of $\UU(1)$ is cyclic, the torsion part
  of the finitely generated group $U_A$ is generated by some primitive
  $m$-th root of unity $\zeta_m$. The group $U_{(m)}$ is torsion free
  by \Cref{rmk:delta-invariant-base-extension}. We thus have the split
  short exact sequence
  \begin{equation}
    \label{SES:grps}
    \begin{tikzcd}
      1 \arrow[r] & \langle \zeta_m \rangle \arrow[r] & U_A \arrow[r,
      "{u \mapsto u^m}"] & U_{(m)} \arrow[r] & 1.
    \end{tikzcd}
  \end{equation}
  After dualizing, we get:
  \begin{equation}
    \label{SES:st-grps}
    \begin{tikzcd}
      1 \arrow[r] & \SF(A_{(m)}) \arrow[r] & \SF(A) \arrow[r] &
      \langle \zeta_m \rangle \arrow[r] & 1.
    \end{tikzcd}
  \end{equation}
  We conclude that $\SF(A)^\circ = \SF(A_{(m)})$ and
  $\SF(A)/\SF(A)^\circ \cong \langle \zeta_m \rangle$.
\end{proof}

\begin{remark} \label{rmk:frob-torus} By definition, $U_A$ is the
  image of $\Gamma_A$ under the radial projection
  $\psi\colon \CC\unit \to \UU(1), z \mapsto z/|z|$. Thus, we have a
  short exact sequence
  \begin{equation}
    \label{SES:radial-gamma}
    \begin{tikzcd}
      1 \arrow[r] & \Gamma_A\cap \RR_{>0} \arrow[r] & \Gamma_A
      \arrow[r, "\psi|_\Gamma"] & U_A \arrow[r] & 1,
    \end{tikzcd}
  \end{equation}
  which is split by the section $u_j \mapsto \alpha_j$. The kernel
  $\Gamma\cap\RR_{>0}$ is free of rank $1$ and contains the group
  $q^{\ZZ}$. The relation between the Serre--Frobenius group $\SF(A)$
  and Serre's Frobenius Torus (see \cite[Volume IV,
  133.]{serre_oeuvres},
  \cite[Section 3]{Chi1992}) can be understood via their character
  groups.
  \begin{itemize}
  \item The (Pontryagin) character group of $\SF(A)$ is $U_A$.
  \item The (algebraic) character group of the Frobenius torus of $A$
    is the torsion free part of $\Gamma_A$.
  \end{itemize}

\end{remark}

\subsection{Equidistribution results}
\label{sec:equidist} Let $(Y,\mu)$ be a measure space in the sense of
Serre (see Appendix A.1 in \cite{serre1997abelian}). Recall that a
sequence $(y_r)_{r = 1}^\infty \subset Y$ is $\mu$-equidistributed if
for every continuous function $f\colon Y \to \CC$ we have that
\begin{equation}
  \int_Y f \mu = \lim_{n \to \infty} \frac1n\sum_{r=1}^n f(y_r) .
\end{equation}

In our setting, $Y$ will be a compact abelian Lie group with
probability Haar measure $\mu$. We have the following lemma.

\begin{lemma}
  \label{lemma:cyclic-equidist}
  Let $G$ be a compact group, and $h \in G$. Let $H$ be the closure of
  the group generated by $h$. Then, the sequence
  $(h^r)_{r = 1}^\infty$ is equidistributed in $H$ with respect to the
  Haar measure $\mu_H$.
\end{lemma}

\begin{proof}
  For a non-trivial character $\phi\colon H \to \CC\unit$, the image
  of the generator $\phi(h) = u \in \UU(1)$ is non-trivial. We see
  that
  \begin{equation*}
    \lim_{n\to\infty} \frac 1n \sum_{r=1}^n \phi(h^r) = \lim_{n\to\infty} \frac 1n \sum_{r=1}^n u^r = 0,
  \end{equation*}
  both when $u$ has finite or infinite order. The latter case follows
  from Weyl's equidistribution theorem in $\UU(1)$. The result follows
  from Lemma 1 in \cite[I-19]{serre1997abelian} and the Peter--Weyl
  theorem.
\end{proof}

\begin{coro}[\Cref{cor:equidist}]
  Let $A$ be a $g$-dimensional abelian variety defined over
  $\FF_q$. Then, the sequence $(x_r)_{r=1}^\infty$ of normalized
  traces of Frobenius is equidistributed in $[-2g,2g]$ with respect to
  the pushforward of the Haar measure on
  $\SF(A) \subseteq \USp_{2g}(\CC)$ via
  \begin{equation*}
    \SF(A) \subseteq \USp_{2g}(\CC) \to [-2g,2g], \quad M \mapsto \Tr(M).
  \end{equation*}
\end{coro}

\begin{proof}
  By \Cref{lemma:cyclic-equidist}, the sequence
  $(\mbf{u}^r)_{r=1}^\infty$ is equidistributed in $\SF(A)$ with
  respect to the Haar measure $\mu_{\SF(A)}$. By definition, the
  sequence $(x_r)_{r=1}^\infty$ is equidistributed with respect to the
  pushforward measure, and it is invariant under relabelling of the
  Frobenius eigenvalues.
\end{proof}

\begin{remark}[Maximal angle rank]
  When $A$ has maximal angle rank $\delta = g$, the Serre--Frobenius
  group is the full torus $\UU(1)^g$, and the sequence of normalized
  traces of Frobenius is equidistributed with respect to the
  pushforward of the measure $\mu_{\UU(1)^g}$; which we denote by
  $\lambda_g(x)$ following the notation\footnote{Beware of the
    different choice of normalization. We chose to use the interval
    $[-2g,2g]$ instead of $[-1,1]$ to be able to compare our
    distributions with the Sato--Tate distributions of abelian
    varieties defined over number fields.} in
  \cite{ahmadi2010shparlinski}.
\end{remark}

\section{Preliminary Results}
\label{sec:generalities}

For this entire section, we let $A$ be an abelian variety over
$\FF_q$, where $q=p^d$ for some prime $p$. Recall from
\Cref{sec:intro} that an abelian variety $A$ splits over a field
extension $\FF_{q^m}$ if $A_{(m)} \sim A_1 \times A_2$ and
$\dim A_1, \dim A_2 < \dim A$, i.e., if $A$ obtains at least one
isogeny factor after extending scalars to $\FF_{q^m}$. We say that $A$
\cdef{splits completely} over $\FF_{q^m}$ if
$A_{(m)} \sim A_1 \times A_2 \times \ldots \times A_k$, where each
$A_i$ is an absolutely simple abelian variety defined over
$\FF_{q^m}$. In other words, $A$ acquires its geometric isogeny
decomposition over $\FF_{q^m}$. We define the \cdef{splitting degree}
of $A$ to be the minimal positive integer $m$ such that $A$ splits
completely over $\FF_{q^m}$.

\subsection{Geometric products of elliptic curves}
\label{sec:prod-SF-groups}

We begin by stating an important lemma, attributed to Bjorn Poonen in
\cite{krajicek-scanlon-2000}.

\begin{lemma}[Poonen]
  \label{lemma:poonen} If $E_1, \dots, E_n$ are $n$ pairwise
  geometrically non-isogenous elliptic curves over $\FF_q$, then their
  eigenvalues of Frobenius $\alpha_1, \dots, \alpha_n$ are
  multiplicatively independent.
\end{lemma}

In fact, for abelian varieties that split completely as products of
elliptic curves, we can explicitly describe the Serre--Frobenius
group.

\begin{lemma}
  \label{lemma:power-ord-ec}
  Let $B/\FF_q$ be an abelian variety that splits over $\FF_{q^m}$ as
  a power of an ordinary elliptic curve, where $m\geq 1$ is the
  splitting degree of $B$. Then, $\SF(B) \cong \UU(1)\times
  C_m$. Furthermore, if $m > 1$, then
  \begin{equation*}
    \SF(B) = \brk{(u, \xi_1^{\nu} u, \xi_2^{\nu}u, \dots, \xi_{g-1}^{\nu}u) : u \in \UU(1), \nu \in \ZZ/m\ZZ} \subset \UU(1)^g,
  \end{equation*}
  with $\xi_1, \dots, \xi_{g-1}$ $m$-th roots of unity whose orders
  have least common multiple $m$.  In particular, when $B$ is simple
  then all the $\xi_j$ are distinct, primitive and
  $g-1 \leq \varphi(m)$.
\end{lemma}

\begin{proof}
  Angle rank is invariant under base change, so
  $\delta_B = \delta_{E^g} = 1$. It remains to show that the angle
  torsion order $m_B$ equals $m$. If $m =1$, then $B = E^g$ and there
  is nothing to show. Assume that $m>1$. Since $B_{(m)} \sim E^g$, we
  have that $P_{B,(m)}(T) = P_E(T)^g$. If we denote by
  $\gamma_1, \overline{\gamma}_1, \ldots \gamma_g, \overline{\gamma}_g
  $ and $\pi_1, \overline{\pi}_1$ the Frobenius eigenvalues of $B$ and
  $E$ respectively, we have that
  $\brk{\gamma_1^m, \overline{\gamma}_1^m, \dots, \gamma_g^m,
    \overline{\gamma}_g^m} = \brk{\pi_1, \overline{\pi}_1}$. Possibly
  after relabelling, we have that $\gamma_{j+1} = \xi_{j}\gamma_1$ for
  $j = 1, \dots, g-1$, where the $\xi_j$'s are $m$-th roots of unity
  and the minimality of $m$ ensures that the lcm of the orders of the
  $\xi_{j}$'s is $m$. This shows that $C_m \subset U_B$, so that
  $m \mid m_B$. On the other hand, we have that
  $\SF(B_{(m)}) = \SF(E^g) \cong \UU(1)$ is connected. This implies
  that $m_B \mid m$ and we conclude that
  $\SF(B) \cong\UU(1)\times C_m$. Assume now $B$ is simple, then
  $P_B(T)$ is irreducible and hence has no repeated roots, and thus
  $\xi_{j} \neq \xi_{i}$ for every $0 < j < i < g$, and every $\xi_i$
  is primitive. This shows that the set
  $\brk{\xi_{j}: j = 1, \dots, g-1}$ has $g-1$ elements, and therefore
  $g-1 \leq \varphi(m)$.
\end{proof}

\begin{lemma}
  \label{lemma:A1xB}
  Let $A = B\times A_1$ be an abelian variety over $\FF_q$ such that
  $A_1$ is supersingular with angle torsion order $m_{A_1} = m_1$ and
  $B$ splits over $\FF_{q^m}$ as the power of an ordinary elliptic
  curve, where $m\geq 1$ is the splitting degree of $B$. Then,
  $\SF(A)^\circ \cong \UU(1)$ and $m_A = \lcm(m_1, m)$. Furthermore,
  \begin{equation*}
    \SF(A) = \diag(\SF(B), \SF(A_1), \overline{\SF(B)}, \overline{\SF(A_1)}) \subset \USp_{2g}(\CC),
  \end{equation*}
  where $\overline{\SF(B)}$ denotes the (pointwise) complex conjugate
  of $\SF(B)$, and similarly for \(A_1\).
\end{lemma}
\begin{proof}
  From \Cref{lemma:power-ord-ec}, we see that
  $U_A = \langle \zeta_{m_1}, \zeta_m, v_1 \rangle$, where $v_1 =$
  $\gamma_1/\sqrt{q}$ is a normalized Frobenius eigenvalue of $ B$ and
  all the other roots $\gamma_j$ can be written as
  $\zeta_m^{\nu_j}\gamma_1$ with $\lcm_j(\ord(\zeta_m^{\nu_j}))=m.$ It
  follows that $U_A = C_{\lcm(m_1,m)} \oplus \langle v_1 \rangle$ so
  that $\delta_A =1$ and $m_A = \lcm(m_1,m)$. Furthermore, $\SF(A)$ is
  generated by
  $(v_1, \xi_1v_1, \dots, \xi_{\dim B-1}v_1, \eta_1, \dots,
  \eta_{g_1})$ for some $\xi_j\in \mu_m$ with $\lcm_j(\ord(\xi_j))=m$
  and $\eta_i \in \mu_{m_1}$.
\end{proof}

\begin{lemma}
  \label{lemma:B-geom-EC}
  Let $B$ be an ordinary abelian variety defined over $\FF_q$ such
  that $B$ is geometrically isogenous to a product of elliptic
  curves. Let $m$ be the splitting degree of $B$, and write
  $B_{(m)} \sim E_1^{g_1}\times \cdots \times E_n^{g_n}$ with $E_j$
  not geometrically isogenous to $E_i$ for $j \neq i$. Then
  $\SF(B) \cong \UU(1)^n\times C_m$. Moreover, we can describe the
  embedding of $\SF(B) \hookrightarrow \UU(1)^g$ as follows:
  \begin{enumerate}
  \item Let $r\geq 1$ be the smallest positive integer such that
    $B_{(r)} \sim B_1\times \cdots \times B_n$ decomposes into
    pairwise non-geometrically isogenous factors.
  \item Let $m_j$ be the splitting degree of $B_j$, so that
    $(B_j)_{(m_j)} \sim E_j^{g_j}$.
  \end{enumerate}
  Then, $m = r\lcm(m_1, \dots, m_n)$ and
  \begin{equation*}
    \SF(B_{(r)}) = \diag(\SF(B_1), \dots, \SF(B_n), \overline{\SF(B_1)}, \dots, \overline{\SF(B_n)}) \subset \USp_{2g}(\CC),
  \end{equation*}
  where each $\SF(B_j)$ is as in \Cref{lemma:power-ord-ec}, and
  $\overline{\SF(B_i)}$ denotes the (pointwise) complex conjugate of
  $\SF(B_i)$.
\end{lemma}
\begin{proof}
  This follows from combining \Cref{lemma:poonen} with the fact that
  the Serre--Frobenius group of $B$ is connected over an extension of
  degree $m$. The proof then proceeds as in \Cref{lemma:power-ord-ec}.
\end{proof}

\subsection{Splitting of simple ordinary abelian varieties of odd
  prime dimension}
\label{sec:simple-ord-odd-prime}

In this section, we analyze the splitting behavior of simple ordinary
abelian varieties of \emph{prime dimension} $g > 2$. Our first result
is analogous to \cite[Theorem 6]{howe2002existence} for odd primes.

\begin{theorem}[\Cref{mainthm:odd-prime}]
  \label{thm:simple-ordinary-prime-splitting}
  Let $A$ be a simple ordinary abelian variety defined over $\FF_q$ of
  \cdef{prime} dimension $g > 2$. Then, exactly one of the following
  conditions holds.
  \begin{enumerate}
  \item $A$ is absolutely simple.
  \item $A$ splits over a degree $g$ extension of $\FF_q$ as a power
    of an elliptic curve, and $\SF(A) \cong \UU(1)\times C_g$.
  \item $A$ splits over a degree $2g + 1$ extension of $\FF_q$ as a
    power of an elliptic curve, and
    $\SF(A) \cong \UU(1)\times C_{2g+1}$. This case only occurs if
    \(2g+1\) is also a prime, i.e., if \(g\) is a Sophie Germain
    prime.
  \end{enumerate}
  Furthermore, in (2) and (3), we have that
  \begin{equation*}
    \SF(A) = \brk{(u, \xi_1^{\nu} u, \xi_2^{\nu}u, \dots, \xi_{g-1}^{\nu}u) : u \in \UU(1), \nu \in \ZZ/m\ZZ},
  \end{equation*}
  with $\xi_1, \dots, \xi_{g-1}$ distinct primitive $m$-th roots of
  unity, for $m = g$ and $m = 2g+1$ respectively.
\end{theorem}

\begin{proof}
  Let $\alpha = \alpha_1$ be a Frobenius eigenvalue of $A$, and denote
  by $K = \QQ(\alpha)\cong \QQ[T]/P(T)$ the number field generated by
  $\alpha$. Since $A$ is ordinary, $\QQ(\alpha^n) \neq \QQ$ is a
  CM-field over $\QQ$ for every positive integer $n$, and $P(T)$ is
  irreducible and therefore $[\QQ(\alpha):\QQ] = 2g$. Suppose that $A$
  is not absolutely simple, and let $m$ be the smallest positive
  integer such that $A_{(m)}$ splits; by \cite[Lemma
  4]{howe2002existence} this is also the smallest $m$ such that
  $\QQ(\alpha^m) \subsetneq \QQ(\alpha)$. Since $\QQ(\alpha^m)$ is
  also a CM field, it is necessarily an imaginary quadratic number
  field.

  Observe first that $m$ must be odd. Indeed, if $m$ was even, then
  $\QQ(\alpha^{m/2}) = \QQ(\alpha)$ and
  $[\QQ(\alpha^{m/2}):\QQ(\alpha^m)] = 2$. This contradicts the fact
  that $[\QQ(\alpha):\QQ] = 2 g$, since $g$ is an odd prime. By
  \cite[Lemma 5]{howe2002existence}, there are two possibilities:
  \begin{enumerate}[label=(\roman*)]
  \item \label{poss:2} $P(T) \in \QQ[T^m]$,
  \item \label{poss:1} $K = \QQ(\alpha^m, \zeta_m)$.
  \end{enumerate}
  
  If \ref{poss:2} holds and $P(T) = T^{2m} + bT^m + q^{g}$, we
  conclude that $m = g$ and $b=a_g$. In this case, the minimal
  polynomial of $\alpha^{g}$ has degree 2 and is of the form
  $h_{(g)}(T) = (T-\alpha^g)(T-\overline{\alpha}^g)$. Note that
  $\alpha^{g}$ and $\overline{\alpha}^{g}$ are distinct, since $A$ is
  ordinary. Thus, $P_{(g)}(T) = h_{(g)}(T)^{g}$ and $A$ must split
  over a degree $g$ extension as the power of an ordinary elliptic
  curve.

  If \ref{poss:1} holds, we have that $\varphi(m) \mid 2g$. Since
  $m>1$ is odd and $\varphi(m)$ takes even values, we have two
  possible options: either $\varphi(m) = 2$ or $\varphi(m) = 2g$. If
  $\varphi(m) = 2$, then $[K:\QQ(\alpha^m)] \leq 2$ which contradicts
  the fact that $K = \QQ(\alpha)$ is a degree $2g$ extension of
  $\QQ$. Therefore, necessarily, $\varphi(m) = 2g$, and
  $\QQ(\alpha) = \QQ(\zeta_m).$ Recall from elementary number theory
  that the solutions to this equation are $(m, g) = (9,3)$ or
  $(m, g) = (2g + 1,g)$ for $g$ a Sophie Germain prime.
  \begin{itemize}
  \item ($g > 3$) In this case, \ref{poss:1} only occurs when $2g+1$
    is prime.
  \item ($g = 3$) In this case, either $m = 7$ or $m = 9$. To conclude
    the proof, we show that $m = 9$ does not occur. More precisely, we
    will show that if $A$ splits over a degree $9$ extension, it
    splits over a degree $3$ extension as well. In fact, suppose that
    $K = \QQ(\zeta) = \QQ(\alpha)$ for $\zeta$ a primitive $9$th root
    of unity. The subfield $F = \QQ(\zeta^3)$ is the only imaginary
    quadratic subfield of $K$, so if a power of $\alpha$ does not
    generate $K$, it must lie in $F$. Suppose $\alpha^9$ lies in
    $F$. Let $\sigma$ be the generator of $\Gal(K/F)$ sending $\zeta$
    to $\zeta^4$. The minimal polynomial of $\alpha$ over $F$ divides
    $T^9 - \alpha^9$, so $\sigma(\alpha) = \alpha\cdot\zeta^j$ for
    some $j$, and $\sigma^2(\alpha) = \alpha\zeta^{5j}$. Since the
    product of the three conjugates of $\alpha$ over $F$ must lie in
    $F$, we have that
    $\alpha^3\cdot\zeta^{6j} =
    (\alpha)(\alpha\cdot\zeta^j)(\alpha\cdot\zeta^{5j}) \in F$, which
    implies that $\alpha^3 \in F$ and we conclude that $A$ splits over
    a degree-$3$ extension of the base field.
  \end{itemize}
  The statement about the structure of the Serre--Frobenius group
  follows from \Cref{lemma:power-ord-ec}.
\end{proof}
We thank Everett Howe for explaining to us why the case $m=9$ above
does not occur.

\subsection{Zarhin's notion of neatness}
\label{sec:zarhin-neatness}
In this section we discuss Zarhin's notion of \emph{neatness}, a
useful technical definition closely related to the angle rank.  Define
\begin{equation}
  \label{eq:R'}
  R_A' \colonequals \brk{u_j^2 : \alpha_j \in R_A}.
\end{equation}
Note that according to our numbering convention, we have that
$u_j\inv = \overline{u}_j = u_{j+g}$ for every
$j \in \brk{1,\dots, g}$.

\begin{defn}[Zarhin]
  \label{def:neat}
  Let $A$ be an abelian variety defined over $\FF_q$. We say that $A$
  is \cdef{neat} if it satisfies the following conditions:
  \begin{enumerate}[label=(N\alph*)]
  \item \label{neat-a} $\Gamma_A$ is torsion free.
  \item \label{neat-b} For every function $e\colon R_A'\to \ZZ$
    satisfying
    \begin{equation*}
      \prod_{\beta \in R_A'}\beta^{e(\beta)}= 1,
    \end{equation*}
    then $e(\beta) = e(\beta\inv)$ for every $\beta \in R_A'$.
  \end{enumerate}
\end{defn}

\begin{remarks} \hfill
  \begin{enumerate}[leftmargin=*, label=(\thecounter.\alph*)]
  \item \label{rmk:ss->neat} If $A$ is supersingular and $\Gamma_A$ is
    torsion free, then $A$ is neat. Indeed, in this case we have that
    $R_A' = \brk{1}$ and condition \ref{neat-b} is trivially
    satisfied.
  \item \label{rmk:neat2} Suppose that the Frobenius eigenvalues of
    $A$ are distinct and not supersingular. Some base extension of $A$
    is neat if and only if $A$ has maximal angle rank.
  \item \label{rmk:neat3} In general, maximal angle rank always
    implies neatness.
  \end{enumerate}
\end{remarks}

\subsection{Supersingular Serre--Frobenius groups}
\label{sec:ss-qST-groups}

Recall that a $q$-Weil number $\alpha$ is called \cdef{supersingular}
if $\alpha/\sqrt{q}$ is a root of unity. In \cite[Proposition
3.1]{zhu2001supersingular}, Zhu classified the minimal polynomials
$h(T)$ of supersingular $q$-Weil numbers. Let $\Phi_r(T)$ denote the
$r$-th cyclotomic polynomial, $\varphi(r) \colonequals \deg \Phi_r(T)$
the Euler totient function, and $\paren{\tfrac{a}{b}}$ the Jacobi
symbol. Then the possibilities for the minimal polynomials of
supersingular $q$-Weil numbers are given in \Cref{table:zhu}.

\begin{table}[ht]
  \setlength{\arrayrulewidth}{0.3mm} \setlength{\tabcolsep}{5pt}
  \renewcommand{\arraystretch}{2}
  \caption{Minimal polynomial of a supersingular $q$-Weil number
    $\alpha$.}
  \label{table:zhu}
  \begin{longtable}{|c|c|c|c|c|}
    \hline
    \rowcolor{header_color} 
    Type & $d$  &                                  & $h(T)$                                 & Roots      \\ \hline
    Z-1 & Even & -                                & $\Phi_{m}^{[\sqrt{q}]}(T) \colonequals \sqrt{q}^{\varphi(m)}\Phi_m(T/\sqrt{q})$ & $\zeta_m^j\sqrt{q} \mfor j \in (\ZZ/m\ZZ)\unit$    \\ \hline
    Z-2 & Odd  & $\QQ(\alpha) \neq \QQ(\alpha^2)$ & $\Phi_n^{[q]}(T^2) \colonequals q^{\varphi(n)} \Phi_n(T^2/q)$            & $\pm \zeta_{2n}^j\sqrt{q} \mfor j \in (\ZZ/n\ZZ)\unit$   \\ \hline
    Z-3 & Odd  & $\QQ(\alpha) = \QQ(\alpha^2)$    & $\displaystyle \prod_{\substack{1 \leq j \leq n \\ \gcd(j,n)=1}}\paren{T - \paren{\dfrac{q}{j}}\zeta_m^{\nu j}\sqrt{q}}$  & $\paren{\dfrac{q}{j}}{\zeta_m^{j}} \sqrt{q} \mfor  j \in (\ZZ/n\ZZ)\unit $ \\ \hline
  \end{longtable}
  \addtocounter{table}{-1}
\end{table}

\begin{notation}[\Cref{table:zhu}]
  In case (Z-1), $m$ is any positive integer. In cases (Z-2) and
  (Z-3), $m$ additionally satisfies $m \not\equiv 2 \md 4$, and
  $n \colonequals m/\gcd(2,m)$.  The symbol $\zeta_m$ denotes the a
  primitive $m$-th root of unity.  Note that in this case,
  $\varphi(n) = \varphi(m)/\gcd(2,m)$. Following the notation in
  \cite{Singh&McGuire&Zaytsev2014}, given a polynomial $f(T)\in K[T]$
  for some field $K$, and a constant $a \in K\unit$, let
  \begin{equation*} \label{eq:expand-roots} f^{[a]}(T) \colonequals
    a^{\deg f}f(T/a).
  \end{equation*}
\end{notation}

Given any supersingular abelian variety $A$ defined over $\FF_q$, the
Frobenius polynomial $P_A(T)$ is a power of the minimal polynomial
$h_A(T)$, and this minimal polynomial is of type (Z-1), (Z-2), or
(Z-3) as above. We say that $A$ is of \cdef{type Z-i} if the minimal
polynomial $h_A(T)$ is of type (Z-i) for $i = 1,2,3$.

Since $U_A$ is finite in the supersingular case, we have that
$\SF(A) \cong U_A \cong C_{m_A}$. Furthermore, we have that
\begin{equation*}
  \SF(A) = \brk{(\xi_1^\nu, \xi_2^\nu,\dots, \xi_g^\nu) : \nu \in \ZZ/m_A\ZZ} \subset \UU(1)^g,
\end{equation*}
with $\xi_i$'s being $m_A$-th roots of unity, whose orders have least
common multiple $m_A$.  In particular, we can read off the character
group $U_A$ from the fourth column in \Cref{table:zhu}. For instance,
if $m=3$ and $d$ is even, then we have a polynomial of type Z-1, and
the Serre--Frobenius group is isomorphic to $C_3$. On the other hand,
if $m=3$ and we have a polynomial of type Z-2, then the
Serre--Frobenius group is isomorphic to $C_6$. Given a $q$-Weil
polynomial $f(T) \in \QQ[T]$ with roots
$\alpha_1, \cdots, \alpha_{2n}$, the associated \cdef{normalized
  polynomial} $\tilde{f}(T) \in \RR[T]$ is the monic polynomial with
roots $u_1 = \alpha_1/\sqrt{q}, \dots, u_{2n} =
\alpha_{2n}/\sqrt{q}$. \Cref{table:zhu} allows us to go back and forth
between $q$-Weil polynomials $f(T)$ and the normalized polynomials
$\tilde{f}(T)$.

\begin{itemize}[leftmargin=*]
\item If $h(T)$ is the minimal polynomial of a supersingular $q$-Weil
  number of type Z-1, the normalized polynomial $\tilde{h}(T)$ is the
  cyclotomic polynomial $\Phi_m(T)$. Conversely, we have that
  $h(T) = \tilde{h}^{[\sqrt{q}]}(T)$.  \item If $h(T)$ is the minimal
  polynomial of a supersingular $q$-Weil number of type Z-2, the
  normalized polynomial $\tilde{h}(T)$ is the polynomial
  $\Phi_n(T^2)$. Conversely, $h(T) = \tilde{h}^{[q]}(T)$.
\end{itemize}

\section{Elliptic Curves}
\label{sec:elliptic-curves}

The goal of this section is to prove
\Cref{mainthm:elliptic-curves}. Furthermore, we give a thorough
description of the set of possible orders $m$ for the supersingular
Serre--Frobenius groups $\SF(E) = C_m$ in terms of $p$ and $q = p^d$.

The isogeny classes of elliptic curves over $\FF_q$ were classified by
Deuring \cite{deuring1941typen} and Waterhouse \cite[Theorem
4.1]{waterhouse1969abelian}. Writing the characteristic polynomial of
Frobenius as $P(T) = T^2 + a_1T + q$, the Weil bounds give
$|a_1| \leq 2\sqrt{q}$. Conversely, the integers $a$ satisfying
$|a| \leq 2\sqrt{q}$ that correspond to the isogeny class of an
elliptic curve are the following.

\begin{theorem}[{\cite[Theorem 2.6.1]{serre2020rational}}]
  \label{thm:waterhouse}
  Let $p$ be a prime and $q = p^d$. Let $a \in \ZZ$ satisfy
  $|a|\leq 2\sqrt{q}$.
  \begin{enumerate}[labelindent=0pt]
  \item \label{case:ordinary} If $p \nmid a$, then $a$ is the trace of
    Frobenius of an elliptic curve over $\FF_q$. This is the ordinary
    case.
  \item If $p \mid a$, then $a$ is the trace of Frobenius of an
    elliptic curve over $\FF_q$ if and only if one of the following
    holds:
    \begin{enumerate}[label=(\roman*)]
    \item \label{case:SS2p} $d$ is even and $a = \pm 2\sqrt{q}$,
    \item \label{case:SSp} $d$ is even and $a = \sqrt{q}$ with
      $p \not\equiv 1 \md 3$,
    \item \label{case:SS-p} $d$ is even and $a = -\sqrt{q}$ with
      $p \not\equiv 1 \md 3$,
    \item \label{case:SS0-even} $d$ is even and $a = 0$ with
      $p \not\equiv 1 \md 4$,
    \item \label{case:SS0-odd}$d$ is odd and $a=0$,
    \item \label{case:SSp=2} $d$ is odd, $a = \pm \sqrt{2q}$ with
      $p = 2$.
    \item \label{case:SSp=3} $d$ is odd, $a = \pm \sqrt{3q}$ with
      $p = 3$.
    \end{enumerate}
    This is the supersingular case.
  \end{enumerate}
\end{theorem}

In the ordinary case, the normalized Frobenius eigenvalue $u_1$ is not
a root of unity, and thus $\SF(E) = \UU(1)$. In the supersingular
case, the normalized Frobenius eigenvalue $u_1$ is a root of unity,
and thus $\SF(E) = C_m$ is cyclic, with $m$ equal to the order of
$u_1$. For each value of $q$ and $a$ in \Cref{thm:waterhouse} part
(2), we get a right triangle of hypotenuse of length $\sqrt{q}$ and
base $a/2$, from which we can deduce the angle $\vartheta_1$ and thus
the order $m$ of the corresponding root of unity $u_1$. We thus obtain
\Cref{mainthm:elliptic-curves} as a restatement of
\Cref{thm:waterhouse}.

\medskip There are eight Serre--Frobenius groups for elliptic curves, sumarized
in \Cref{table:elliptic-curves}, and they correspond to eight possible
Frobenius distributions of elliptic curves over finite fields. For ordinary
elliptic curves (as explained in Section \ref{sec:intro}), the sequence of
normalized traces $(x_r)_{r=1}^\infty$ is equidistributed in the interval
$[-2,2]$ with respect to the measure $\lambda_1(x)$ (\Cref{eq:lambda1})
obtained as the pushforward of the Haar measure $\mu_{\UU(1)}$ under
$z \mapsto z + \overline{z}$. See \Cref{fig:ordinary-ec}.

\medskip The remaining seven Serre--Frobenius groups are finite and
cyclic; they correspond to supersingular elliptic curves. For a given
$C_m = \langle\zeta_m\rangle \subset \UU(1)$, denote by $\delta_m$ the
measure obtained by pushforward along $z \mapsto z + \overline{z}$ of
the normalized counting measure,
\begin{equation*}
  \mu_{C_m}(f)  \colonequals \int f\, \mu_{C_m} \colonequals \frac{1}{m}\sum_{j=1}^m f(\zeta_m^j).  
\end{equation*}

\begin{table}[ht]
  \setlength{\arrayrulewidth}{0.3mm} \setlength{\tabcolsep}{5pt}
  \renewcommand{\arraystretch}{1.3}
  \caption{Serre--Frobenius groups of elliptic curves.}
  \begin{longtable}{|c|c|c|c|c|c|c|c|}
    \hline
    \rowcolor{header_color}
    \Cref{thm:waterhouse} & $p$ & $d$ & $a$ & $\SF(E)$ & Generator  & Example & \Cref{fig:elliptic-distributions} \\ \hline
    (\ref{case:ordinary}) & - & - & $p \nmid a$ & $\UU(1)$ & $u_1$ & \href{https://www.lmfdb.org/Variety/Abelian/Fq/1/2/ab}{\texttt{1.2.ab}} & \ref{fig:ordinary-ec} \\ \hline
    2-\ref{case:SS2p} & - & Even & $2\sqrt{q}$   & $C_1$ & $1$ &  \href{https://www.lmfdb.org/Variety/Abelian/Fq/1/4/ae}{\texttt{1.4.ae}} & \ref{fig:C1} \\ \hline
    2-\ref{case:SS2p} & - & Even & $- 2\sqrt{q}$   & $C_2$ & -$1$ & \href{https://www.lmfdb.org/Variety/Abelian/Fq/1/4/e}{\texttt{1.4.e}} & \ref{fig:C2} \\ \hline
    2-\ref{case:SS-p} & $p\not\equiv 1 \md 3$  & Even & $-\sqrt{q}$ & $C_3$ & $\zeta_3$ & \href{https://www.lmfdb.org/Variety/Abelian/Fq/1/4/c}{\texttt{1.4.c}} & \ref{fig:C3} \\ \hline
    2-\ref{case:SS0-even} & $p\not\equiv 1 \md 4$ & Even & $0$ & $C_4$ & $\zeta_4$  & \href{https://www.lmfdb.org/Variety/Abelian/Fq/1/4/a}{\texttt{1.4.a}} & \ref{fig:C4} \\ \hline
    2-\ref{case:SS0-odd} & - & Odd & $0$ & $C_4$ & $\zeta_4$ &
                                                               \href{https://www.lmfdb.org/Variety/Abelian/Fq/1/2/a}{\texttt{1.2.a}}& \ref{fig:C4} \\ \hline  
    2-\ref{case:SSp} & $p\not\equiv 1 \md 3$  & Even & $\sqrt{q}$ & $C_6$ & $\zeta_6$ & \href{https://www.lmfdb.org/Variety/Abelian/Fq/1/4/ac}{\texttt{1.4.ac}} & \ref{fig:C6} \\ \hline
    2-\ref{case:SSp=2} & 2 & Odd  & $\pm\sqrt{2q}$ & $C_8$ & $\zeta_8$ & \href{https://www.lmfdb.org/Variety/Abelian/Fq/1/2/ac}{\texttt{1.2.ac}} & \ref{fig:C8} \\ \hline
    2-\ref{case:SSp=3} & 3 & Odd & $\pm\sqrt{3q}$ & $C_{12}$ & $\zeta_{12}$ & \href{https://www.lmfdb.org/Variety/Abelian/Fq/1/3/ad}{\texttt{1.3.ad}} & \ref{fig:C12} \\ \hline
  \end{longtable}
  \addtocounter{table}{-1}
  \label{table:elliptic-curves}
\end{table}

\vspace{2cm}
\begin{figure}
  \centering
  \begin{subfigure}{0.24\textwidth}
    \centering
    \includegraphics[width=1.07\textwidth]{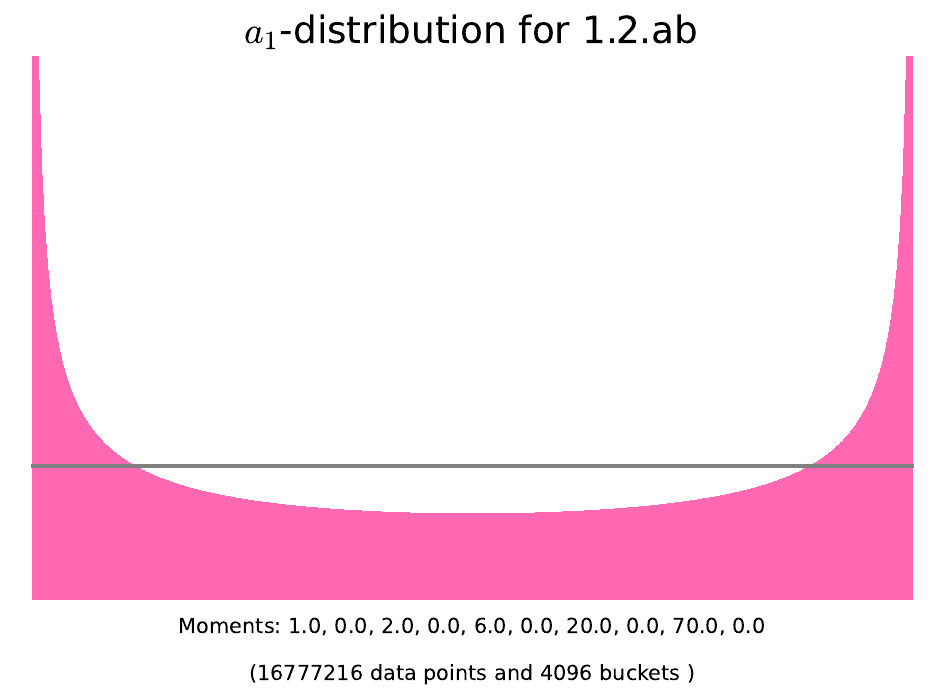}
    \caption{$\SF(E) = \UU(1)$}
    \label{fig:ordinary-ec}
  \end{subfigure}
  \hfill
  \begin{subfigure}{0.24\textwidth}
    \centering \includegraphics[width=0.95\textwidth]{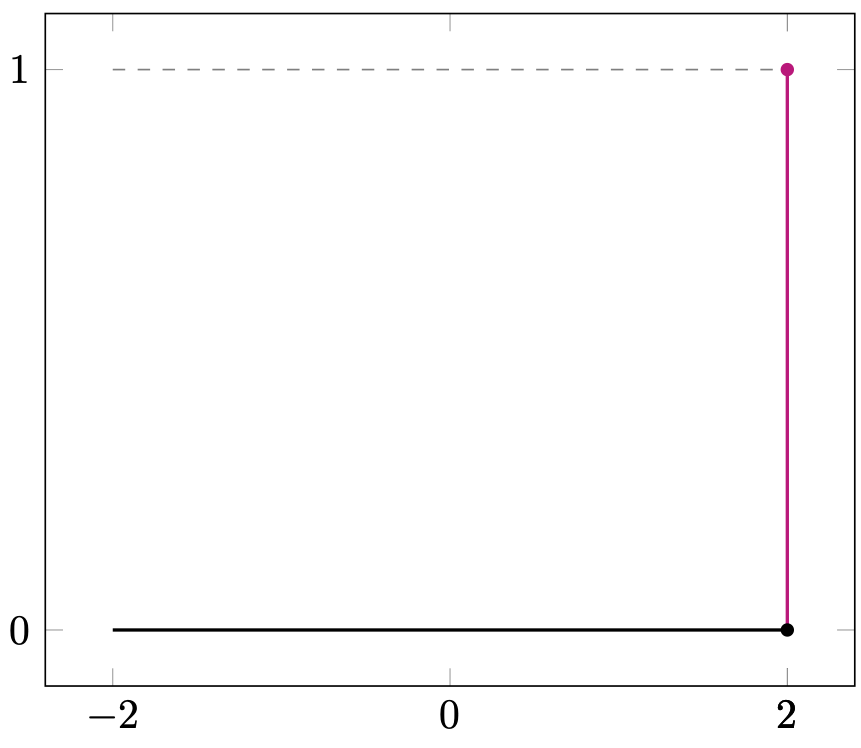}
    \caption{$\SF(E) = C_1$}
    \label{fig:C1}
  \end{subfigure}
  \hfill \hfill
  \begin{subfigure}{0.24\textwidth}
    \centering \includegraphics[width=\textwidth]{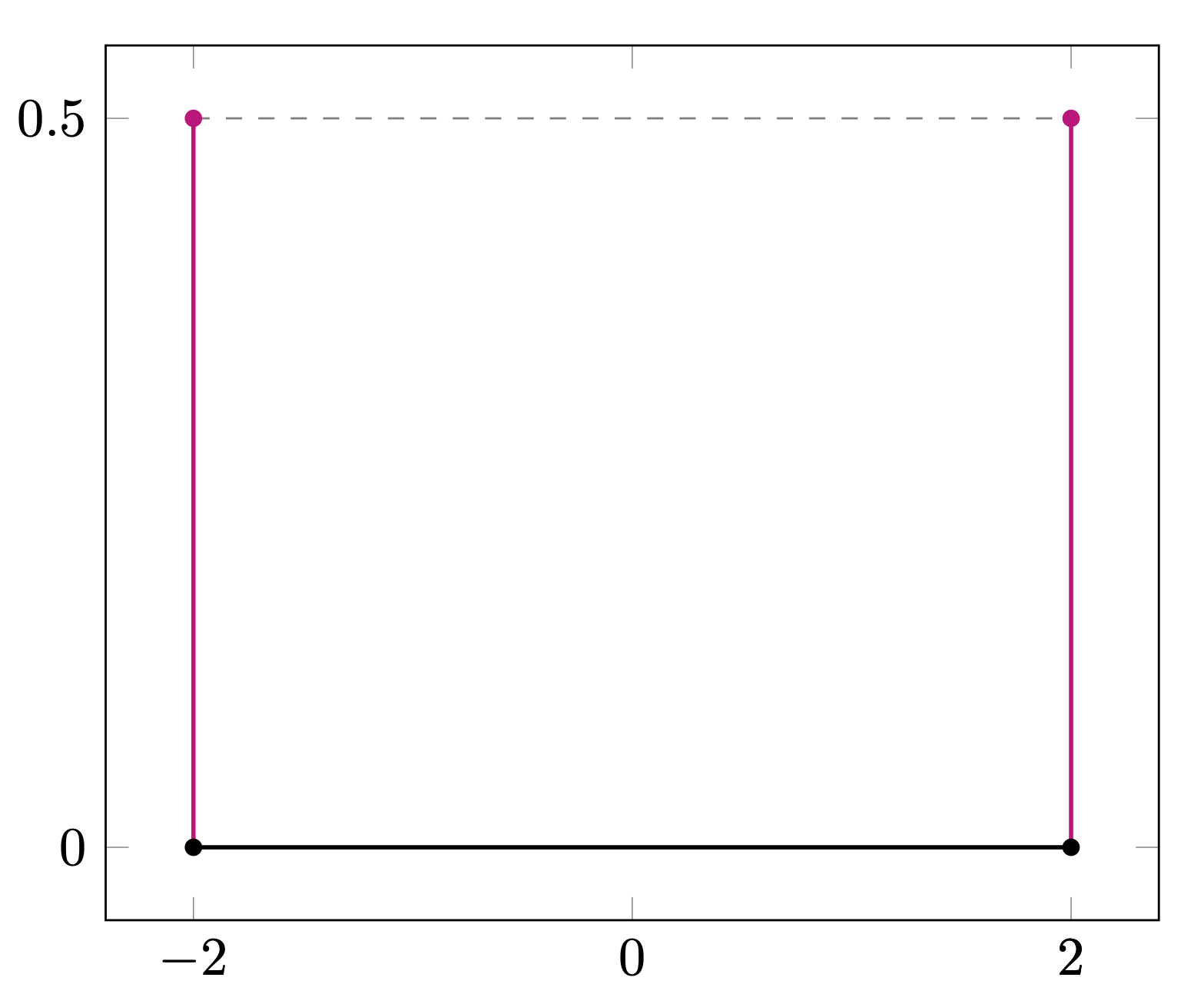}
    \caption{$\SF(E) = C_2$}
    \label{fig:C2}
  \end{subfigure}
  \begin{subfigure}{0.24\textwidth}
    \centering \includegraphics[width=\textwidth]{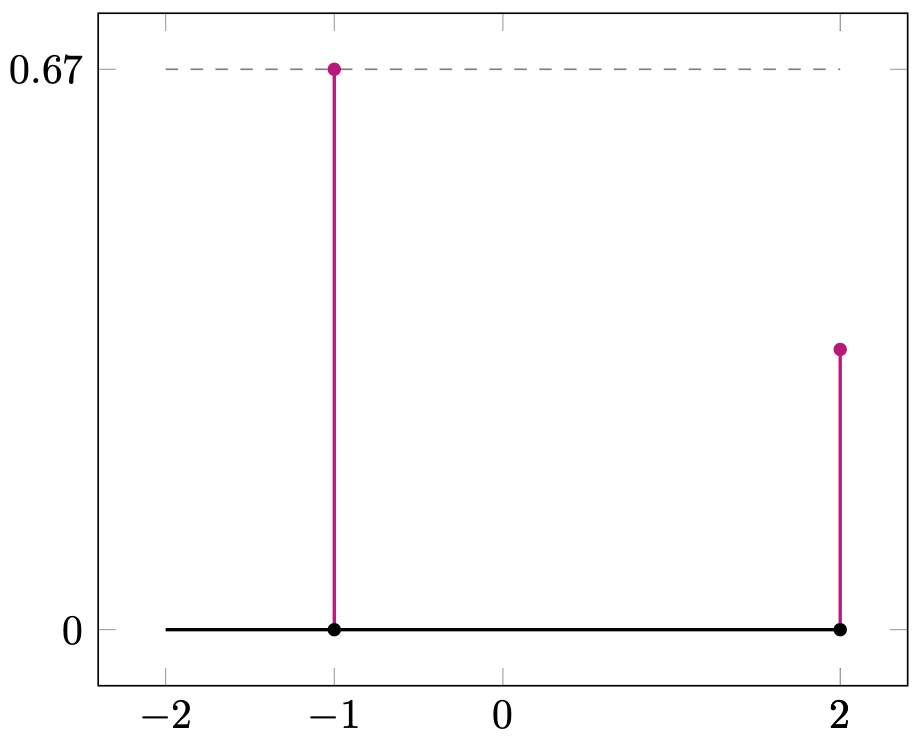}
    \caption{$\SF(E) = C_3$}
    \label{fig:C3}
  \end{subfigure}
  \hfill
  \begin{subfigure}{0.24\textwidth}
    \centering \includegraphics[width=\textwidth]{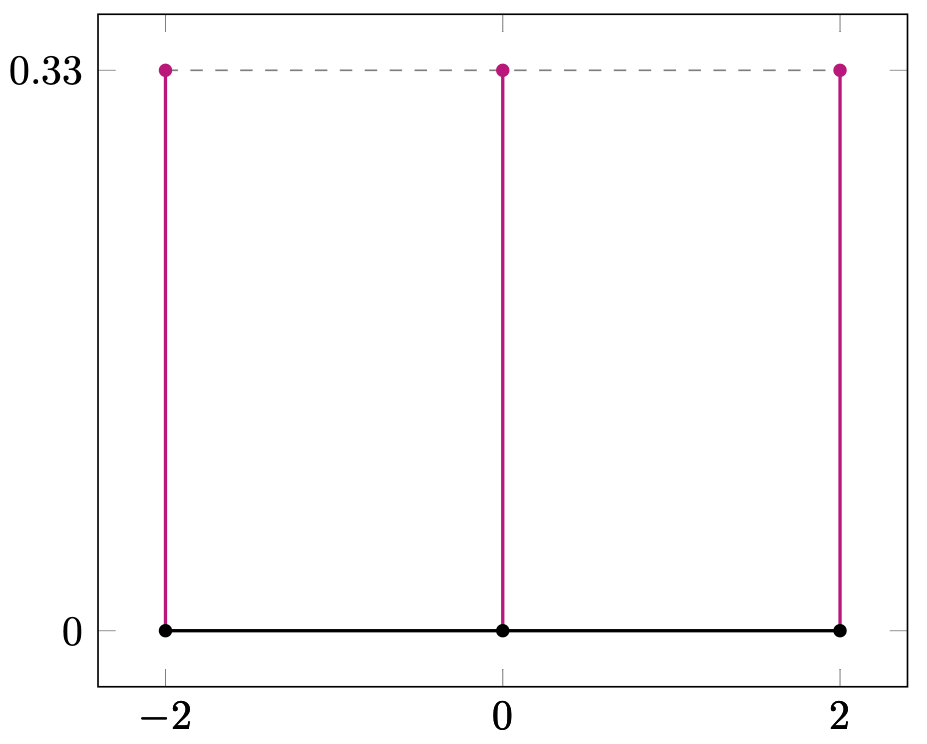}
    \caption{$\SF(E) = C_4$}
    \label{fig:C4}
  \end{subfigure}
  \hfill
  \begin{subfigure}{0.24\textwidth}
    \centering \includegraphics[width=\textwidth]{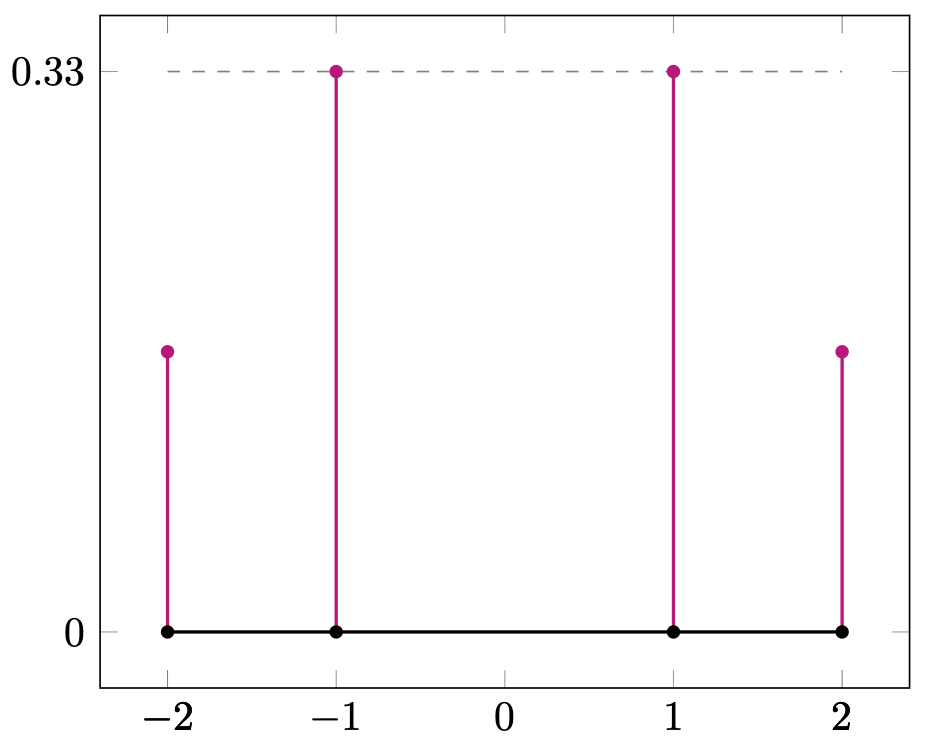}
    \caption{$\SF(E) = C_6$}
    \label{fig:C6}
  \end{subfigure}
  \hfill
  \begin{subfigure}{0.24\textwidth}
    \centering \includegraphics[width=\textwidth]{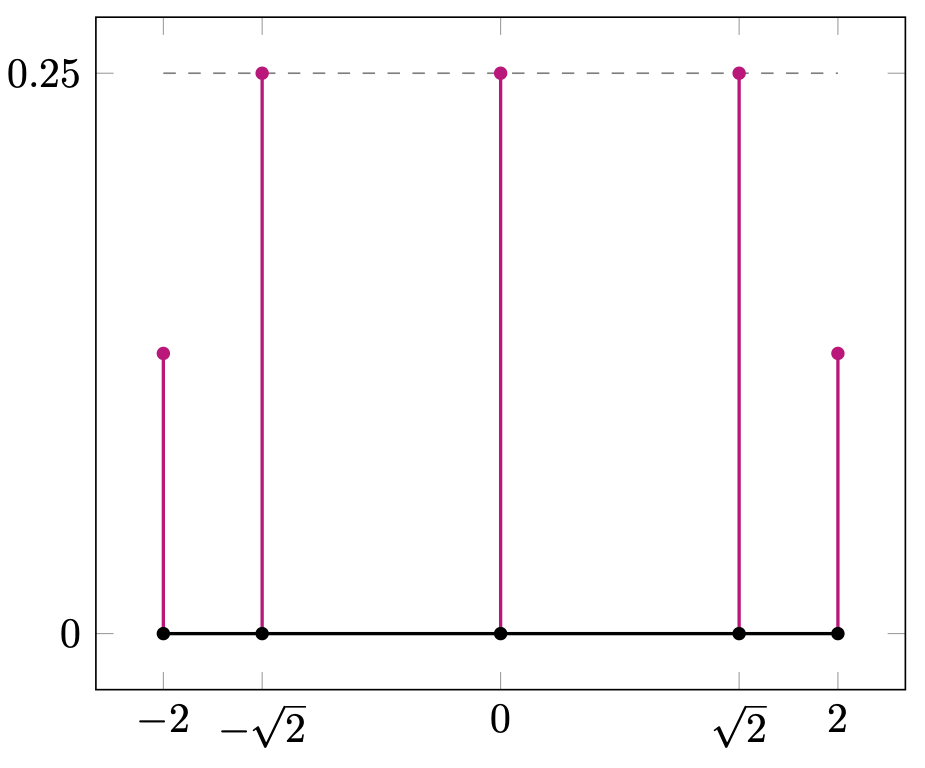}
    \caption{$\SF(E) = C_8$}
    \label{fig:C8}
  \end{subfigure}
  \hfill
  \begin{subfigure}{0.24\textwidth}
    \centering \includegraphics[width=\textwidth]{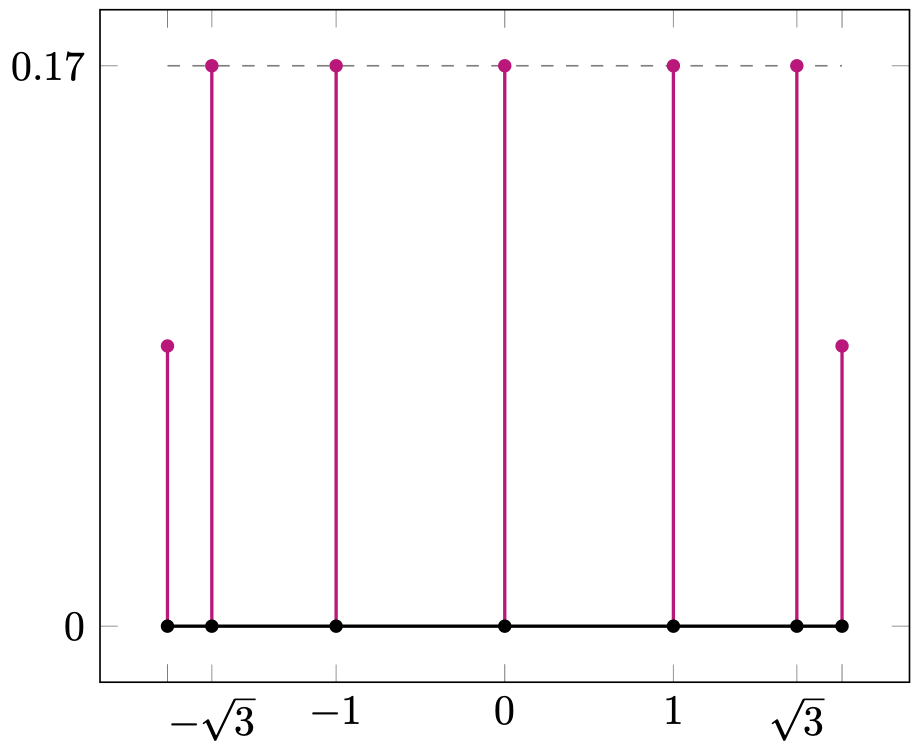}
    \caption{$\SF(E) = C_{12}$}
    \label{fig:C12}
  \end{subfigure}
  \caption{$a_1$-histograms of elliptic curves.}
  \label{fig:elliptic-distributions}
\end{figure}

\section{Abelian Surfaces}
\label{sec:surfaces}
The goal of this section is to classify the possible Serre--Frobenius
groups of abelian surfaces (\Cref{mainthm:surfaces}). The proof is a
careful case-by-case analysis, described by Flowchart
\ref{fig:proof-2}.

\begin{figure}[h]
  \centering \includegraphics[width=0.8\textwidth]{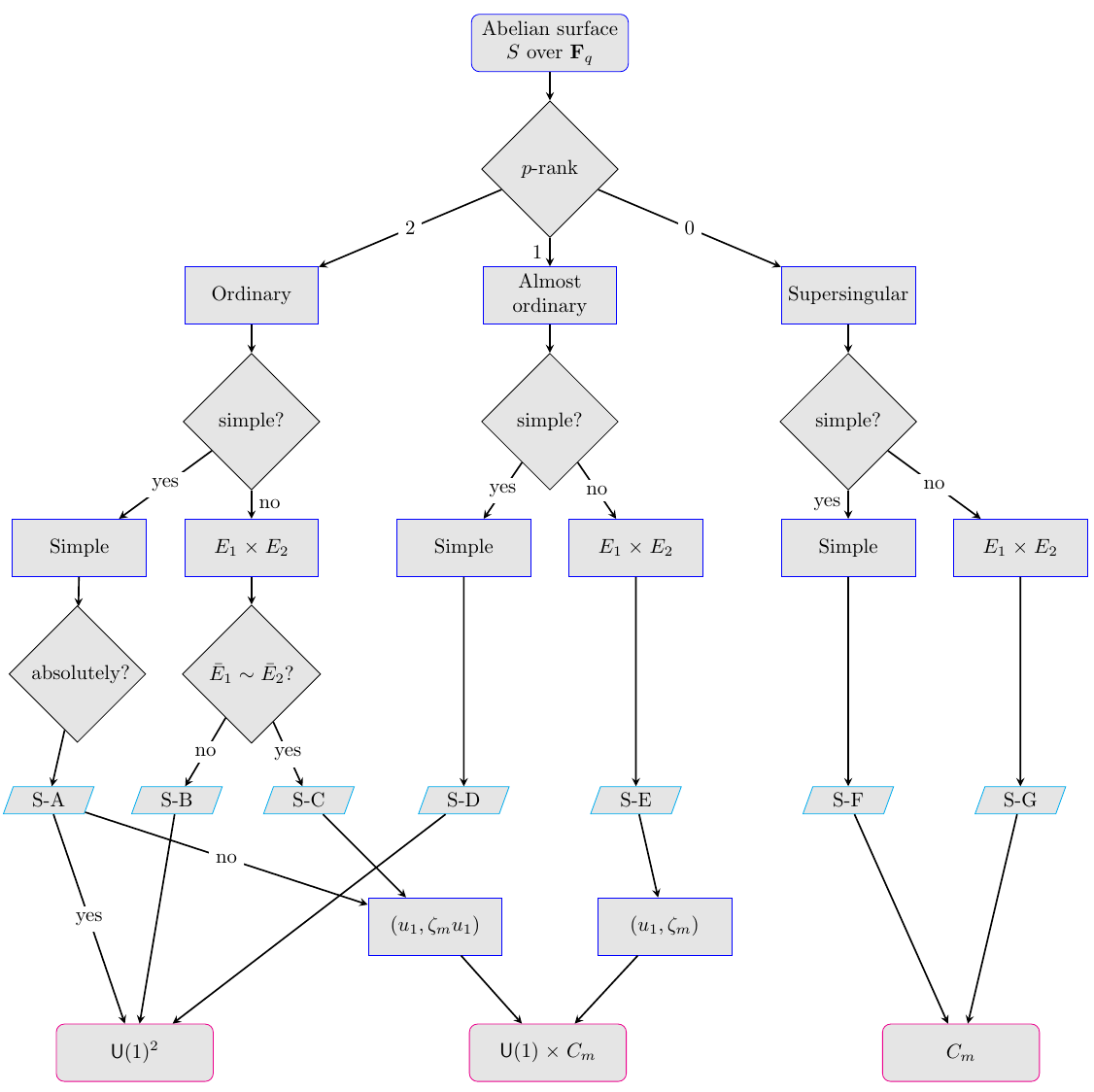}
  \caption{\Cref{mainthm:surfaces}: Classification in dimension 2.}
  \label{fig:proof-2}
\end{figure}

We separate our cases first according to $p$-rank, and then according
to simplicity. In the supersingular and almost ordinary cases this
stratification is enough. In the ordinary case, we have to further
consider the geometric isogeny type of the surface. When the angle
rank is $2$, the Serre--Frobenius group is the full torus
$\UU(1)^2$. When the angle rank is $1$, the Serre--Frobenius group is
isomorphic to $\UU(1)\times C_m$ but there are two non-conjugate ways
to embed $\SF(S)$ into $\UU(1)^2$; these are determined by the
topological generator of the group, which can be either
$(u_1,\zeta_m u_1)$ or $(u_1, \zeta_m)$.

\subsection{Simple ordinary surfaces}
\label{sec:simple-ordinary-surfaces}

We restate a theorem of Howe and Zhu in our
notation.  \begin{theorem}[{\cite[Theorem
    6]{howe2002existence}}] \label{thm:HZ6} Suppose that
  $P(T) = T^4 + a_1T^3 + a_2T^2 + qa_1T + q^2$ is the Frobenius
  polynomial of a simple ordinary abelian surface $S$ defined over
  $\FF_q$. Then, exactly one of the following conditions holds.
  \begin{enumerate}[label=(\alph*)] \item
    \label{case:ord-abs-simple-surface} $S$ is absolutely
    simple.  \item \label{case:ord-simple-2-split} $a_1 = 0$ and $S$
    splits over a quadratic extension.  \item
    \label{case:ord-simple-3-split} $a_1^2 = q + a_2$ and $S$ splits
    over a cubic extension.  \item \label{case:ord-simple-4-split}
    $a_1^2 = 2a_2$ and $S$ splits over a quartic extension.  \item
    \label{case:ord-simple-6-split} $a_1^2 = 3a_2 -3q$ and $S$ splits
    over a sextic extension.  \end{enumerate} \end{theorem}

\begin{lemma}[Node S-A in \Cref{fig:proof-2}]
  \label{lemma:SA}
  Let $S$ be a simple ordinary abelian surface over $\FF_q$. Then,
  exactly one of the following conditions holds.
  \begin{enumerate}[label=(\alph*)]
  \item \label{entry:ord-abs-simple-surface} $S$ is absolutely simple
    and $\SF(S) = \UU(1)^2$.
  \item \label{entry:ord-simple-2-split} $S$ splits over a quadratic
    extension and $\SF(S) \cong \UU(1)\times C_2$.
  \item \label{entry:ord-simple-3-split} $S$ splits over a cubic
    extension and $\SF(S) \cong \UU(1)\times C_3$.
  \item \label{entry:ord-simple-4-split} $S$ splits over a quartic
    extension and $\SF(S) \cong \UU(1)\times C_4$.
  \item \label{entry:ord-simple-6-split}$S$ splits over a sextic
    extension and $\SF(S) \cong \UU(1)\times C_6$.
  \end{enumerate}
  In cases
  \ref{entry:ord-simple-2-split}-\ref{entry:ord-simple-6-split}, we
  have that
  $\SF(A) = \brk{(u, \zeta_m^\nu u) : u \in \UU(1), \, \nu \in
    (\ZZ/m\ZZ)}$, for some primitive $m$-th root of unity $\zeta_m$.
\end{lemma}
\begin{proof}
  \hfill
  \begin{enumerate}[wide]
  \item[(a)] From \cite[Theorem 1.1]{Zarhin2015EigenFrob}, we conclude
    that some finite base extension of an absolutely simple abelian
    surface is neat and therefore has maximal angle rank by Remark
    \ref{rmk:neat3}. Alternatively, this also follows from the proof
    of \cite[Theorem~2]{ahmadi2010shparlinski} for Jacobians of genus
    2 curves, which generalizes to any abelian
    surface. \Cref{mainthm:structure-thm} then implies that
    $\SF(S) = \UU(1)^2$.
    
  \item[(b,c,d,e)] Denote by $m$ the splitting degree of $S$. By
    \Cref{thm:HZ6} we know that $m \in \brk{2,3,4,6}$. Let
    $\alpha \in \brk{\alpha_1, \overline{\alpha}_1, \alpha_2,
      \overline{\alpha}_2}$ be a Frobenius eigenvalue of $S$. From
    \cite[Lemma 4]{howe2002existence} and since $S$ is ordinary, we
    have that $[\QQ(\alpha):\QQ(\alpha^m)] = [\QQ(\alpha^m):\QQ] =
    2$. In particular, the minimal polynomial $h_{(m)}(T)$ of
    $\alpha^m$ is quadratic, and $P_{(m)}(T) = h_{(m)}(T)^2$. This
    implies that
    $\brk{\alpha_1^m, \overline{\alpha}_1^m} = \brk{\alpha_2^m,
      \overline{\alpha}_2^m}$, so that (up to relabelling) there is a
    primitive\footnote{Note that $\zeta_m$ must be primitive, since
      otherwise, $P_{(n)}(T)$ would split for some $n \leq m$,
      contradicting the minimality of $m$.} $m$-th root of unity
    $\zeta_m$ such that $\alpha_2 = \zeta_m \alpha_1$.  It follows
    that
    \begin{equation*}
      \SF(S) = \overline{\langle (u_1, \zeta_m u_1) \rangle} = \brk{(u, \zeta_m^\nu u) : u \in \UU(1), \, \nu \in \ZZ/m\ZZ} \cong \UU(1)\times C_m
    \end{equation*}
    and $\SF(S)^\circ$ embeds diagonally in $\UU(1)^2$.
  \end{enumerate}
\end{proof}

\begin{figure}[H]
  \centering
  \begin{subfigure}{0.32\textwidth}
    \centering
    \includegraphics[width=\textwidth]{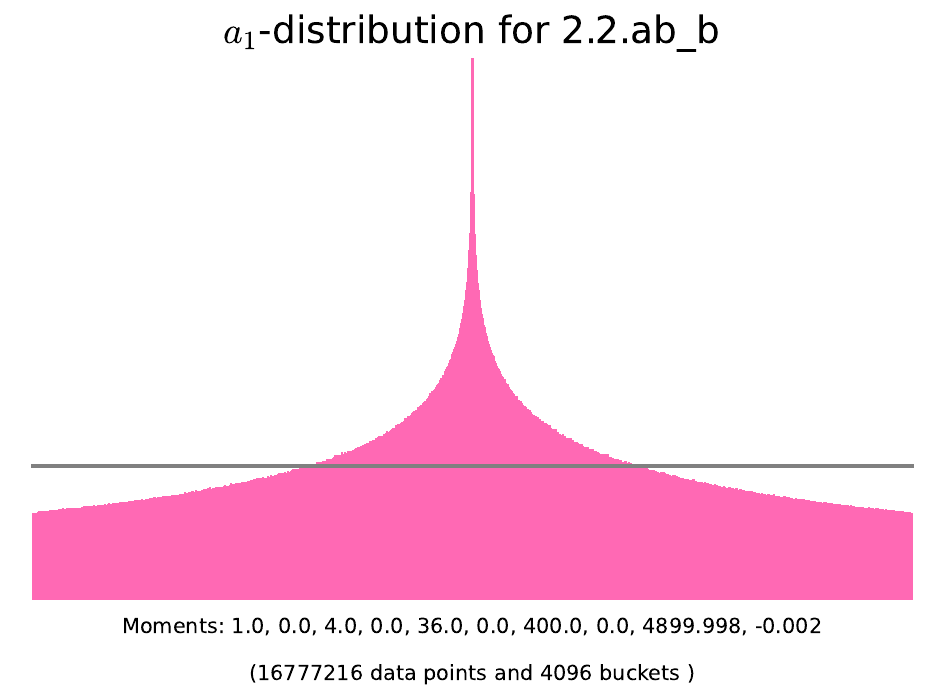}
    \caption{\href{https://www.lmfdb.org/Variety/Abelian/Fq/2/2/ab_b}{\texttt{2.2.ab\_b}}}
    \label{fig:simple-ord-S}
  \end{subfigure}
  \hfill
  \begin{subfigure}{0.32\textwidth}
    \centering
    \includegraphics[width=\textwidth]{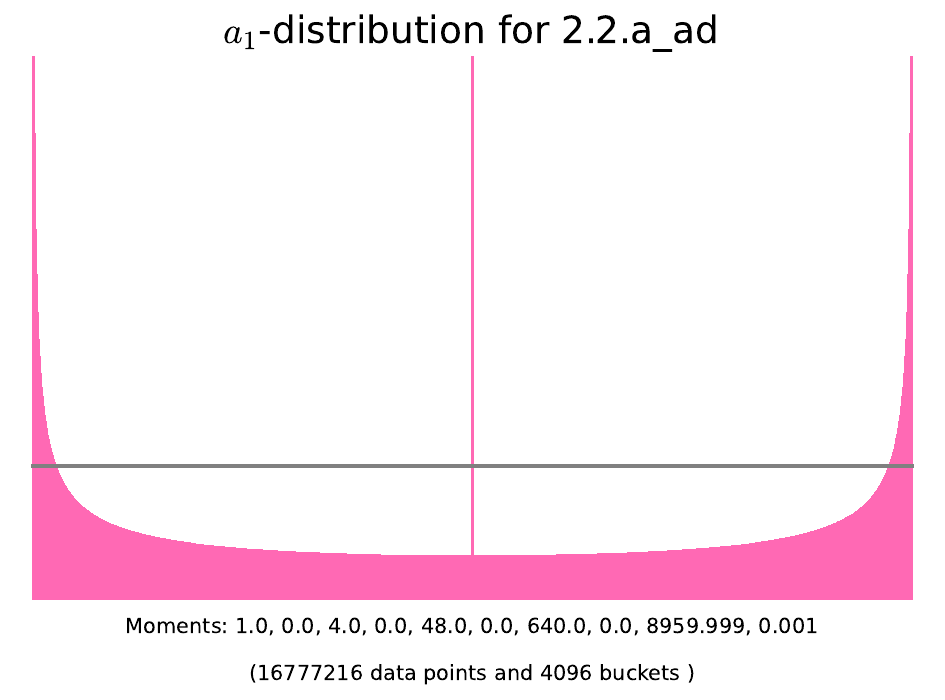}
    \caption{\href{https://www.lmfdb.org/Variety/Abelian/Fq/2/2/a_ad}{\texttt{2.2.a\_ad}}}
    \label{fig:S-quad-split}
  \end{subfigure}
  \hfill
  \begin{subfigure}{0.32\textwidth}
    \centering
    \includegraphics[width=\textwidth]{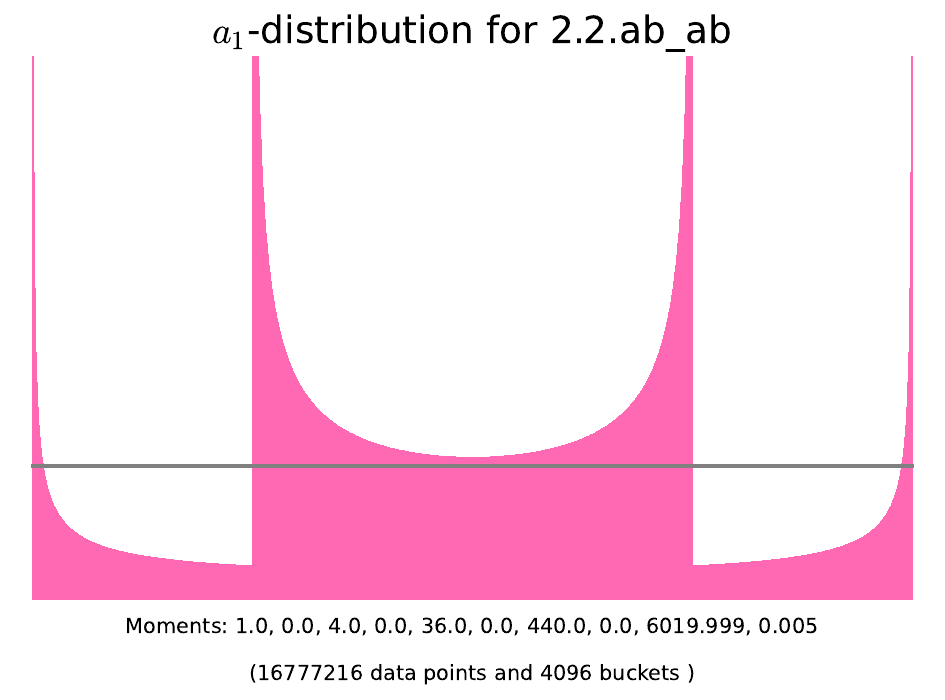}
    \caption{\href{https://www.lmfdb.org/Variety/Abelian/Fq/2/2/ab_ab}{\texttt{2.2.ab\_ab}}}
    \label{fig:S-cubic-split}
  \end{subfigure}
  \hfill \vspace{5mm}
  \begin{subfigure}{0.32\textwidth}
    \centering
    \includegraphics[width=\textwidth]{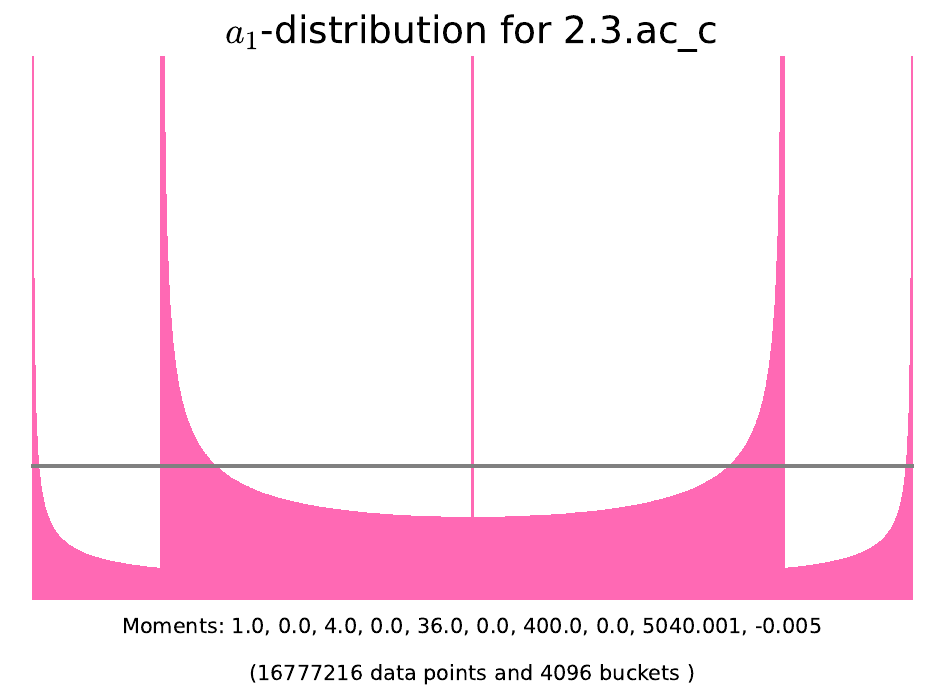}
    \caption{\href{https://www.lmfdb.org/Variety/Abelian/Fq/2/3/ac_c}{\texttt{2.3.ac\_c}}}
    \label{fig:S-quartic-split}
  \end{subfigure}
  \hspace{2cm}
  \begin{subfigure}{0.32\textwidth}
    \centering
    \includegraphics[width=\textwidth]{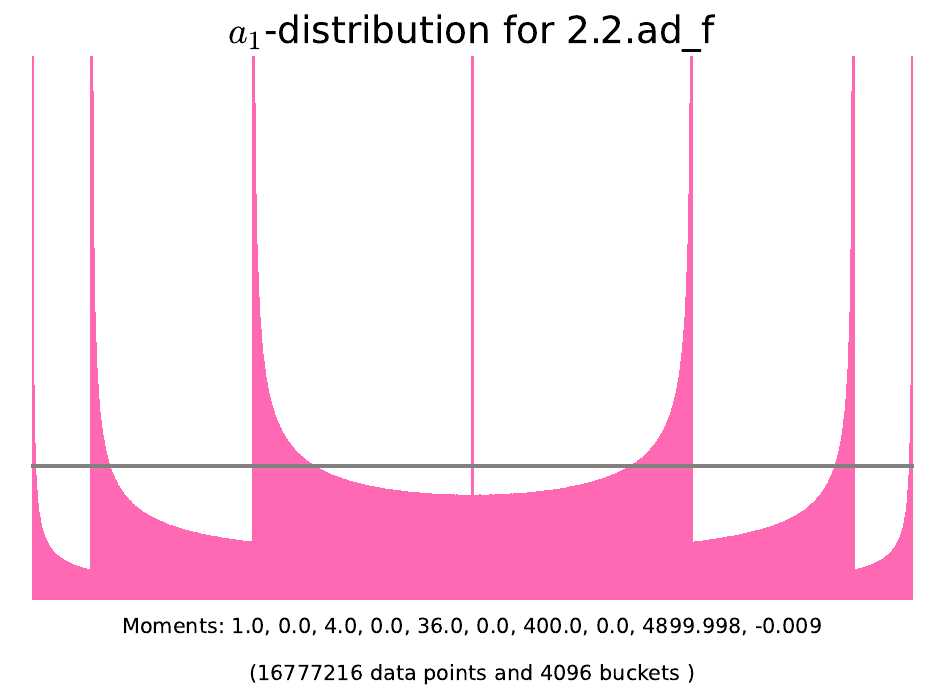} \caption{\href{https://www.lmfdb.org/Variety/Abelian/Fq/2/2/ad_f}{\texttt{2.2.ad\_f}}}
    \label{fig:S-sextic-split}
  \end{subfigure}
  \caption{$a_1$-histograms for simple ordinary abelian surfaces.}
  \label{fig:simple-ord-surfaces}
\end{figure}

\begin{notation}
  In \Cref{table:simple-ordinary-surfaces}, the \emph{Splitting type}
  title refers to that of the abelian variety $S$, the \emph{$\cong$
    class} title refers to the Serre--Frobenius group $\SF(S)$, and
  the \emph{Generator} title refers to the topological generator of
  $\SF(S)$; which precisely and succinctly captures the data of the
  embedding of the Serre--Frobenius group into the ambient unitary
  symplectic group. We will follow the same conventions in the
  following tables.
\end{notation}

\begin{table}[H]
  \setlength{\arrayrulewidth}{0.3mm} \setlength{\tabcolsep}{5pt}
  \renewcommand{\arraystretch}{1.3}
  \caption{Serre--Frobenius groups of simple ordinary surfaces.}
  \begin{longtable}{|c|c|c|c|c|}
    \hline
    \rowcolor{header_color} 
    Splitting type& $\cong$ class& Generator& Example & \Cref{fig:simple-ord-surfaces} \\ \hline
    Absolutely simple & $\UU(1)^2$ & $(u_1,u_2)$ & \href{https://www.lmfdb.org/Variety/Abelian/Fq/2/2/ab_b}{\texttt{2.2.ab\_b}} & \ref{fig:simple-ord-S} \\ \hline
    $S_{(2)} \sim E^2$ & $\UU(1)\times C_2$  & $(u_1, -u_1)$ & \href{https://www.lmfdb.org/Variety/Abelian/Fq/2/2/a_ad}{\texttt{2.2.a\_ad}} & \ref{fig:S-quad-split} \\ \hline
    $S_{(3)} \sim E^2$ & $\UU(1)\times C_3$ & $(u_1, \zeta_3 u_1)$ & \href{https://www.lmfdb.org/Variety/Abelian/Fq/2/2/ab_ab}{\texttt{2.2.ab\_ab}} & \ref{fig:S-cubic-split} \\ \hline
    $S_{(4)} \sim E^2$ & $\UU(1)\times C_4 $  & $(u_1, \zeta_4u_1)$ & \href{https://www.lmfdb.org/Variety/Abelian/Fq/2/3/ac_c}{\texttt{2.3.ac\_c}} & \ref{fig:S-quartic-split} \\ \hline
    $S_{(6)} \sim E^2$  & $\UU(1)\times C_6 $  & $(u_1, \zeta_6 u_1)$ & \href{https://www.lmfdb.org/Variety/Abelian/Fq/2/2/ad_f}{\texttt{2.2.ad\_f}} & \ref{fig:S-sextic-split} \\ \hline
  \end{longtable}
  \addtocounter{table}{-1}
  \label{table:simple-ordinary-surfaces}
\end{table}

\subsection{Non-simple ordinary surfaces}
\label{sec:non-simple-ordinary-surfaces}

Let $S$ be a non-simple ordinary abelian surface defined over
$\FF_q$. Then $S$ is isogenous to a product of two ordinary elliptic
curves $E_1\times E_2$. As depicted in \Cref{fig:proof-2}, we consider
two cases: \begin{enumerate}[wide] \item[(S-B)]
  \label{case:SB} $E_1$ and $E_2$ are not isogenous over
  $\overline{\FF}_q$.  \item[(S-C)] \label{case:SC} $E_1$ and $E_2$
  become isogenous over some base extension
  $\FF_{q^{m_1}} \supseteq \FF_q$, for $m_1\geq 1$.  \end{enumerate}
The Serre-Frobenius groups corresponding to these isogeny decomposition types are sumarized in \Cref{table:non-simple-ord-surfaces}. The proof of the following lemma is a straightforward application of \Cref{lemma:poonen}.

\begin{lemma}[Node S-B in \Cref{fig:proof-2}]
  \label{lemma:SB} Let $S$ be an abelian surface defined over $\FF_q$
  such that $S$ is isogenous to $E_1 \times E_2$, for $E_1$ and $E_2$
  geometrically non-isogenous ordinary elliptic curves. Then $S$ has
  maximal angle rank $\delta = 2$ and $\SF(S) = \UU(1)^2$.
\end{lemma}

\begin{lemma}[Node S-C in \Cref{fig:proof-2}]
  Let $S$ be an abelian surface defined over $\FF_q$ such that $S$ is
  isogenous to $E_1 \times E_2$, for $E_1$ and $E_2$ geometrically
  isogenous ordinary elliptic curves. Then $S$ has angle rank
  $\delta = 1$ and $\SF(S) = \UU(1)\times C_m$ for
  $m \in \brk{1,2,3,4,6}$. Furthermore, $m$ is precisely the degree of
  the extension of $\FF_q$ over which $E_1$ and $E_2$ become
  isogenous.
\end{lemma}

\begin{proof}
  Let $\alpha_1,\overline{\alpha}_1$ and
  $ \alpha_2,\overline{\alpha}_2$ denote the Frobenius eigenvalues of
  $E_1$ and $E_2$ respectively. Let $m_1$ be the smallest positive
  integer such that $(E_1)_{(m_1)} \sim (E_2)_{(m_1)}$. From
  \Cref{lemma:power-ord-ec}, we immediately have that
  $\SF(S) \cong \UU(1) \times C_m$, where $m=m_1$.  In order to find
  the value of $m$, observe that
  $\brk{\alpha_1^m,\overline{\alpha}_1^m} =
  \brk{\alpha_2^m,\overline{\alpha}_2^m}$, from which we may assume,
  possibly after relabelling, that $\alpha_2 = \zeta_m \alpha_1$ for
  some primitive $m$-th root of unity $\zeta_m$. Since the curves
  $E_1$ and $E_2$ are ordinary, the number fields $\QQ(\alpha_1)$ and
  $\QQ(\alpha_2)$ are imaginary quadratic and
  $\QQ(\alpha_1) = \QQ(\alpha_1^m) = \QQ(\alpha_2^m) =
  \QQ(\alpha_2)$. Hence, $\zeta_m \in \QQ(\alpha_1)$ and thus
  $\varphi(m) = [\QQ(\zeta_m):\QQ] \in \brk{1,2}$; therefore
  $m \in \brk{1,2,3,4,6}$. Finally, we have by definition that
  $\SF(S) = \brk{(u,\zeta_m^\nu u) : u \in \UU(1), \, \nu \in
    \ZZ/m\ZZ} \cong \UU(1)\times C_m$.
\end{proof}

\begin{figure}[H]
  \centering
  \begin{subfigure}{0.32\textwidth}
    \centering
    \includegraphics[width=0.8\textwidth]{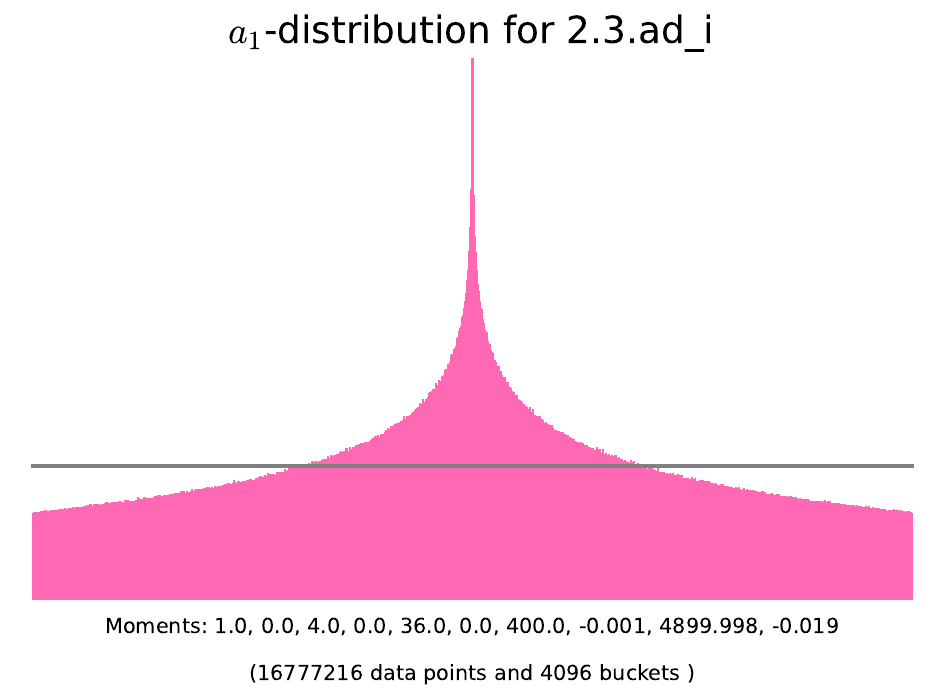}
    \caption{\href{https://www.lmfdb.org/Variety/Abelian/Fq/2/2/ac_f}{\texttt{2.2.ac\_f}}}
    \label{fig:non-isog-ord}
  \end{subfigure}
  \hfill
  \begin{subfigure}{0.32\textwidth}
    \centering
    \includegraphics[width=0.8\textwidth]{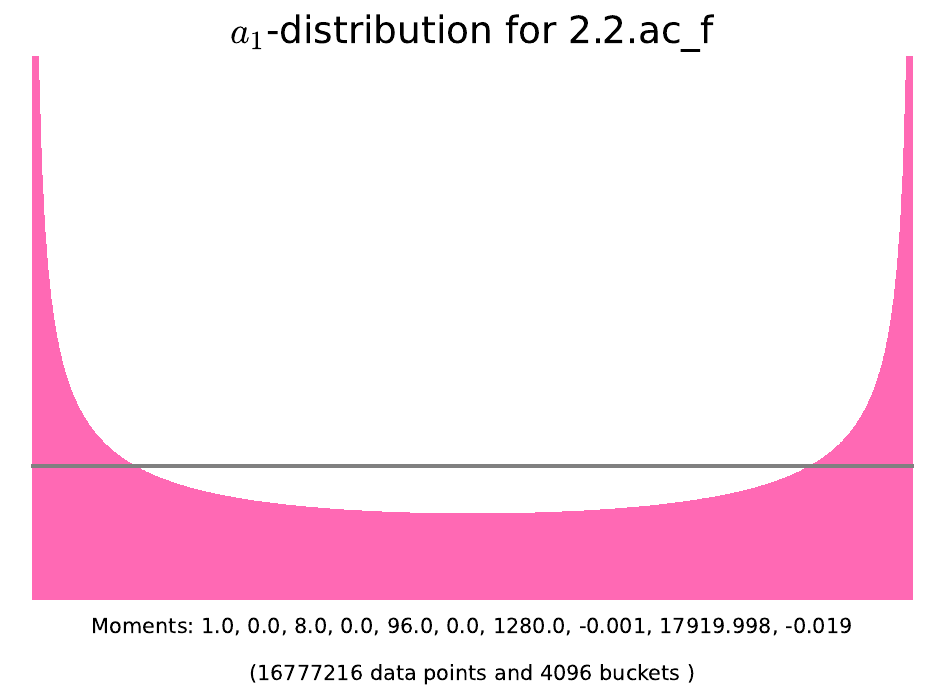}
    \caption{\href{https://www.lmfdb.org/Variety/Abelian/Fq/2/2/ac_f}{\texttt{2.2.ac\_f}}}
    \label{fig:isog-ord}
  \end{subfigure}
  \hfill
  \begin{subfigure}{0.32\textwidth}
    \centering
    \includegraphics[width=0.8\textwidth]{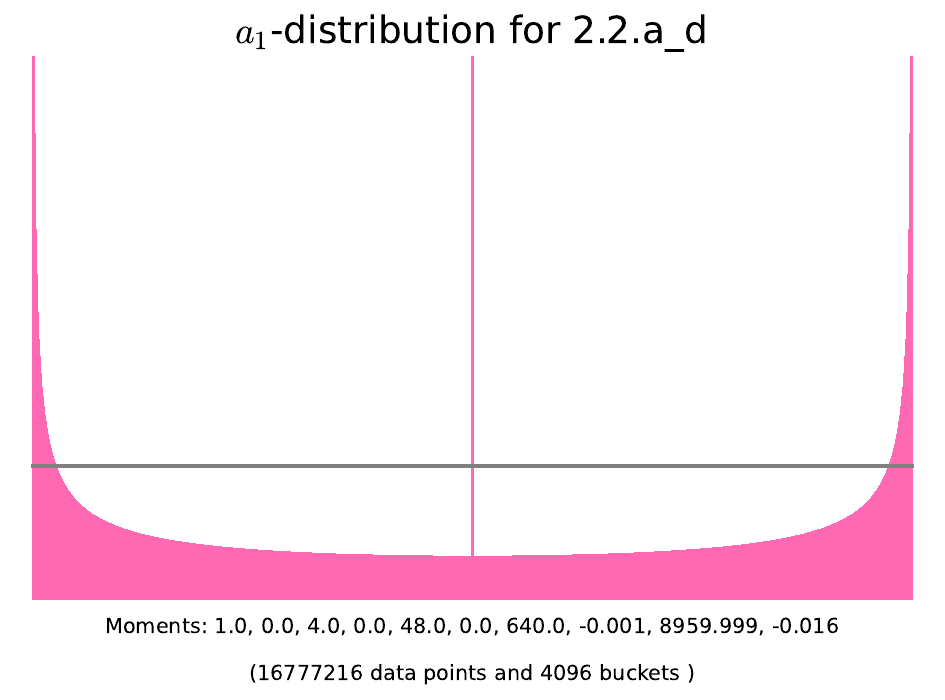}
    \caption{\href{https://www.lmfdb.org/Variety/Abelian/Fq/2/2/a_d}{\texttt{2.2.a\_d}}}
    \label{fig:isog-2-ord}
  \end{subfigure}
  \begin{subfigure}{0.32\textwidth}
    \centering
    \includegraphics[width=0.8\textwidth]{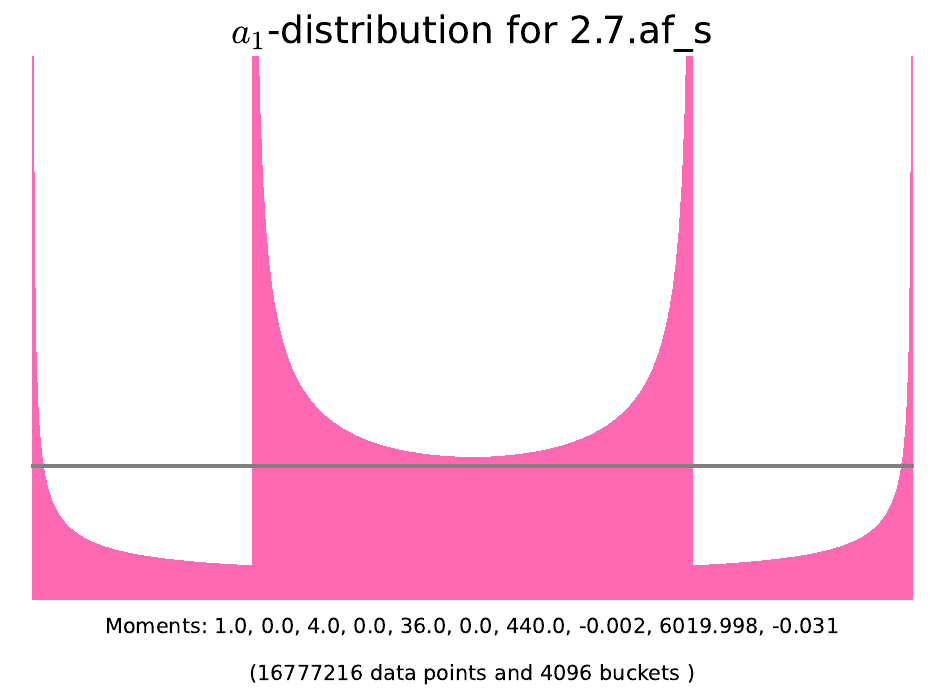}
    \caption{\href{https://www.lmfdb.org/Variety/Abelian/Fq/2/7/af_s}{\texttt{2.7.af\_s}}}
    \label{fig:isog-3-ord}
  \end{subfigure}
  \hfill
  \begin{subfigure}{0.32\textwidth}
    \centering
    \includegraphics[width=0.8\textwidth]{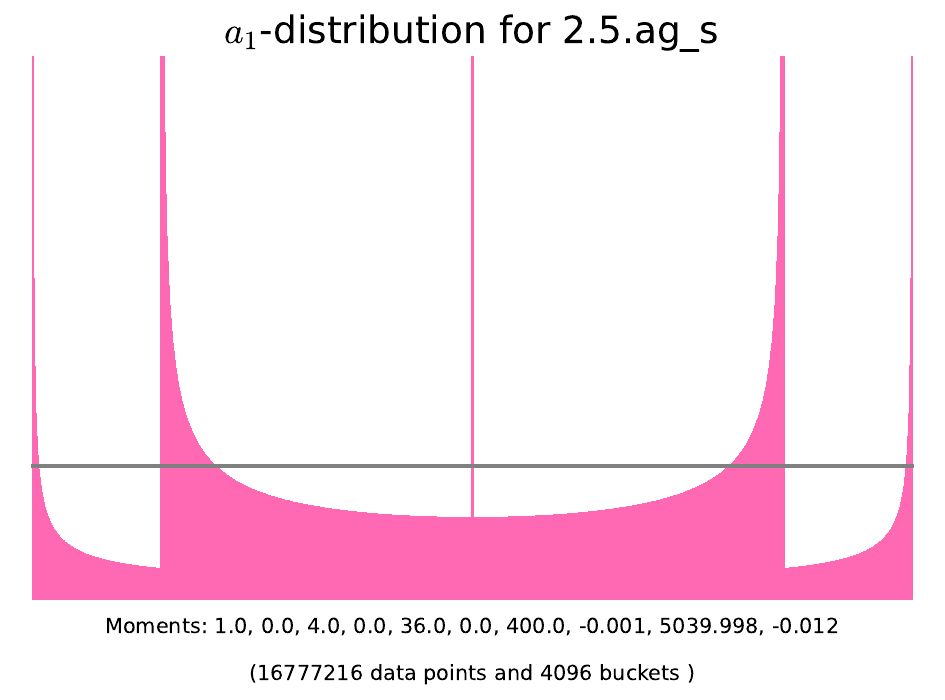}
    \caption{\href{https://www.lmfdb.org/Variety/Abelian/Fq/2/5/ag_s}{\texttt{2.5.ag\_s}}}
    \label{fig:isog-4-ord}
  \end{subfigure}
  \hfill
  \begin{subfigure}{0.32\textwidth}
    \centering
    \includegraphics[width=0.8\textwidth]{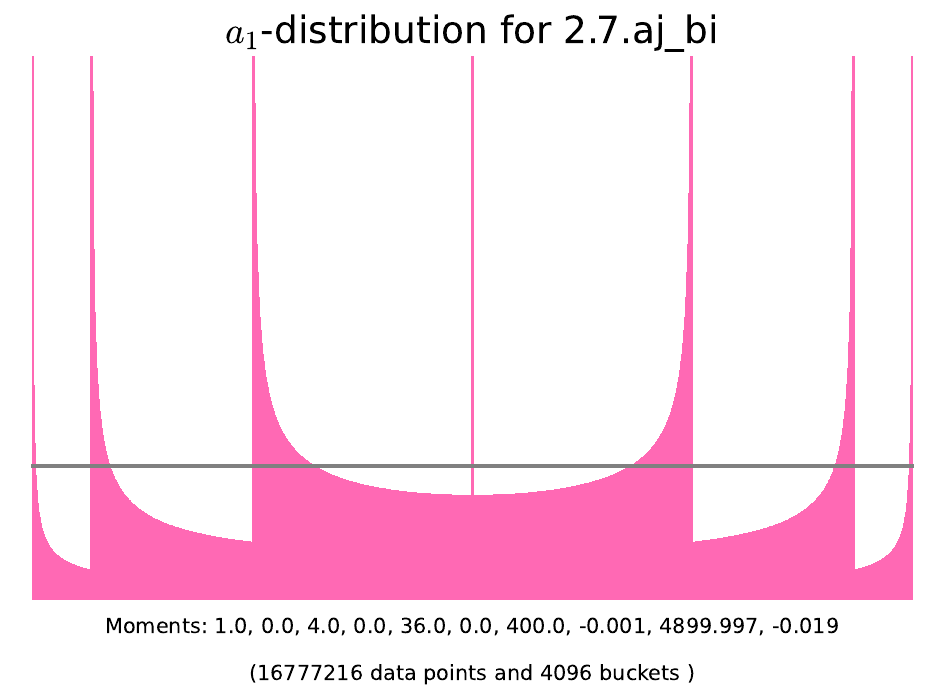}
    \caption{\href{https://www.lmfdb.org/Variety/Abelian/Fq/2/7/aj_bi}{\texttt{2.7.aj\_bi}}}
    \label{fig:isog-6-ord}
  \end{subfigure}
  \caption{$a_1$-histograms of non-simple ordinary abelian surfaces.}
  \label{fig:non-simple-ord-surfaces}
\end{figure}

\begin{table}[H]
  \setlength{\arrayrulewidth}{0.3mm} \setlength{\tabcolsep}{5pt}
  \renewcommand{\arraystretch}{1.3}
  \caption{Serre--Frobenius groups of non-simple ordinary surfaces
    $S = E_1\times E_2$.}
  \begin{longtable}{|c|c|c|c|c|}
    \hline
    \rowcolor{header_color} 
    Splitting type & $\cong$ class& Generator& Example & \Cref{fig:non-simple-ord-surfaces} \\ \hline
    $(E_1)_{\overline{\FF}_p} \not \sim (E_2)_{\overline{\FF}_p}$ & $\UU(1)^2$ & $(u_1,u_2)$ & \href{https://www.lmfdb.org/Variety/Abelian/Fq/2/3/ad_i}{\texttt{2.3.ad\_i}}  & \ref{fig:non-isog-ord} \\ \hline
    $E_1 \sim E_2$ & $\UU(1)$ & $(u_1, u_1)$ & \href{https://www.lmfdb.org/Variety/Abelian/Fq/2/2/ac_f}{\texttt{2.2.ac\_f}}  & \ref{fig:isog-ord} \\ \hline
    $(E_1)_{(2)} \sim (E_2)_{(2)}$ & $\UU(1)\times C_2 $ & $(u_1, -u_1)$ & \href{https://www.lmfdb.org/Variety/Abelian/Fq/2/2/a_d}{\texttt{2.2.a\_d}} & \ref{fig:isog-2-ord} \\ \hline
    $(E_1)_{(3)} \sim (E_2)_{(3)}$ & $\UU(1)\times C_3$ & $(u_1, \zeta_3u_1)$ & \href{https://www.lmfdb.org/Variety/Abelian/Fq/2/7/af_s}{\texttt{2.7.af\_s}}  & \ref{fig:isog-3-ord} \\ \hline
    $(E_1)_{(4)} \sim (E_2)_{(4)}$ & $\UU(1)\times C_4 $ & $(u_1, \zeta_4u_1)$ & \href{https://www.lmfdb.org/Variety/Abelian/Fq/2/5/ag_s}{\texttt{2.5.ag\_s}} & \ref{fig:isog-4-ord} \\ \hline
    $(E_1)_{(6)} \sim (E_2)_{(6)}$ & $\UU(1)\times C_6$ & $(u_1, \zeta_6u_1)$ & \href{https://www.lmfdb.org/Variety/Abelian/Fq/2/7/aj_bi}{\texttt{2.7.aj\_bi}}  & \ref{fig:isog-6-ord} \\ \hline
  \end{longtable}
  \addtocounter{table}{-1}
  \label{table:non-simple-ord-surfaces}
\end{table}

\subsection{Simple almost ordinary surfaces}
\label{sec:simple-ao-surf}

In \cite{lenstra1993} Lenstra and Zarhin carried out a careful study
of the multiplicative relations of Frobenius eigenvalues of simple
almost ordinary varieties (see \cref{subsec:background-av} for the
definition), which was later generalized in \cite{dupuy2022angle}. In
particular, they prove that even dimensional simple almost ordinary
abelian varieties have maximal angle rank (\cite[Theorem
5.8]{lenstra1993}). Since every abelian surface of $p$-rank $1$ is
almost ordinary, their result allows us to deduce the following.

\begin{lemma}[Node S-D in \Cref{fig:proof-2}]
  \label{lemma:simple-ao-surfaces}
  Let $S$ be a simple and almost ordinary abelian surface defined over
  $\FF_q$. Then $S$ has maximal angle rank $\delta = 2$ and
  $\SF(S) = \UU(1)^2$.
\end{lemma}

\subsection{Non-simple almost ordinary surfaces}
\label{sec:non-simple-ao-surf}
If $S$ is almost ordinary and not simple, then $S$ is isogenous to the product
of an ordinary elliptic curve $E_1$ and a supersingular elliptic curve $E_2$.
The corresponding Serre-Frobenius groups are sumarized in
\Cref{table:non-simple-ao-surfaces}.

\begin{lemma}[Node S-E in \Cref{fig:proof-2}]
  Let $S$ be a non-simple almost ordinary abelian surface defined over
  $\FF_q$. Then $S$ has angle rank $\delta =1$ and angle torsion order
  $m \in \brk{1, 2,3,4,6,8,12}$. Furthermore,
  $\SF(S) = \brk{(u,\zeta_m^\nu) : u\in \UU(1), \, \nu \in \ZZ/m\ZZ}
  \cong \UU(1)\times C_m$.
\end{lemma}

\begin{proof}
  Let $E_1$ be an ordinary elliptic curve and $E_2$ a supersingular
  elliptic curve such that $S \sim E_1\times E_2$. By
  \Cref{lemma:A1xB},
  $\SF(S) = \SF(E_1)\times\SF(E_2) \cong \UU(1)\times C_m$ with $m$ in
  the list of possible orders of Serre--Frobenius groups of
  supersingular elliptic curves.
\end{proof}

\begin{figure}[H]
  \centering
  \begin{subfigure}{0.24\textwidth}
    \centering
    \includegraphics[width=\textwidth]{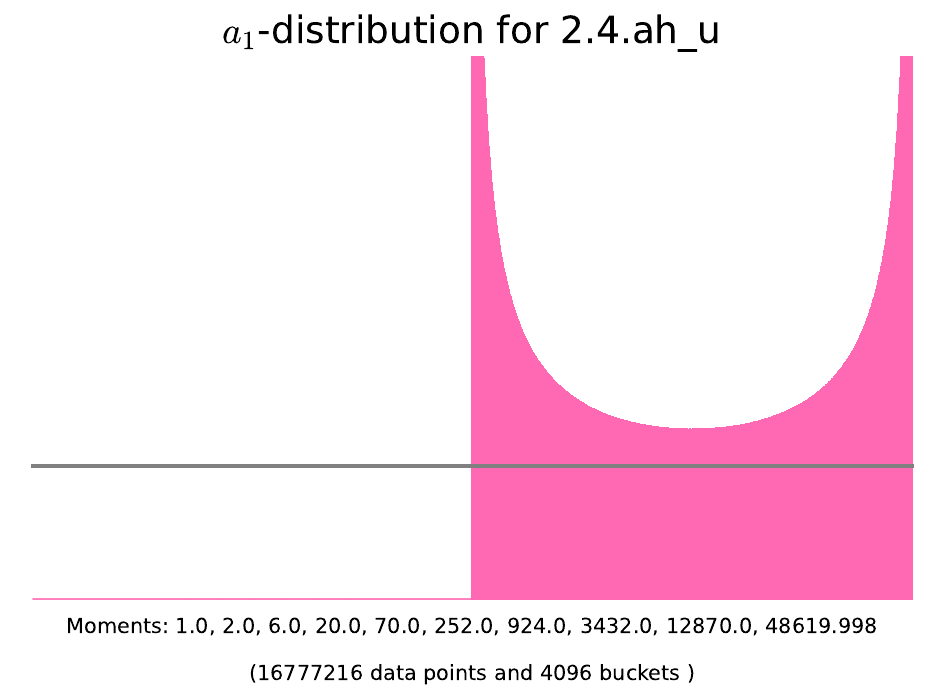}
    \caption{\href{https://www.lmfdb.org/Variety/Abelian/Fq/2/4/ah_u}{\texttt{2.4.ah\_u}}}
    \label{fig:Ux1}
  \end{subfigure}
  \hfill
  \begin{subfigure}{0.24\textwidth}
    \centering
    \includegraphics[width=\textwidth]{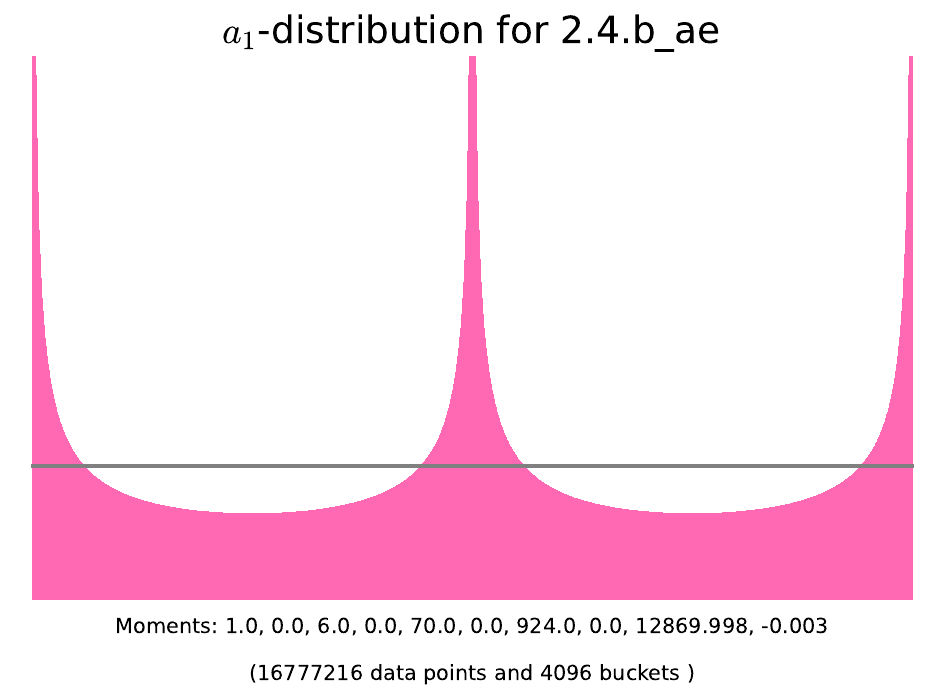}
    \caption{\href{https://www.lmfdb.org/Variety/Abelian/Fq/2/4/b_ae}{\texttt{2.4.b\_ae}}}
    \label{fig:Ux2}
  \end{subfigure}
  \hfill
  \begin{subfigure}{0.24\textwidth}
    \centering
    \includegraphics[width=\textwidth]{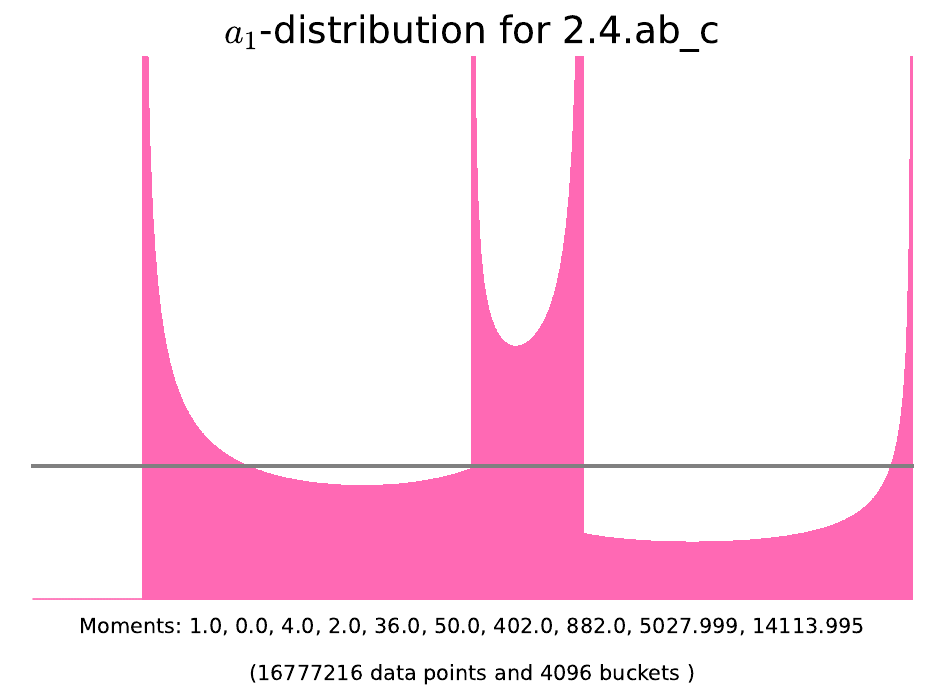}
    \caption{\href{https://www.lmfdb.org/Variety/Abelian/Fq/2/4/ab_c}{\texttt{2.4.ab\_c}}}
    \label{fig:Ux3}
  \end{subfigure}
  \hfill
  \begin{subfigure}{0.24\textwidth}
    \centering
    \includegraphics[width=\textwidth]{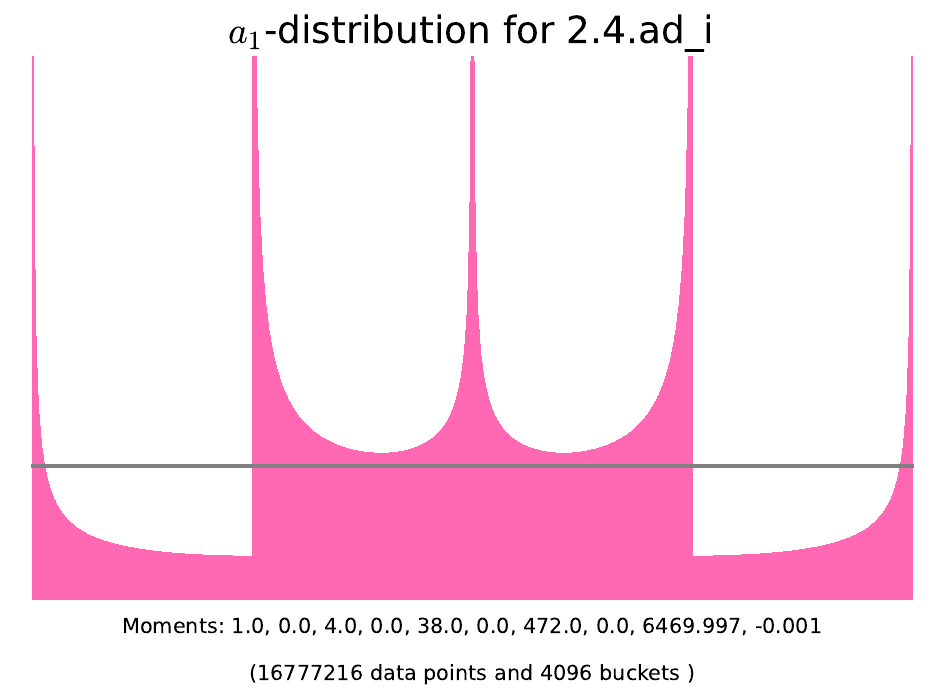}
    \caption{
      \href{https://www.lmfdb.org/Variety/Abelian/Fq/2/4/ad_i}{\texttt{2.4.ad\_i}}}
    \label{fig:Ux4}
  \end{subfigure}
  \hfill
  \begin{subfigure}{0.32\textwidth}
    \centering
    \includegraphics[width=0.8\textwidth]{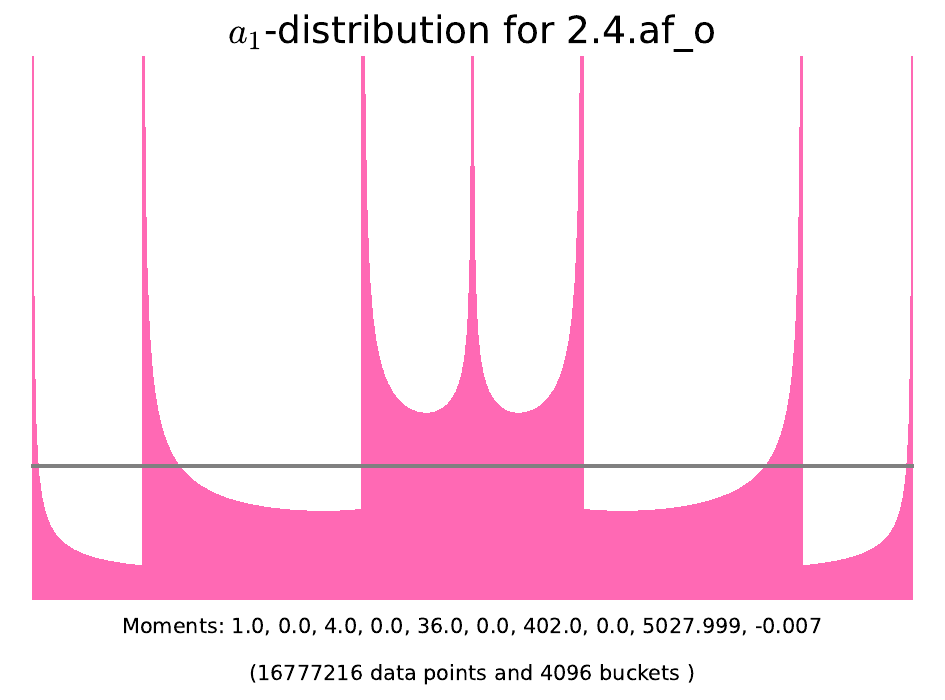}
    \caption{\href{https://www.lmfdb.org/Variety/Abelian/Fq/2/4/af_o}{\texttt{2.4.af\_o}}}
    \label{fig:Ux6}
  \end{subfigure}
  \hfill
  \begin{subfigure}{0.32\textwidth}
    \centering
    \includegraphics[width=0.8\textwidth]{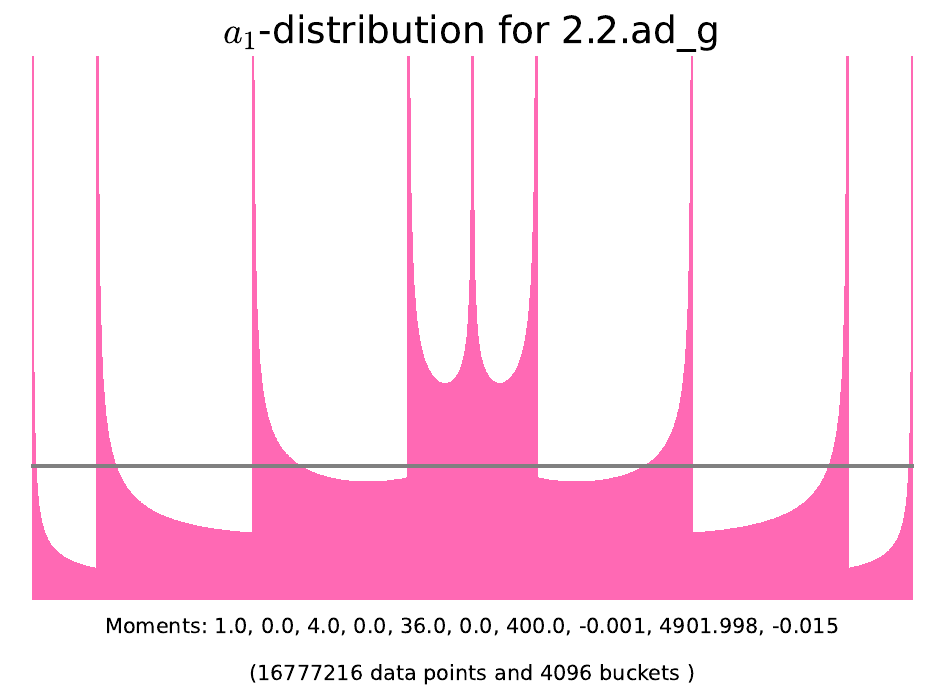}
    \caption{\href{https://www.lmfdb.org/Variety/Abelian/Fq/2/2/ad_g}{\texttt{2.2.ad\_g}}}
    \label{fig:Ux8}
  \end{subfigure}
  \hfill
  \begin{subfigure}{0.32\textwidth}
    \centering
    \includegraphics[width=0.8\textwidth]{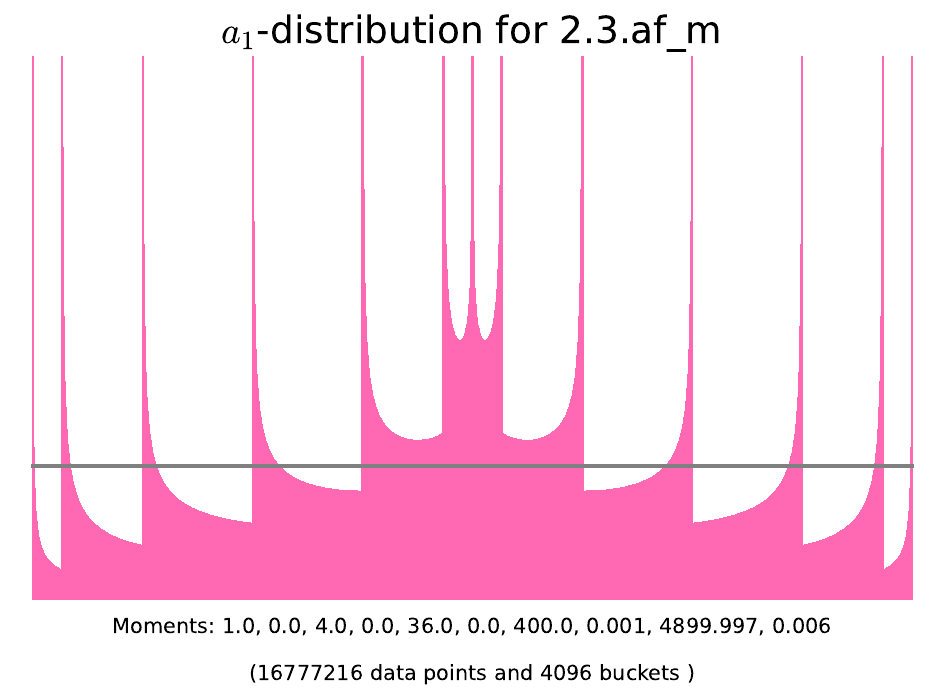}
    \caption{\href{https://www.lmfdb.org/Variety/Abelian/Fq/2/3/af_m}{\texttt{2.3.af\_m}}}
    \label{fig:Ux12}
  \end{subfigure}
  \caption{$a_1$-distributions of non-simple almost ordinary abelian
    surfaces.}
  \label{fig:non-simple-almost-ord-S}
\end{figure}

\begin{table}[H]
  \setlength{\arrayrulewidth}{0.3mm} \setlength{\tabcolsep}{5pt}
  \renewcommand{\arraystretch}{1.3}
  \caption{Serre--Frobenius groups of non-simple almost ordinary
    surfaces.}
  \begin{longtable}{|c|c|c|c|}
    \hline
    \rowcolor{header_color} 
    $\cong$ class& Generator& Example & \Cref{fig:non-simple-almost-ord-S} \\ \hline
    $\UU(1)$ & $(u_1,1)$ & \href{https://www.lmfdb.org/Variety/Abelian/Fq/2/4/ah_u}{\texttt{2.4.ah\_u}}  & \ref{fig:Ux1} \\ \hline
    $\UU(1)\times C_2$ & $(u_1, -1)$ & \href{https://www.lmfdb.org/Variety/Abelian/Fq/2/4/b_ae}{\texttt{2.4.b\_ae}}  & \ref{fig:Ux2} \\ \hline
    $\UU(1)\times C_3 $ & $(u_1, \zeta_3)$ & \href{https://www.lmfdb.org/Variety/Abelian/Fq/2/4/ab_c}{\texttt{2.4.ab\_c}} & \ref{fig:Ux3} \\ \hline
    $\UU(1)\times C_4$ & $(u_1, \zeta_4)$ & \href{https://www.lmfdb.org/Variety/Abelian/Fq/2/4/ad_i}{\texttt{2.4.ad\_i}}  & \ref{fig:Ux4} \\ \hline
    $\UU(1)\times C_6 $ & $(u_1, \zeta_6)$ & \href{https://www.lmfdb.org/Variety/Abelian/Fq/2/4/af_o}{\texttt{2.4.af\_o}} & \ref{fig:Ux6} \\ \hline
    $\UU(1)\times C_8$ & $(u_1, \zeta_8)$ & \href{https://www.lmfdb.org/Variety/Abelian/Fq/2/2/ad_g}{\texttt{2.2.ad\_g}}  & \ref{fig:Ux8} \\ \hline
    $\UU(1)\times C_{12}$ & $(u_1, \zeta_{12})$ & \href{https://www.lmfdb.org/Variety/Abelian/Fq/2/3/af_m}{\texttt{2.3.af\_m}}  & \ref{fig:Ux12} \\ \hline
  \end{longtable}
  \addtocounter{table}{-1}
  \label{table:non-simple-ao-surfaces}
\end{table}

\subsection{Simple supersingular surfaces}
\label{sec:simple-ss-surf}
Since every supersingular abelian variety is geometrically isogenous
to a power of an elliptic curve, the Serre--Frobenius group only
depends on the splitting degree. We separate our analysis into the
simple and non-simple cases.

The classification of Frobenius polynomials of supersingular abelian
surfaces over finite fields was completed by Maisner and Nart
\cite[Theorem 2.9]{maisner&nart2002} building on work of Xing
\cite{xing96} and R\"{u}ck \cite{ruck1990abelian}. Denoting by
$(a_1, a_2)$ the isogeny class of abelian surfaces over $\FF_q$ with
Frobenius polynomial $P_S(T) = T^4 + a_1T^3+ a_2T^2 + qa_1T + q^2$,
the following lemma gives the classification of Serre--Frobenius
groups of simple supersingular surfaces.

\begin{lemma}[Node S-F in \Cref{fig:proof-2}]
  \label{lemma:simple-ss-surfaces}
  Let $S$ be a simple supersingular abelian surface defined over
  $\FF_q$. The Serre--Frobenius group of $S$ is classified according
  to \Cref{table:ss-simple-surfaces}.
\end{lemma}

\begin{table}[H]
  \setlength{\arrayrulewidth}{0.3mm} \setlength{\tabcolsep}{5pt}
  \renewcommand{\arraystretch}{1.3}
  \caption{Serre--Frobenius groups of simple supersingular surfaces.}
  \begin{longtable}{|c|c|c|c|c|c|c|c|}
    \hline
    \rowcolor{header_color} 
    $(a_1, a_2)$       & $p$                    & $d$  & $e$ & Type & $\tilde{h}(T)$               & $\SF(S)$ & Example \\ \hline
    $(0,0)$            & $\not\equiv 1 \md 8$   & even & $1$ & Z-1  & $\Phi_8(T)$                  & $C_8$  &  \href{https://www.lmfdb.org/Variety/Abelian/Fq/2/4/a_a}{\texttt{2.4.a\_a}}\\ \hline
    $(0,0)$            & $\neq 2$               & {odd}  & $1$ & Z-2  & $\Phi_8(T)$                  & $C_8$ &  \href{https://www.lmfdb.org/Variety/Abelian/Fq/2/3/a_a}{\texttt{2.3.a\_a}} \\ \hline
    $(0,q)$            & -                      & odd  & $1$ & Z-2  & $\Phi_3(T^2)$                & $C_6$  & \href{https://www.lmfdb.org/Variety/Abelian/Fq/2/2/a_c}{\texttt{2.2.a\_c}} \\ \hline
    $(0,-q)$           & $\not \equiv 1 \md 12$ & even & $1$ & Z-1  & $\Phi_{12}(T)$               & $C_{12}$ &\href{https://www.lmfdb.org/Variety/Abelian/Fq/2/4/a_ae}{\texttt{2.4.a\_ae}}
    \\ \hline
    $(0,-q)$           & $\neq 3$               & odd  & $1$ & Z-2  & $\Phi_6(T^2) = \Phi_{12}(T)$ & $C_{12}$ & \href{https://www.lmfdb.org/Variety/Abelian/Fq/2/2/a_ac}{\texttt{2.2.a\_ac}}\\ \hline
    $(\sqrt{q},q)$     & $\not\equiv 1 \md 5$   & even & $1$ & Z-1  & $\Phi_5(T)$                  & $C_5$  & \href{https://www.lmfdb.org/Variety/Abelian/Fq/2/4/c_e}{\texttt{2.4.c\_e}} \\ \hline
    $(-\sqrt{q},q)$    & $\not\equiv 1 \md 5$   & even & $1$ & Z-1  & $\Phi_{10}(T) = \Phi_5(-T)$  & $C_{10}$ & \href{https://www.lmfdb.org/Variety/Abelian/Fq/2/4/ac_e}{\texttt{2.4.ac\_e}}\\ \hline
    $(\sqrt{5q}, 3q)$  & $=5$   & odd  & $1$ & Z-3  & $\Psi_{5,1}(T) $             & $C_{10}$ &  \href{https://www.lmfdb.org/Variety/Abelian/Fq/2/5/f_p}{\texttt{2.5.f\_p}}\\ \hline
    $(-\sqrt{5q}, 3q)$ & $=5$                   & odd  & $1$ & Z-3  & $\Psi_{5,1}(-T) $            & $C_{10}$ & \href{https://www.lmfdb.org/Variety/Abelian/Fq/2/5/af_p}{\texttt{2.5.af\_p}}\\ \hline
    $(\sqrt{2q},q)$    & $=2$     & odd  & $1$ & Z-3  & $\Psi_{2,3}(T)$              & $C_{24}$ & \href{https://www.lmfdb.org/Variety/Abelian/Fq/2/2/c_c}{\texttt{2.2.c\_c}} \\ \hline
    $(-\sqrt{2q},q)$   & $=2$                   & odd  & $1$ & Z-3  & $\Psi_{2,3}(-T)$             & $C_{24}$ &\href{https://www.lmfdb.org/Variety/Abelian/Fq/2/2/ac_c}{\texttt{2.2.ac\_c}}\\ \hline
    $(0,-2q)$          & -                      & odd  & $2$ & Z-2  & $\Phi_1(T^2)$                  & $C_2$  & \href{https://www.lmfdb.org/Variety/Abelian/Fq/2/2/a_ae}{\texttt{2.2.a\_ae}} \\ \hline
    $(0, 2q)$          & $\equiv 1\md 4$        & even & $2$ & Z-1  & $\Phi_4(T)$                  & $C_4$   & \href{https://www.lmfdb.org/Variety/Abelian/Fq/2/25/a_by}{\texttt{2.25.a\_by}}\\ \hline
    $(2\sqrt{q}, 3q)$  & $\equiv 1 \md 3$       & even & $2$ & Z-1  & $\Phi_3(T)$                  & $C_3$   &\href{https://www.lmfdb.org/Variety/Abelian/Fq/2/49/o_fr}{\texttt{2.49.o\_fr}} \\ \hline
    $(-2\sqrt{q}, 3q)$ & $\equiv 1 \md 3$       & even & $2$ & Z-1  & $\Phi_6(T) = \Phi_3(-T)$     & $C_6$   &\href{https://www.lmfdb.org/Variety/Abelian/Fq/2/49/ao_fr}{\texttt{2.49.ao\_fr}} \\ \hline
  \end{longtable}
  \addtocounter{table}{-1}
  \label{table:ss-simple-surfaces}
\end{table}
The notation for polynomials of type Z-3 is taken from
\cite{Singh&McGuire&Zaytsev2014}, where the authors classify simple
supersingular Frobenius polynomials for $g\leq 7$. We
have \begin{align}
       \label{eq:psi_5,1} \Psi_{5,1}(T) &\colonequals
                                          \prod_{a\in(\ZZ/5)\unit}\paren{T-\paren{\tfrac{a}{5}}\zeta_5^a} =
                                          T^4 + \sqrt{5}T^3 + 3T^2 + \sqrt{5}T + 1, \\
       \label{eq:psi_2,3} \Psi_{2,3}(T) &\colonequals
                                          \prod_{a\in(\ZZ/3)\unit}\paren{T-\zeta_8\zeta_3^a}\paren{T-\overline{\zeta}_8\zeta_3^a}=
                                          T^4 + \sqrt{2}T^3 + T^2 + \sqrt{2}T + 1.  \end{align}

We exhibit the proof of the second line in \Cref{table:ss-simple-surfaces}
for exposition. The remaining cases can be checked
similarly. If $(a_1,a_2) = (0,0)$, $p \neq 2$ and $q$ is an odd
power of $p$: then, $P(T) = T^4 + q^2 =
\sqrt{q}^{4}\Phi_8(T/\sqrt{q})
= q^2\Phi_4(T^2/q)$ and $\tilde{h}(T) = \Phi_8(T)$. Thus $U_S$ is
generated by a primitive 8th root of unity.

\subsection{Non-simple supersingular surfaces}
\label{sec:non-simple-ss-surf}
If $S$ is a non-simple supersingular surface, then
$S$ is isogenous to a product of two supersingular elliptic
curves $E_1$ and $E_2$. If $m_{E_1}$ and $m_{E_2}$ denote
the torsion orders of $E_1$ and $E_2$ respectively, then
the extension over which $E_1$ and $E_2$ become isogenous is
precisely $\lcm(m_{E_1}, m_{E_2})$. Thus, we have the
following result, depending on the values of $q=p^d$ as in
\Cref{table:elliptic-curves}.

\begin{lemma}[Node S-G in
  \Cref{fig:proof-2}]
  \label{lemma:non-simple-SS-surfaces}
  Let $S$ be a non-simple supersingular abelian
  surface defined over $\FF_q$. Then $S$ has angle
  rank $\delta = 0$ and $\SF(S) = C_m$ for $m$ in
  the set $M = M(p, d)$ described in
  \Cref{table:non-simple-ss-S-torsion}.
\end{lemma}

\begin{table}[H]
  \setlength{\arrayrulewidth}{0.3mm} \setlength{\tabcolsep}{5pt}
  \renewcommand{\arraystretch}{1.3}
  \caption{Angle torsion set for non-simple supersingular surfaces
    defined over $\FF_q$, with $q = p^d$.}
  \label{table:non-simple-ss-S-torsion}
  \begin{longtable}{|c|c|c|}
    \hline
    \rowcolor{header_color} 
    $d$  & $p$                    & $M(p,d)$      \\ \hline
    Even & -                      & $\{1,2\}$     \\ \hline
    Even & $p \not\equiv 1 \md 3$ & $\{1,2,3,6\}$ \\ \hline
    Even & $p \not\equiv 1 \md 4$ & $\{1,2,4\}$   \\ \hline
    Odd  & -                      & $\{4\}$       \\ \hline
    Odd  & $p=2$                  & $\{4,8\}$     \\ \hline
    Odd  & $p=3$                  & $\{4,12\}$    \\ \hline
  \end{longtable}
  \addtocounter{table}{-1}
\end{table}

\section{Abelian Threefolds}
\label{sec:threefolds}

In this section, we classify the Serre--Frobenius groups of abelian
threefolds (see \Cref{fig:proof-3}). Let $X$ be an abelian variety of
dimension $3$ defined over $\FF_q$. For our analysis, we will first
stratify the cases by $p$-rank and then by simplicity. Before we
proceed, we make some observations about simple threefolds that will
be useful later.

\begin{figure}[h]
  \centering \includegraphics[width=1.02\textwidth]{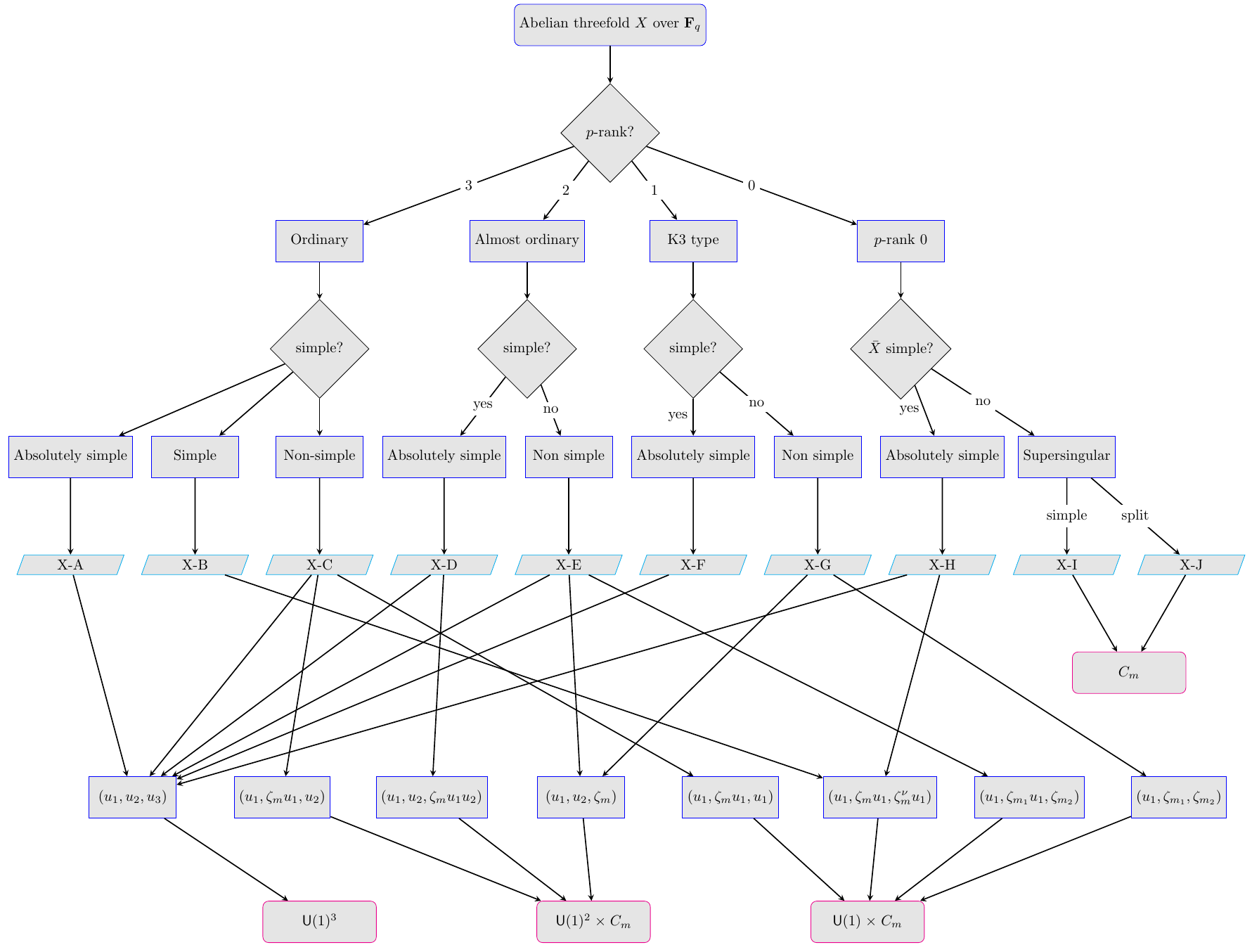}
  \caption{\Cref{mainthm:threefolds}: Classification in dimension 3.}
  \label{fig:proof-3}
\end{figure}

\subsection{Simple abelian threefolds}
\label{subsec:simple-threefolds}

If $X$ is a simple abelian threefold, there are only two possibilities
for the Frobenius polynomial $P_X(T) = h_X(T)^e$:
\begin{align}
  P_X(T) =& \, h_X(T) \label{case:e=1}\\
  P_X(T) =& \, h_X(T)^3. \label{case:e=3}
\end{align}
Indeed, if $h_X(T)$ were a linear or cubic polynomial, it would have a
real root $\pm \sqrt{q}$. By an argument of Waterhouse \cite[Chapter
2]{waterhouse1969abelian}, the $q$-Weil numbers $\pm \sqrt{q}$ must
come from simple abelian varieties of dimension 1 or 2.  Further, Xing
\cite{xing94} showed that \Cref{case:e=3} can only occur in very
special cases.

\begin{theorem}[\cite{xing94}, {\cite[Proposition 1.2]{haloui}}]
  \label{thm:xing}
  Let $X$ be a simple abelian threefold over $\FF_q$. Then,
  $P_X(T) = h_X(T)^3$ if and only if $3$ divides $d = \log_p(q)$ and
  $h_X(T) = T^2 + aq^{1/3}T+q$ with $\gcd(a,p) = 1$.
\end{theorem}

When \(P_X(T)\) is a cube as above, since $\gcd(a,p)=1$, the
\(q\)-adic valuation of its middle coefficient is the same as that of
$aq$, which in turn is \(1\). Thus, $X$ is non-supersingular of
$p$-rank $0$ and its Newton Polygon has slopes
$(\tfrac13, \tfrac13, \tfrac13, \tfrac23, \tfrac23,
\tfrac23)$. Furthermore, every simple abelian threefold is either
absolutely simple or geometrically isogenous to the cube of an
elliptic curve. Thus, we have the following.

\begin{lemma}
  If $X$ is an abelian threefold defined over $\FF_q$ that is not
  ordinary or supersingular, then $X$ is simple if and only if it is
  absolutely simple.
\end{lemma}

\begin{proof}
  Assume $X$ is a simple abelian threefold that is not ordinary or
  supersingular. Assume also that $X$ is not absolutely simple. Let
  $r>1$ be the splitting degree of $X$. Recall that since $X$ is
  simple, one either has $P_X(T) = h_X(T)$ or $P_X(T) = h_X(T)^3$,
  where $h_X(T)$ is irreducible of even degree. We will show that in
  each case $X_{(r)} \sim E^3$, contradicting the assumption that $X$
  is not ordinary or supersingular.

  Assume $P_X(T) = h_X(T)^3$, then $P_{X_{(r)}}(T) =
  h_{X,(r)}(T)^3$. Observe that necessarily $X_{(r)}$ has an elliptic
  curve $E/\FF_{q^r}$ as an isogeny factor. Then $P_E(T)$, a quadratic
  polynomial, must divide $P_{X_{(r)}}(T)$, and we conclude
  $P_E(T) = h_{X,(r)}(T)$. Thus $X_{(r)} \sim E^3$.

  Assume instead that $P_X(T) = h_X(T)$, that is, $P_X(T)$ is
  irreducible. Therefore $\QQ(\pi_X)$ is a degree $6$ extension and
  $\QQ(\pi_X) \hookrightarrow \mathrm{End}^0(X) \hookrightarrow
  \mathrm{End}^0(X_{(r)})$. Then by
  \cite[Theorem~1.3.1.1]{Chai-conrad-oort} $X_{(r)}$ is isotypic, so
  that $X_{(r)} \sim E^3$.
\end{proof}

In each case of the classification that follows, we will denote by
\(M\), the set of possible angle torsion orders that occur for that
case. When we want to emphasize the dependence on the prime \(p\) and
the power \(d\), we will denote this by \(M(p,d)\).

\subsection{Simple ordinary threefolds}
\label{sec:simple-ord-threefolds}
In this section, $X$ will denote a simple ordinary abelian threefold defined over $\FF_q$. The corresponding Serre-Frobenius groups are sumarized in \Cref{table:simple-ord-threefolds}.

As a corollary to
\Cref{thm:simple-ordinary-prime-splitting}, we have the following.

\begin{prop}
  Let $X$ be a simple ordinary abelian threefold defined over
  $\FF_q$. Then, exactly one of the following conditions is satisfied.
  \begin{enumerate}
  \item $X$ is absolutely simple.
  \item $X$ splits over a degree $3$ extension and
    $P_X(T) = T^{6}+a_3T^3 + q^3$.
  \item $X$ splits over a degree $7$ extension and the number field of
    $P_X(T)$ is $ \QQ(\zeta_7)$.
  \end{enumerate}
\end{prop}

\begin{lemma}[Node X-A in \Cref{fig:proof-3}]
  \label{lemma:abs-simple-ordinary-threefolds}
  Let $X$ be an absolutely simple abelian threefold defined over
  $\FF_q$. Then $X$ has maximal angle rank $\delta = 3$ and
  $\SF(X) = \UU(1)^3$.
\end{lemma}
\begin{proof}
  Let $m = m_X$ be the order of the torsion subgroup of $\Gamma_X$. By
  \cite[Theorem 1.1]{Zarhin2015EigenFrob}, we have that $X_{(m)}$ is
  neat. Since $X_{(m)}$ is ordinary and simple, its Frobenius
  eigenvalues are distinct and non-real. Remark \ref{rmk:neat2}
  implies that $X_{(m)}$ has maximal angle rank. Since angle rank is
  invariant under base extension
  (\Cref{rmk:delta-invariant-base-extension}) we have that
  $\delta(X) = \delta(X_{(m)}) = 3$ as we wanted to show.
\end{proof}

\begin{lemma}[Node X-B in \Cref{fig:proof-3}]
  \label{lemma:ordinary-simple-threefolds}
  Let $X$ be a simple ordinary abelian threefold over $\FF_q$ that is
  not absolutely simple. Then $X$ has angle rank $1$ and
  \begin{enumerate}[label=(\alph*)]
  \item $\SF(X) \cong \UU(1)\times C_3$ if $X$ splits over a degree
    $3$ extension, or
  \item $\SF(X) \cong \UU(1)\times C_7$ if $X$ splits over a degree
    $7$ extension.
  \end{enumerate}
  Furthermore, in (a) and (b), we have that
  \begin{equation*}
    \SF(X) = \brk{(u, \xi_1^{\nu} u, \xi_2^{\nu}u) : u \in \UU(1), \nu \in \ZZ/m\ZZ},
  \end{equation*}
  with $\xi_1, \xi_{2}$ distinct primitive $m$-th roots of unity, for
  $m = 3$ and $m = 7$ respectively.
\end{lemma}

\begin{proof}
  From the proof of \Cref{thm:simple-ordinary-prime-splitting}, we
  have that the torsion free part of $U_X$ is generated by a fixed
  normalized root $u_1 = \alpha_1/\sqrt{q}$, and all other roots $u_j$
  for $1 < j \le g$ are related to $u_1$ by a primitive root of unity
  of order $3$ or $7$ respectively; $u_2 = \xi_1 u_1$ and
  $u_3 = \xi_2 u_1$ with $\xi_1 \neq \xi_2$.
\end{proof}

\begin{table}[H]
  \setlength{\arrayrulewidth}{0.3mm} \setlength{\tabcolsep}{5pt}
  \renewcommand{\arraystretch}{1.3}
  \caption{Serre--Frobenius groups of simple ordinary threefolds $X$.}
  \begin{longtable}{|c|c|c|c|c|}
    \hline
    \rowcolor{header_color} 
    Splitting type& $\cong$ class& Generator& Example & \Cref{fig:hist-simple-ord-X}\\ \hline
    Absolutely simple & $\UU(1)^3$ & $(u_1,u_2,u_3)$ & 
                                                       \href{https://www.lmfdb.org/Variety/Abelian/Fq/3/2/ad_f_ah}{ \texttt{3.2.ad\_f\_ah}}
                                                      & \ref{fig:headless-number-theorist} \\ \hline
    $X_{(3)} \sim E^3$ & $\UU(1) \times C_3$ & $(u_1, \zeta_3 u_1, \zeta_3^\nu u_1)$ & \href{https://www.lmfdb.org/Variety/Abelian/Fq/3/2/a_a_ad}{\texttt{3.2.a\_a\_ad}} & \ref{fig:cubic-split} \\ \hline
    $X_{(7)} \sim E^3$ & $\UU(1) \times C_7$ & $(u_1, \zeta_7 u_1, \zeta_7^\nu u_1)$ & \href{https://www.lmfdb.org/Variety/Abelian/Fq/3/2/ae_j_ap}{\texttt{3.2.ae\_j\_ap}} & \ref{fig:7-split} \\ \hline
  \end{longtable}
  \addtocounter{table}{-1}
  \label{table:simple-ord-threefolds}
\end{table}

\begin{figure}[H]
  \centering
  \begin{subfigure}{0.3\textwidth}
    \centering
    \includegraphics[width=0.8\textwidth]{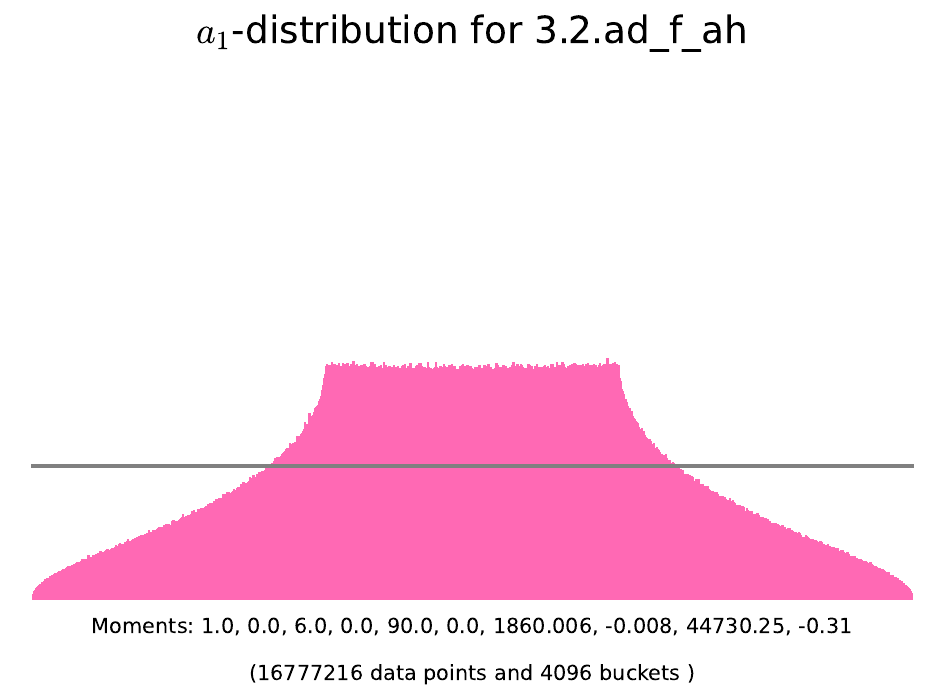}
    \caption{\href{https://www.lmfdb.org/Variety/Abelian/Fq/3/2/ad_f_ah}{\texttt{3.2.ad\_f\_ah}}}
    \label{fig:headless-number-theorist}
  \end{subfigure}
  \hfill
  \begin{subfigure}{0.3\textwidth}
    \centering
    \includegraphics[width=0.8\textwidth]{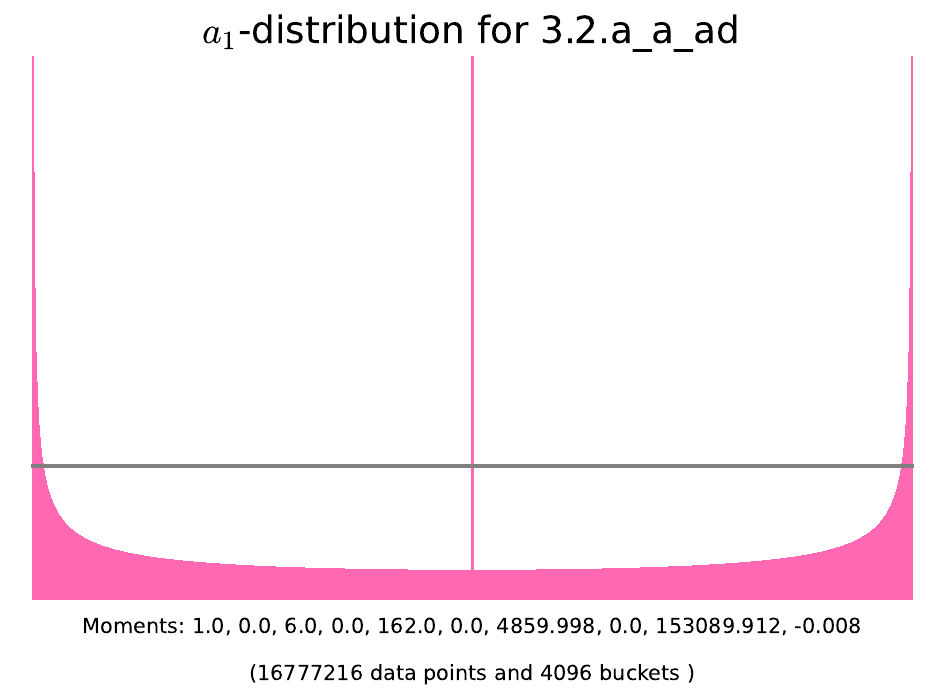}
    \caption{\href{https://www.lmfdb.org/Variety/Abelian/Fq/3/2/a_a_ad}{\texttt{3.2.a\_a\_ad}}}
    \label{fig:cubic-split}
  \end{subfigure}
  \hfill
  \begin{subfigure}{0.3\textwidth}
    \includegraphics[width=0.8\textwidth]{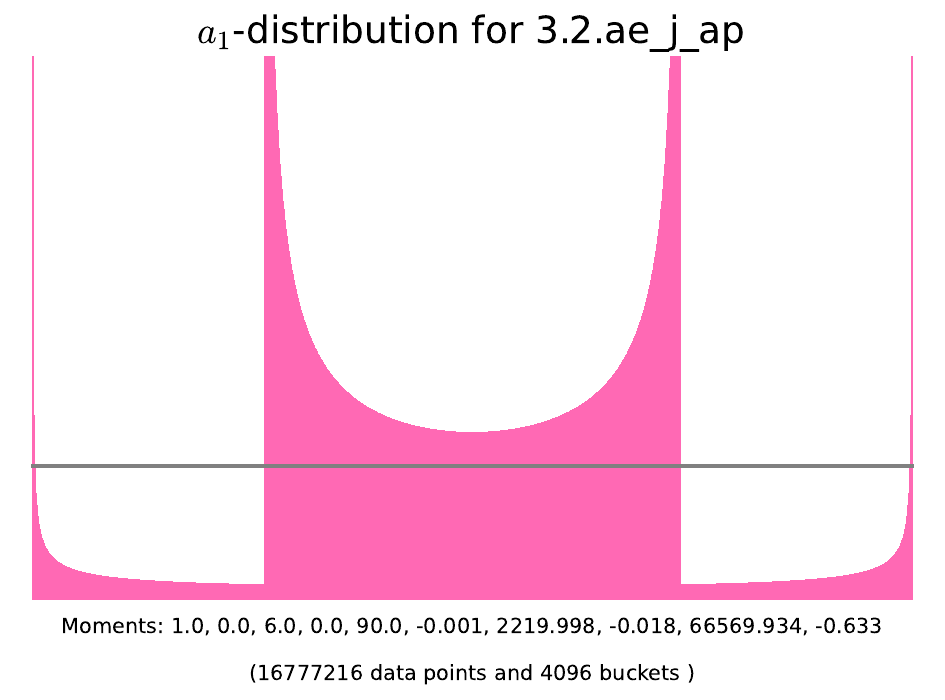}
    \caption{\href{https://www.lmfdb.org/Variety/Abelian/Fq/3/2/ae_j_ap}{\texttt{3.2.ae\_j\_ap}}}
    \label{fig:7-split}
  \end{subfigure}
  \caption{$a_1$-distributions for simple ordinary threefolds.}
  \label{fig:hist-simple-ord-X}
\end{figure}

\subsection{Non-simple ordinary threefolds}
\label{sec:non-sim-ord-threefolds}

Let $X$ be a non-simple ordinary abelian threefold defined over
$\FF_q$. Then $X$ is isogenous to a product $S\times E$, for some
ordinary surface $S$ and some ordinary elliptic curve $E$.

The Frobenius polynomial of $X$ is the product of those of $S$ and
$E$. Further, exactly one of the following is true for $S$: either it
is absolutely simple, or it is simple and geometrically isogenous to
the power of a single elliptic curve, or it is not simple (see
observation after \Cref{lemma:SA}). The Serre--Frobenius group of $X$
depends on its geometric isogeny decomposition, of which there are the
following four possibilities.
\begin{enumerate}[label=(\thesubsection-\alph*), leftmargin = 1.5cm]
\item \label{item:E^3} $X$ is geometrically isogenous to $E^3$.
\item \label{item:E1^2xE} $X$ is geometrically isogenous to
  $E_1^2\times E$, for some ordinary elliptic curve $E_1$, with
  $(E_1)_{\ol{\FF}_q} \not\sim (E)_{\ol{\FF}_q}$.
\item \label{item:E1xE2xE} $X$ is geometrically isogenous to
  $E_1 \times E_2 \times E$, for ordinary and pairwise geometrically
  non-isogenous elliptic curves $E_1, E_2$ and $E$.
\item \label{item:SxE} $X$ is geometrically isogenous to $S\times E$
  for an absolutely simple ordinary surface $S$ and an ordinary
  elliptic curve $E$.
\end{enumerate}

\begin{lemma}[Node X-C in \Cref{fig:proof-3}]
  Let $X$ be a non-simple ordinary abelian threefold over $\FF_q$. The
  possible Serre--Frobenius groups of $X$ are given in
  \Cref{table:non-simple-ord-threefolds}.
\end{lemma}

\begin{table}[H]
  \setlength{\arrayrulewidth}{0.3mm} \setlength{\tabcolsep}{5pt}
  \renewcommand{\arraystretch}{1.2}
  \caption{Serre--Frobenius groups of non-simple ordinary threefolds
    $X = S \times E$.}
  \label{table:non-simple-ord-threefolds}
  \begin{longtable}{|c|c|c|c|c|}
    \hline
    \rowcolor{header_color} 
    Splitting type& $\cong$ class $\SF(X)$ & Generator& $m \in M$ & Examples \\ \hline
    \ref{item:E^3} & $\UU(1)\times C_m$ & $(u_1, \zeta_m u_1, u_1)$ & $\brk{1,2,3,4,6}$ & \Cref{example:ns-ord-X-6.3a}\\ \hline
    \ref{item:E1^2xE} & $\UU(1)^2\times C_m$ & $(u_1, \zeta_m u_1, u_2)$ & $\brk{1,2,3,4,6}$ & \Cref{example:ns-ord-X-6.3b} \\ \hline
    \ref{item:E1xE2xE} & $\UU(1)^3$ & $(u_1,u_2,u_3)$ & $\brk{1}$ & \href{https://www.lmfdb.org/Variety/Abelian/Fq/3/5/ai_bi_ado}{\texttt{3.5.ai\_bi\_ado}} \\ \hline
    \ref{item:SxE} & $\UU(1)^3$ & $(u_1,u_2,u_3)$ & $\brk{1}$ & \href{https://www.lmfdb.org/Variety/Abelian/Fq/3/2/ad_h_al}{\texttt{3.2.ad\_h\_al}}
    \\ \hline
  \end{longtable}
  \addtocounter{table}{-1}
\end{table}

\begin{proof} Recall that $X \sim S\times E$ over $\FF_q$. \\
  \ref{item:E^3} If $X$ is geometrically isogenous to $E^3$, then $S$ is
  geometrically isogenous to $E^2$. By \Cref{lemma:B-geom-EC}
  $\SF(X) \cong \UU(1) \times C_m$, where $m$ is the splitting degree of
  $S$, and so $S_{(m)} \sim E^2$. By \cite[Theorem 6]{howe2002existence},
  we have that $m \in \{1, 2, 3, 4,6\}$. \\

  \ref{item:E1^2xE} In this case, by \Cref{lemma:B-geom-EC},
  $\SF(X) \cong \UU(1)^2 \times C_m$, where $m$ is the splitting
  degree of $S$ and $S_{(m)} \sim E_1^2$. As in the previous case,
  $m \in \{1,2,3,4,6\}$.\\

  \ref{item:E1xE2xE} In this case $S \sim E_1\times E_2$ over the base field.
  \Cref{lemma:poonen} implies $\delta_X = 3$. \\

  \ref{item:SxE} In this case, $X \sim S \times E$ with $S$ absolutely simple.
  By \cite[Theorem 1.1]{Zarhin2015EigenFrob}, we know that $X$ is neat. Since
  $X$ is ordinary and $S$ is simple, all Frobenius eigenvalues are distinct and
  not supersingular. By Remark \ref{rmk:neat2}, we conclude that
  $\delta_X = 3$.
\end{proof}

\begin{example}[Non-simple ordinary threefolds of splitting type
  \ref{item:E^3}]
  \label{example:ns-ord-X-6.3a}
  \hfill
  \begin{figure}[H]
    \centering
    \begin{subfigure}{0.32\textwidth}
      \centering
      \includegraphics[width=0.8\textwidth]{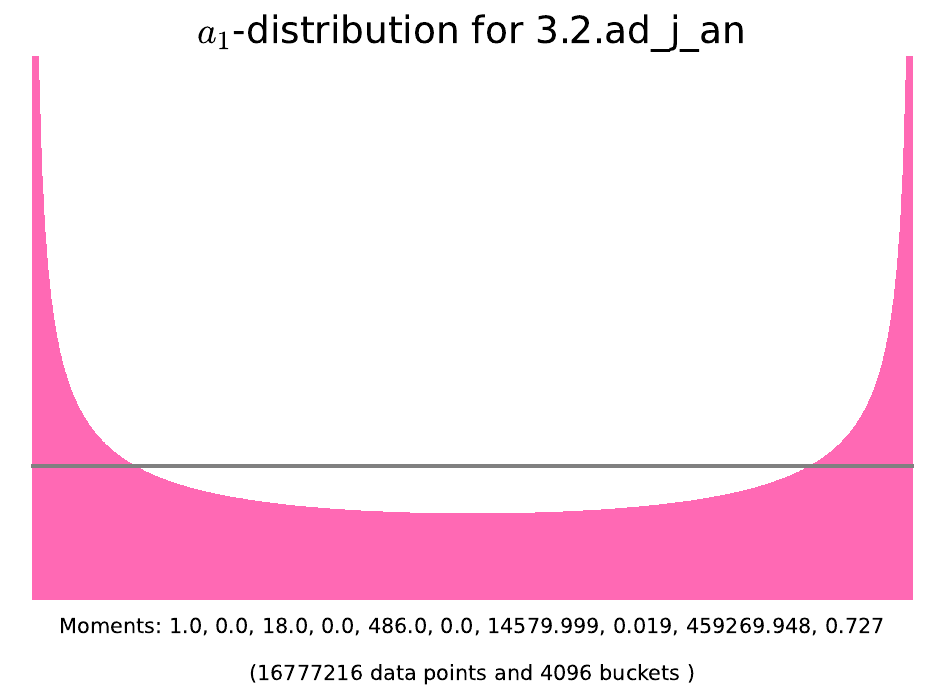}
      \caption{
        \href{https://www.lmfdb.org/Variety/Abelian/Fq/3/2/ad_j_an}{\texttt{3.2.ad\_j\_an}}}
    \end{subfigure}
    \hfill
    \begin{subfigure}{0.32\textwidth}
      \centering
      \includegraphics[width=0.8\textwidth]{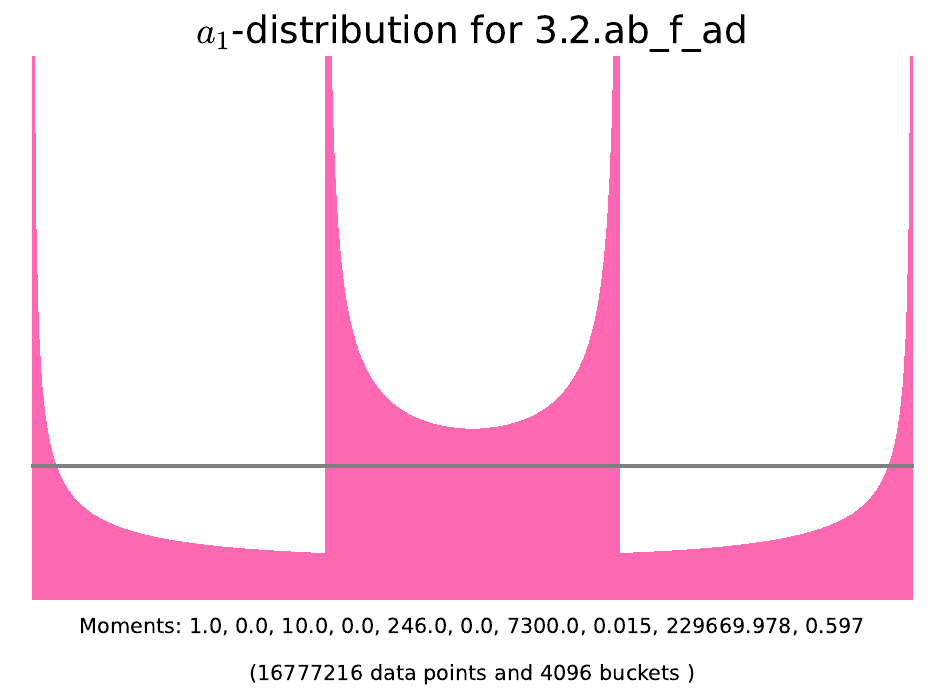}
      \caption{
        \href{https://www.lmfdb.org/Variety/Abelian/Fq/3/2/ab_f_ad}{\texttt{3.2.ab\_f\_ad}}}
    \end{subfigure}
    \hfill
    \begin{subfigure}{0.32\textwidth}
      \centering
      \includegraphics[width=0.8\textwidth]{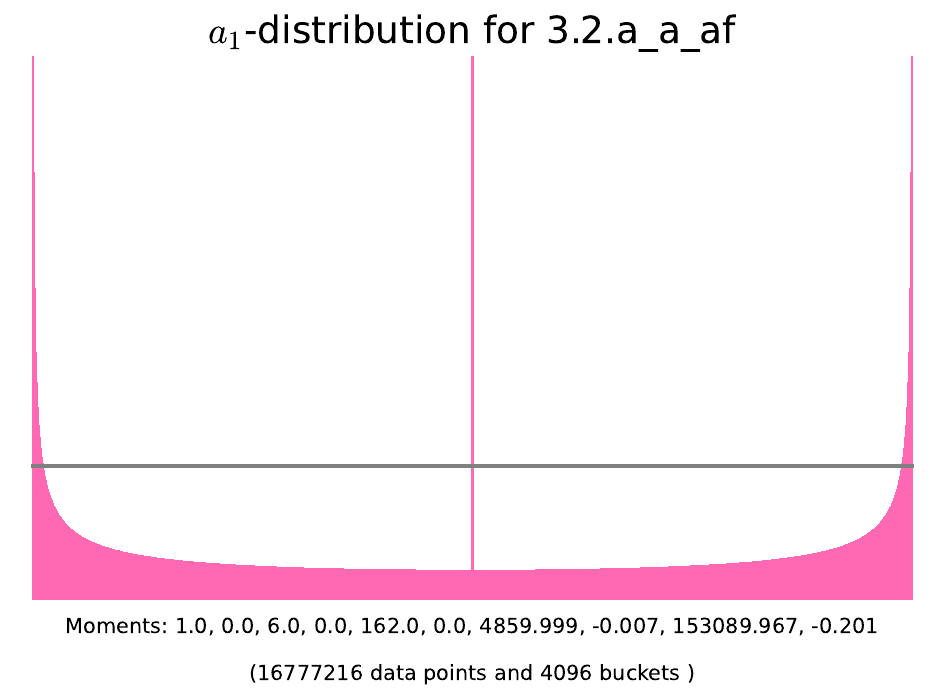}
      \caption{
        \href{https://www.lmfdb.org/Variety/Abelian/Fq/3/2/a_a_af}{\texttt{3.2.a\_a\_af}}}
    \end{subfigure}
    \hfill
    \begin{subfigure}{0.32\textwidth}
      \centering
      \includegraphics[width=0.8\textwidth]{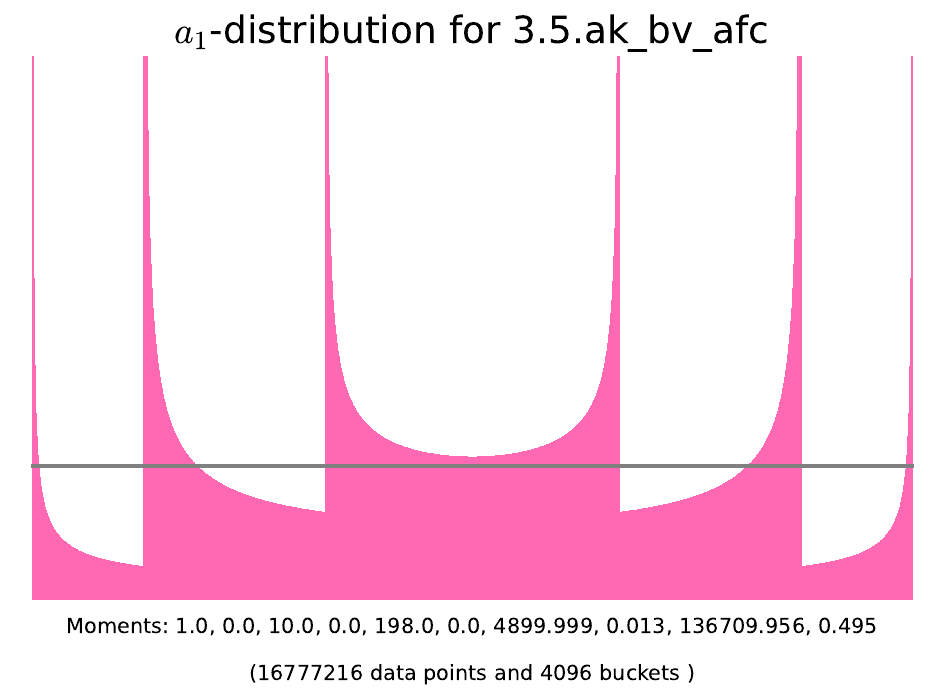}
      \caption{
        \href{https://www.lmfdb.org/Variety/Abelian/Fq/3/5/ak_bv_afc}{\texttt{3.5.ak\_bv\_afc}}}
    \end{subfigure}
    \hspace{2cm}
    \begin{subfigure}{0.32\textwidth}
      \centering
      \includegraphics[width=0.8\textwidth]{images/a1_2.7.aj_bi_16_6.pdf}
      \caption{
        \href{https://www.lmfdb.org/Variety/Abelian/Fq/3/7/ao_di_alk}{\texttt{3.7.ao\_di\_alk}}}
    \end{subfigure}
    \caption{$a_1$-distributions for non-simple ordinary abelian
      threefolds of splitting type \ref{item:E^3}.}
    \label{fig:non-simple-ord-threefolds-i}
  \end{figure}
\end{example}
  
\begin{example}[Non-simple ordinary threefolds of splitting type
  \ref{item:E1^2xE}] \hfill
  \label{example:ns-ord-X-6.3b}
  \hfill

\begin{figure}[H]
  \centering
  \begin{subfigure}{0.32\textwidth}
    \centering
    \includegraphics[width=0.8\textwidth]{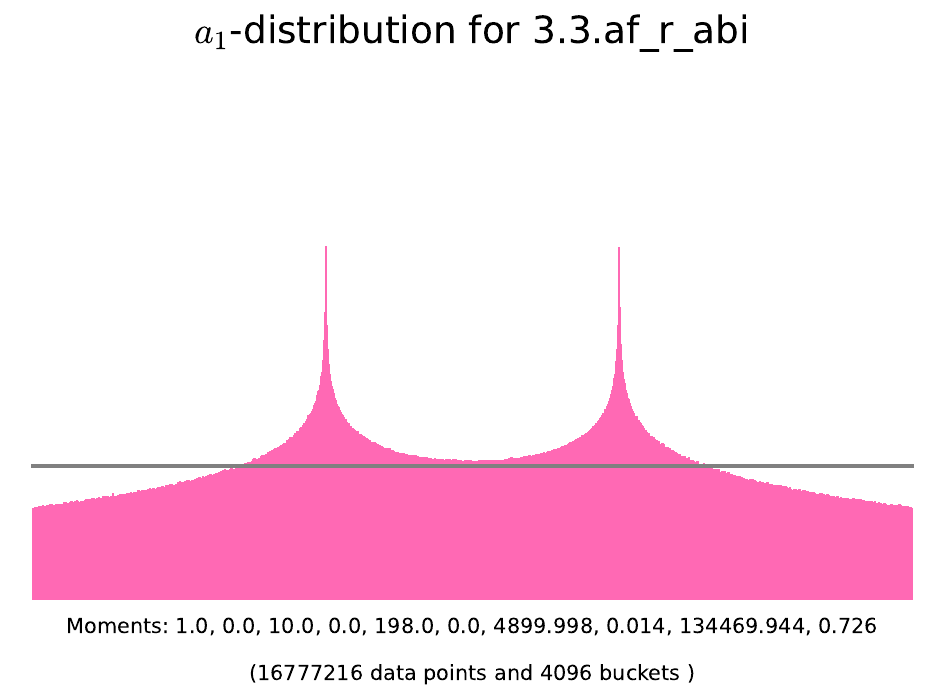}
    \caption{
      \href{https://www.lmfdb.org/Variety/Abelian/Fq/3/3/af_r_abi}{\texttt{3.3.af\_r\_abi}}}
  \end{subfigure}
  \hfill
  \begin{subfigure}{0.32\textwidth}
    \centering
    \includegraphics[width=0.8\textwidth]{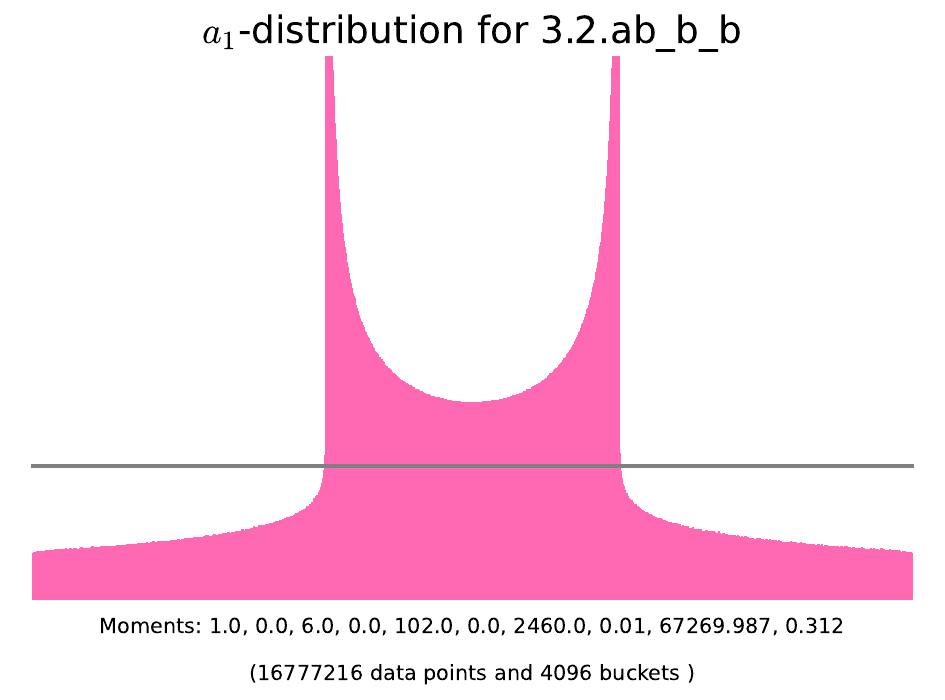}
    \caption{\href{https://www.lmfdb.org/Variety/Abelian/Fq/3/2/ab_b_b}{\texttt{3.2.ab\_b\_b}}}
  \end{subfigure}
  \hfill
  \begin{subfigure}{0.32\textwidth}
    \centering
    \includegraphics[width=0.8\textwidth]{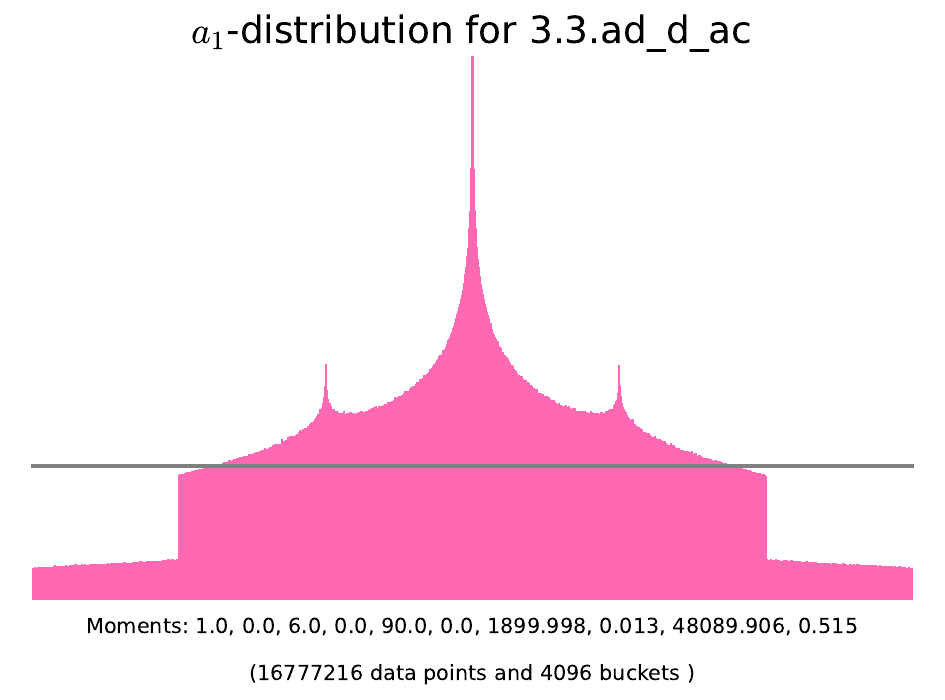}
    \caption{\href{https://www.lmfdb.org/Variety/Abelian/Fq/3/3/ad_d_ac}{\texttt{3.3.ad\_d\_ac}}}
  \end{subfigure}
  \hfill
  \begin{subfigure}{0.32\textwidth}
    \centering
    \includegraphics[width=0.8\textwidth]{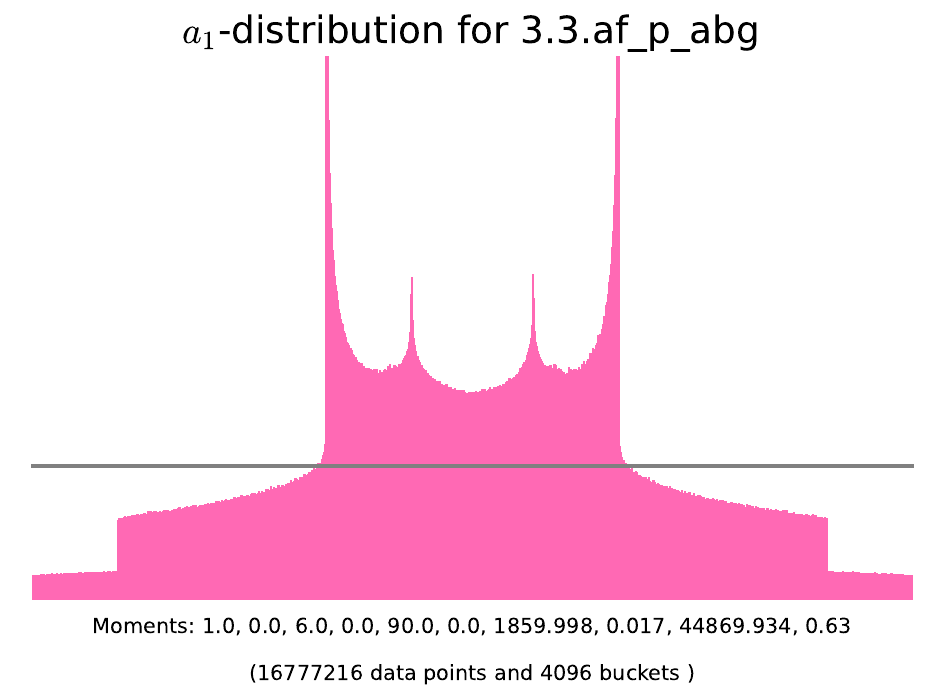}
    \caption{
      \href{https://www.lmfdb.org/Variety/Abelian/Fq/3/3/af_p_abg}{\texttt{3.3.af\_p\_abg}}}
  \end{subfigure}
  \hspace{2cm}
  \begin{subfigure}{0.32\textwidth}
    \centering
    \includegraphics[width=0.8\textwidth]{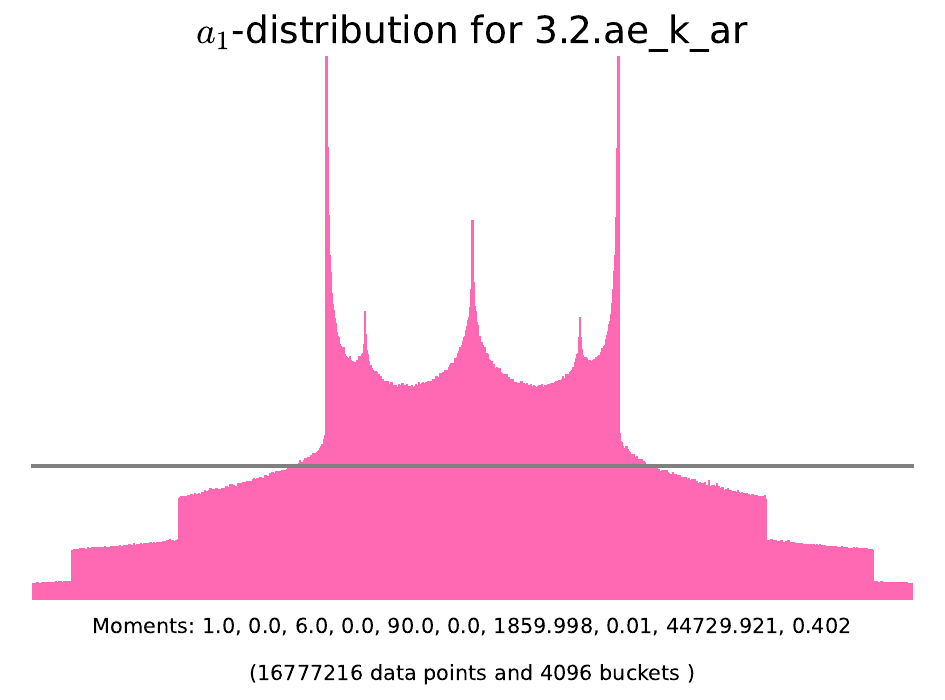}
    \caption{\href{https://www.lmfdb.org/Variety/Abelian/Fq/3/2/ae_k_ar}{\texttt{3.2.ae\_k\_ar}}}
  \end{subfigure}
  \caption{$a_1$-distributions for non-simple ordinary abelian
    threefolds of splitting type \ref{item:E1^2xE}.}
  \label{fig:non-simple-ord-threefolds-ii}
\end{figure}
\end{example}

\subsection{Simple almost ordinary threefolds}
\label{sec:simple-ao-threefolds} Let $X$ be a simple and almost
ordinary abelian threefold over $\FF_q$. Recall that $X$ is in fact
absolutely simple, so that the Frobenius polynomial $P_{(r)}(T)$ is
irreducible for every positive integer $r$.

\begin{lemma}[Node X-D in \Cref{fig:proof-3}]
  Let $X$ be a simple almost ordinary abelian threefold over
  $\FF_q$. The Serre--Frobenius group of $X$ can be read from
  \Cref{table:simple-ao-threefolds}.
\end{lemma}

\begin{table}[H]
  \caption{Serre--Frobenius groups of simple almost ordinary
    threefolds.}
  \setlength{\arrayrulewidth}{0.3mm} \setlength{\tabcolsep}{5pt}
  \renewcommand{\arraystretch}{1.2}
  \begin{longtable}{|c|c|c|c|c|c|}
    \hline
    \rowcolor{header_color} 
    Def. \ref{def:neat} & $\sqrt{q} \in \QQ(\pi_X) $ & $\cong$ class & Generator& $m \in M$ & Example \\ \hline
    Neat & - & $\UU(1)^3$ & $(u_1,u_2,u_3)$ & $\{$1$\}$ & \href{https://www.lmfdb.org/Variety/Abelian/Fq/3/2/ab_ab_c}{\texttt{3.2.ab\_ab\_c}}\\ \hline
    Not neat & Yes & $\UU(1)^2\times C_m$ & $(u_1,u_2,\zeta_m u_1u_2)$ & $\{$1, 2, 3, 4, 6$\}$ & \Cref{ex:simple-ao-X-rank-2}\\ \hline
    Not neat & No & $\UU(1)^2\times C_m$ & $(u_1,u_2,\zeta_m u_1u_2)$ & $\{$4, 6, 8, 12$\}$ & \Cref{ex:simple-ao-X-rank-2} \\ \hline
  \end{longtable}
  \addtocounter{table}{-1}
  \label{table:simple-ao-threefolds}
\end{table}

\begin{proof}
  Let $m \colonequals m_X$ be the torsion order of $U_X$, and consider
  the base extension $Y \colonequals X_{(m)}$. By \cite[Theorem
  5.7]{lenstra1993}, we know that $\delta_X = \delta_Y \geq
  2$. Furthermore, since $Y$ is absolutely simple, by the discussion
  in \Cref{subsec:simple-threefolds}, the roots of
  $P_Y(T) = P_{(m)}(T)$ are distinct and non-supersingular. If $Y$ is
  neat, Remark \ref{rmk:neat2} implies that $\delta_X = \delta_Y =
  3$. Assume then that $Y$ is not neat, so that $\delta_X = 2$. Let
  $\alpha = \alpha_1$ be a Frobenius eigenvalue of $X$. By
  \cite[Theorem 1.1]{Zarhin2015EigenFrob} and the discussion
  thereafter, we have that the sextic CM-field
  $\QQ(\alpha) = \QQ(\alpha^m)$ contains an imaginary quadratic field
  $B$, and
  $(u_1u_2u_3)^{2m} = \mathrm{Norm}_{\QQ(\alpha)/B}(u_1^{2m}) =
  1$. Further,
  \(\mathrm{Norm}_{\QQ(\alpha)/B}(\alpha_1) =
  \alpha_1\alpha_2\alpha_3\). Since $U_Y$ has no torsion, this implies
  that $(u_1u_2u_3)^m = 1$. Moreover, this means that
  $u_1u_2u_3 = \zeta$ for some primitive\footnote{The primitivity of
    $\zeta$ follows from the fact that $m$ is the minimal positive
    integer such that $U_Y = U_{(m)}$ is torsion free.} $m$-th root of
  unity $\zeta$. Therefore,
  \begin{equation}
    \zeta^2  = \mathrm{Norm}_{\QQ(\alpha)/B}(u_1^2) \in B.
  \end{equation}
  If $m$ is odd, $\zeta^2$ is also primitive, so that
  $\varphi(m) \leq 2$ and $m \in \brk{1,3}$. If $m$ is even, then we
  may distinguish between two cases. If $\sqrt{q} \in \QQ(\alpha)$, we
  know that $u_1 \in \QQ(\alpha)$ so that in fact
  $\pm \zeta = \mathrm{Norm}_{\QQ(\alpha)/B}(u_1) \in B$ and
  $\varphi(m)\leq 2$ implies that $m \in \brk{2,4,6}$. If
  $\sqrt{q}\not\in\QQ(\alpha)$, then $\zeta^2$ is a primitive
  $m/2$-root of unity and $m/2 \in \brk{1,2,3,4,6}$.
   
  In the setting where \(Y\) is not neat and
  \(\sqrt{q} \notin \QQ(\alpha)\), we notice that $u_1u_2u_3 = \pm 1$
  implies that
  $\sqrt{q} = \pm \alpha_1\alpha_2\alpha_3/q \in \QQ(\alpha)$, so the
  cases $m = 1,2$ don't occur when $\sqrt{q} \not\in \QQ(\alpha)$.
  Similarly, if $u_1u_2u_3 = \zeta_3$, then
  $\sqrt{q} = (\alpha_1\alpha_2\alpha_3)^3/q^4 \in \QQ(\alpha)$. Thus,
  the torsion orders \(m = 1,2,3\) do not occur in this case.
\end{proof}

\begin{example}[$a_1$-distributions of simple almost ordinary abelian
  threefolds with angle rank $2$] \label{ex:simple-ao-X-rank-2} The
  histograms corresponding to the following examples are presented in
  \Cref{fig:simple-almost-ord-threefolds-rank-2}. In these examples we
  use \sage \cite{sagemath} to initialize the degree-6 number field
  $K = \QQ(\alpha)$ corresponding to the Frobenius polynomial, find
  the corresponding quadratic subfield $B$, and check that
  $\mathrm{Norm}_{\QQ(\alpha)/B}(u_1)$ is the root of unity in
  question. The code for generating the histograms is available on the
  \github repository \cite{Frob_dist_code}.

\begin{figure}[H]
  \centering
  \begin{subfigure}{0.24\textwidth}
    \centering
    \includegraphics[width=\textwidth]{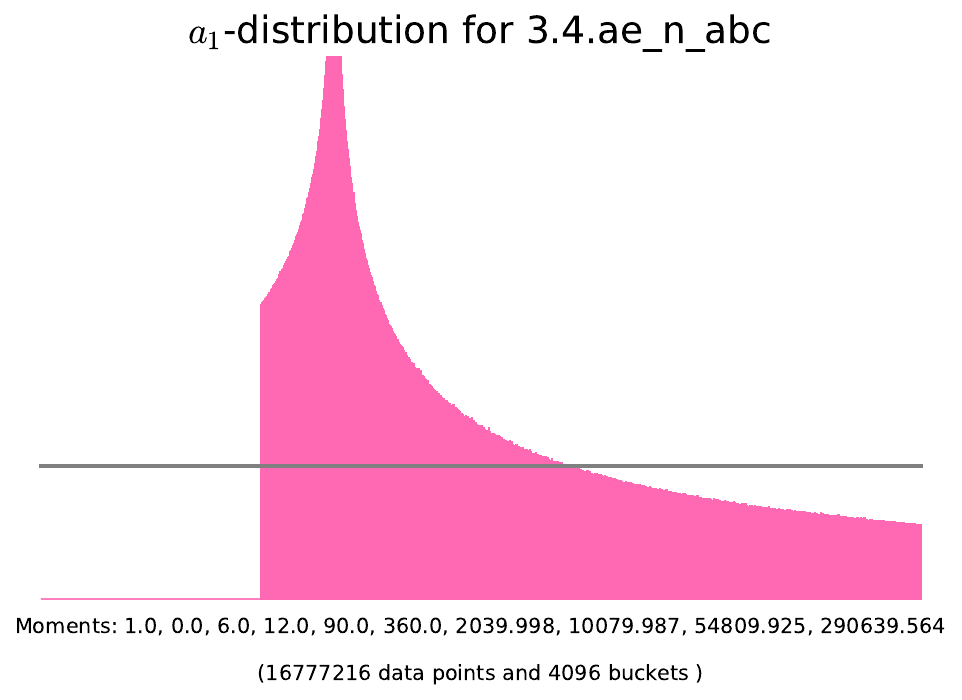}
    \caption{
      \href{https://www.lmfdb.org/Variety/Abelian/Fq/3/4/ae_n_abc}{\texttt{3.4.ae\_n\_abc}}}
  \end{subfigure}
  \hfill
  \begin{subfigure}{0.24\textwidth}
    \centering
    \includegraphics[width=\textwidth]{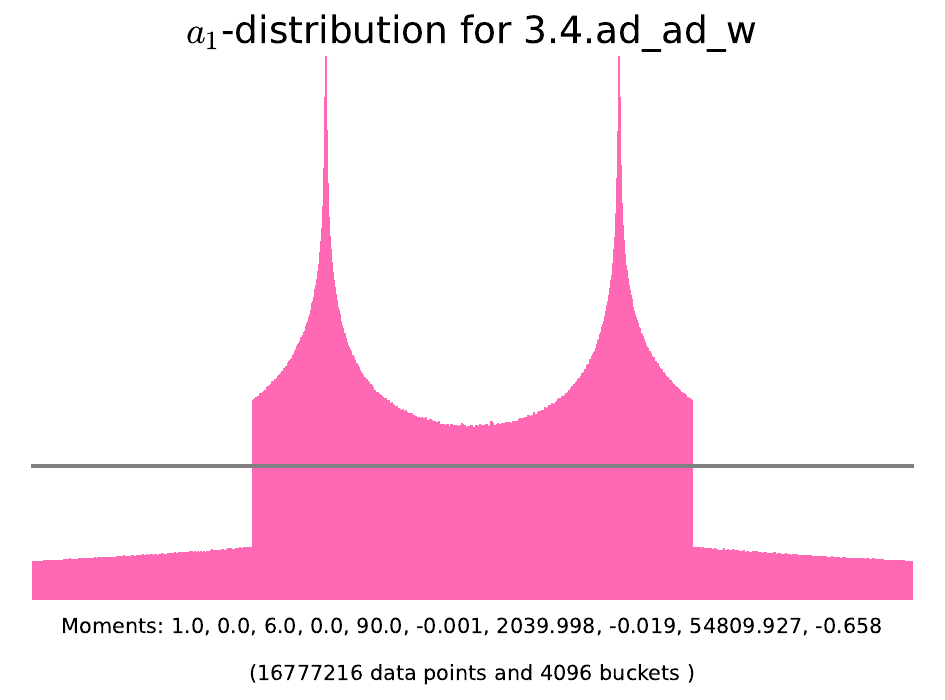}
    \caption{\href{https://www.lmfdb.org/Variety/Abelian/Fq/3/4/ad_ad_w}{\texttt{3.4.ad\_ad\_w}}}
  \end{subfigure}
  \hfill
  \begin{subfigure}{0.24\textwidth}
    \centering
    \includegraphics[width=\textwidth]{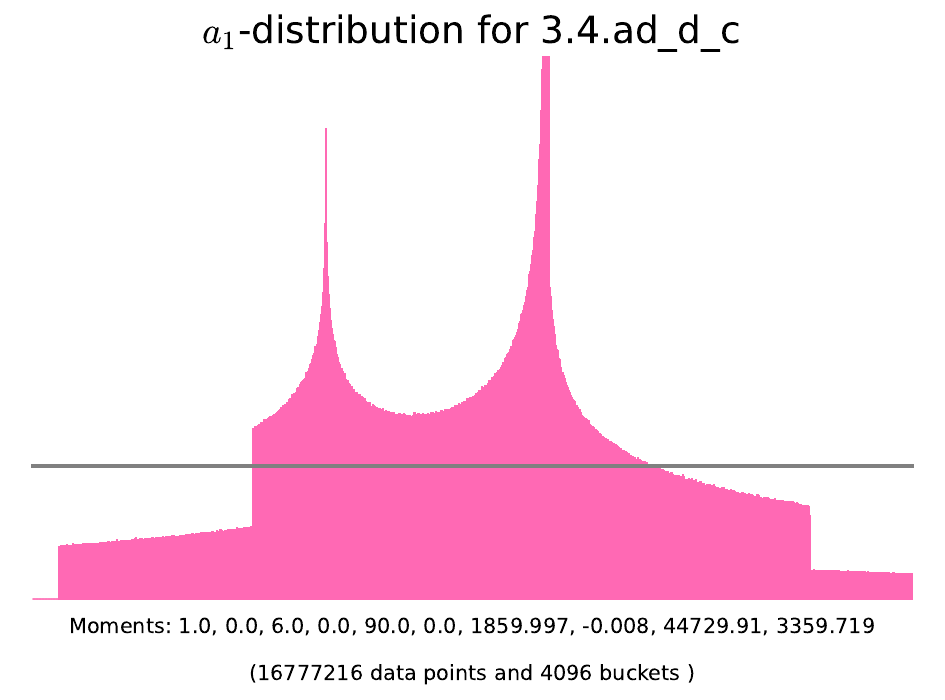}
    \caption{\href{https://www.lmfdb.org/Variety/Abelian/Fq/3/4/ad_d_c}{\texttt{3.4.ad\_d\_c}}}
  \end{subfigure}
  \hfill
  \begin{subfigure}{0.24\textwidth}
    \centering
    \includegraphics[width=\textwidth]{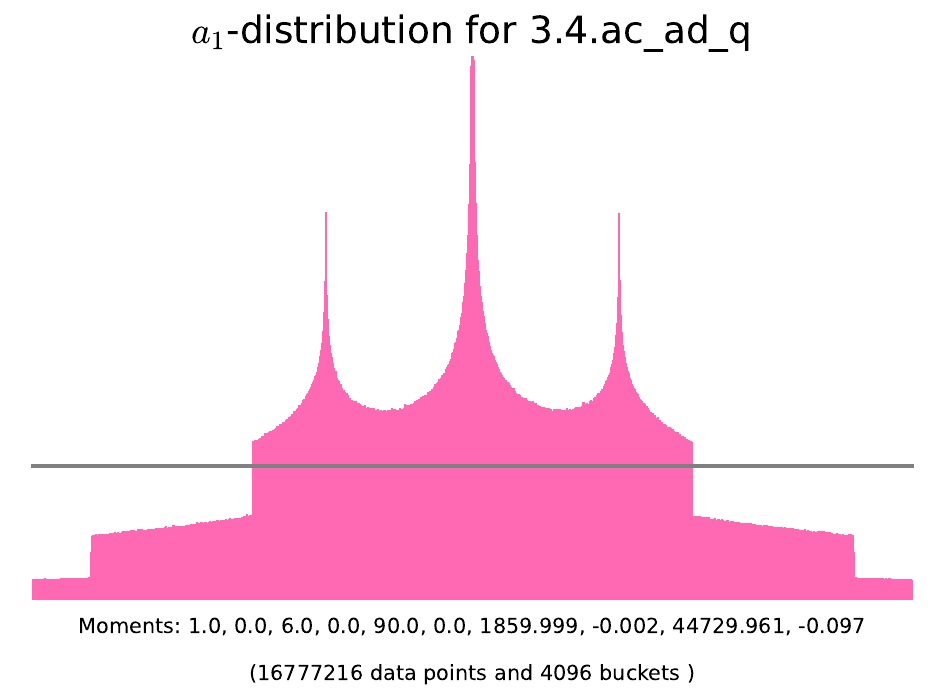}
    \caption{
      \href{https://www.lmfdb.org/Variety/Abelian/Fq/3/4/ac_ad_q}{\texttt{3.4.ac\_ad\_q}}}
  \end{subfigure}
  \hfill
  \begin{subfigure}{0.32\textwidth}
    \centering
    \includegraphics[width=0.8\textwidth]{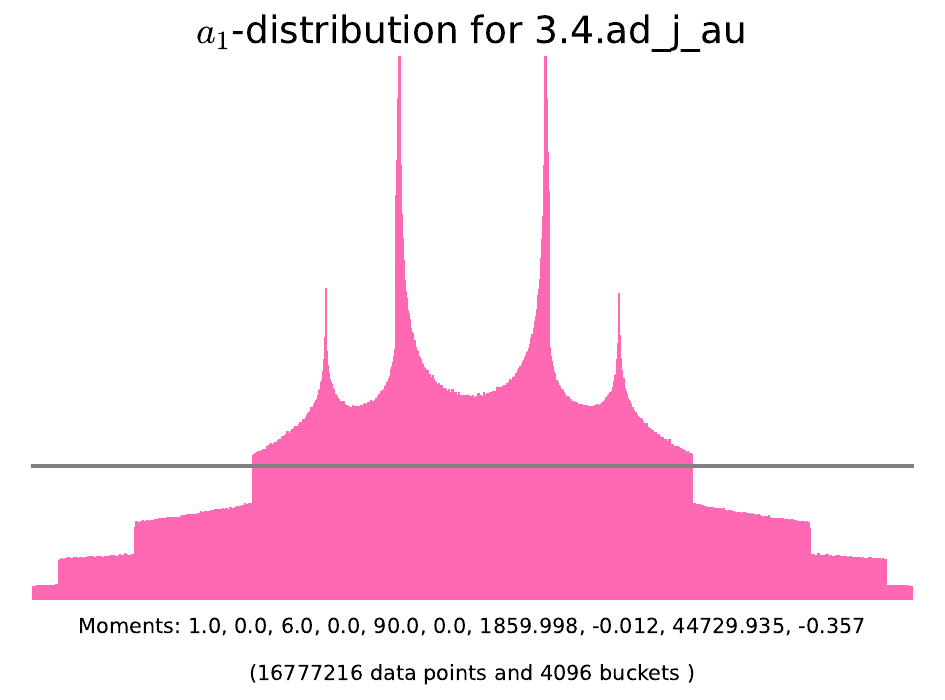}
    \caption{\href{https://www.lmfdb.org/Variety/Abelian/Fq/3/4/ad_j_au}{\texttt{3.4.ad\_j\_au}}}
  \end{subfigure}
  \hfill
  \begin{subfigure}{0.32\textwidth}
    \centering
    \includegraphics[width=0.8\textwidth]{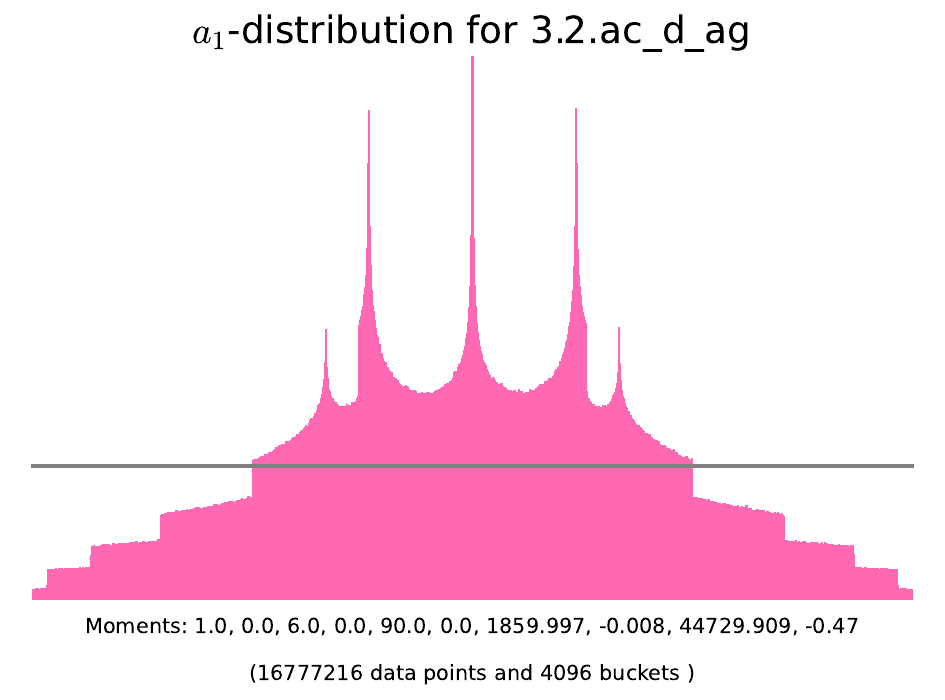}
    \caption{\href{https://www.lmfdb.org/Variety/Abelian/Fq/3/2/ac_d_ag}{\texttt{3.2.ac\_d\_ag}}}
  \end{subfigure}
  \hfill
  \begin{subfigure}{0.32\textwidth}
    \centering
    \includegraphics[width=0.8\textwidth]{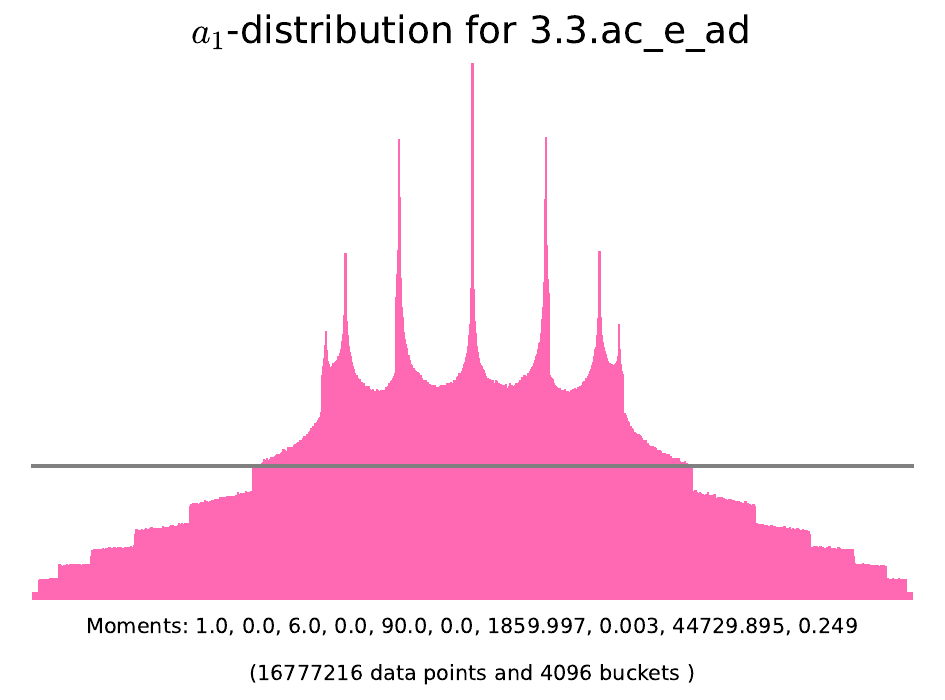}
    \caption{\href{https://www.lmfdb.org/Variety/Abelian/Fq/3/3/ac_e_ad}{\texttt{3.3.ac\_e\_ad}}}
  \end{subfigure}
  \caption{$a_1$-distributions of simple almost ordinary abelian
    threefolds of angle rank $2$.}
  \label{fig:simple-almost-ord-threefolds-rank-2}
\end{figure}
\end{example}

\subsection{Non-simple almost ordinary threefolds}
\label{sec:non-simple-ao-threefolds}
Since $X$ is not simple, we have that $X\sim S\times E$ for some
surface $S$ and some elliptic curve $E$. For this section, we let
$\pi_1, \overline{\pi}_1, \pi_2, \overline{\pi}_2$ and
$\alpha, \overline{\alpha}$ be the Frobenius eigenvalues of $S$ and
$E$ respectively. The normalized eigenvalues will be denoted by
$u_1 \colonequals \pi_1/\sqrt{q}, u_2 = \pi_2/\sqrt{q}$ and
$u \colonequals \alpha/\sqrt{q}$.  If $X$ has a geometric
supersingular factor, by Honda--Tate theory, it must have a
supersingular factor over the base field; and without loss of
generality we may assume that this factor is $E$.

\begin{figure}[H]
  \centering \includegraphics[width=0.5\textwidth]{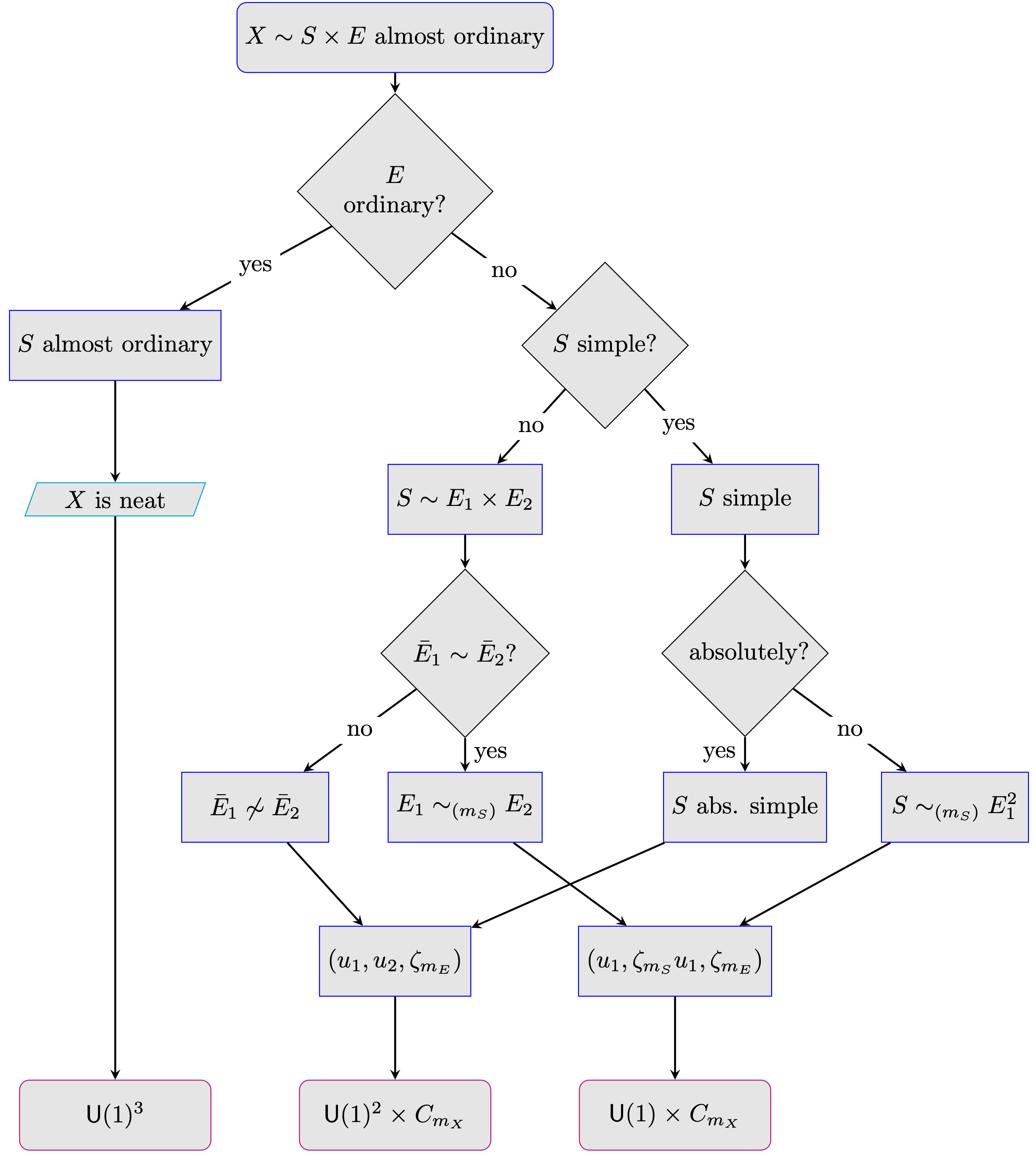}
  \caption{Serre--Frobenius groups of non-simple almost ordinary
    threefolds.}
  \label{flow:ns-ao-X}
\end{figure}

\begin{lemma}[Node X-E in \Cref{fig:proof-3}]
  Let $X \sim S\times E$ be a non-simple almost ordinary abelian
  threefold over $\FF_q$. The Serre--Frobenius group of $X$ can be
  read from Flowchart \ref{flow:ns-ao-X}. In particular, if $X$ has no
  supersingular factor, then $\delta_X = 3$. If $E$ is supersingular,
  then $\delta_X \in \brk{1,2}$ and $m_X = \lcm(m_S, m_E) $.  The list
  of possible torsion orders $m_X$ in this case is given
  \Cref{table:p-rank-2-threefolds}.
\end{lemma}
\begin{proof}
  First, suppose that $X$ has no supersingular factor. Thus $E$ is
  ordinary and $S$ is almost ordinary and absolutely simple. This
  implies that $\QQ(\pi_1^r)$ and $\QQ(\alpha^r)$ are CM-fields of
  degrees $4$ and $2$ respectively, for every positive integer $r$. In
  particular,
  $\#\brk{\pi_1^r, \overline{\pi}_1^r, \pi_2^r,
    \overline{\pi}_2^r,\alpha^r, \overline{\alpha}^r}$ $= 6$ for every
  $r$. Let $m = m_X$ and consider the base extension $X_{(m)}$. Since
  $X_{(m)}$ is not simple, \cite[Theorem 1.1]{Zarhin2015EigenFrob}
  implies that $X_{(m)}$ is neat. The eigenvalues of $X_{(m)}$ are all
  distinct and not supersingular, so that
  $\delta(X) = \delta(X_{(m)}) = 3$ by Remark \ref{rmk:neat2}. The
  case where $X$ has a supersingular factor follows from
  \Cref{lemma:A1xB}.
\end{proof}

\begin{table}[H]
  \caption{Serre--Frobenius groups of non-simple almost ordinary
    threefolds $X = S\times E$.}
  \setlength{\arrayrulewidth}{0.3mm} \setlength{\tabcolsep}{5pt}
  \renewcommand{\arraystretch}{1.2}
  \begin{longtable}{|c|c|c|c|c|c|}
    \hline
    \rowcolor{header_color} 
    \(\delta_E\) & $\cong$ class & Generator & \(d = \log_p(q)\) & $m \in M(p,d)$ & Example \\ \hline
    $1$ & $\UU(1)^3$ & $(u_1,u_2,u_3)$ & - & $\brk{1}$ & \href{https://www.lmfdb.org/Variety/Abelian/Fq/3/2/ac_d_ae}{\texttt{3.3.ac\_d\_ae}} \\ \hline
    $0$ & $\UU(1)^2\times C_m$ & $(u_1,u_2,\zeta_{m_E})$ & - & $m= m_E \in \{1,2,3,4,6,8,12\}$ & \Cref{fig:non-simple-almost-ord-threefolds-rank-2} \\ \hline
    $0$ & $\UU(1)\times C_m$ & $(u_1,\zeta_{m_S}u_1,\zeta_{m_E})$ & even & $m = \lcm(m_S, m_E)\in \brk{1,2,3,4,6, 12}$ & \Cref{fig:non-simple-almost-ord-threefolds-rank-1} \\ \hline
    $0$ & $\UU(1)\times C_m$ & $(u_1,\zeta_{m_S}u_1,\zeta_{m_E})$ & odd & $m = \lcm(m_S, m_E)\in \brk{4, 8, 12, 24}$ & \Cref{fig:non-simple-almost-ord-threefolds-rank-1} \\ \hline
  \end{longtable}
  \addtocounter{table}{-1}
  \label{table:p-rank-2-threefolds}
\end{table}
    
\begin{figure}[H]
  \centering
  \begin{subfigure}{0.24\textwidth}
    \centering
    \includegraphics[width=\textwidth]{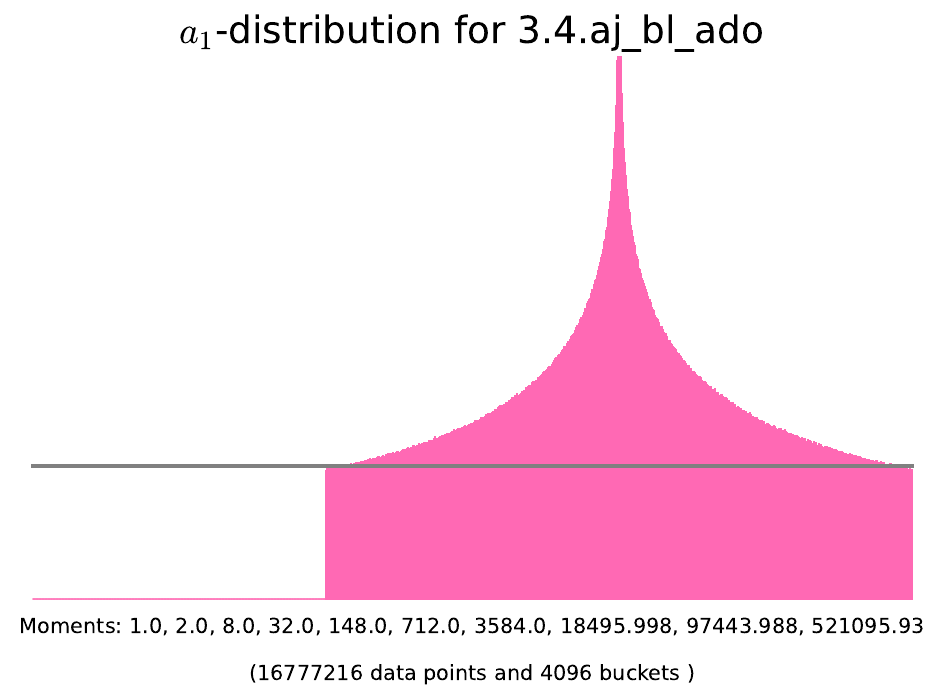}
    \caption{
      \href{https://www.lmfdb.org/Variety/Abelian/Fq/3/4/aj_bl_ado}{\texttt{3.4.aj\_bl\_ado}}}
  \end{subfigure}
  \hfill
  \begin{subfigure}{0.24\textwidth}
    \centering
    \includegraphics[width=\textwidth]{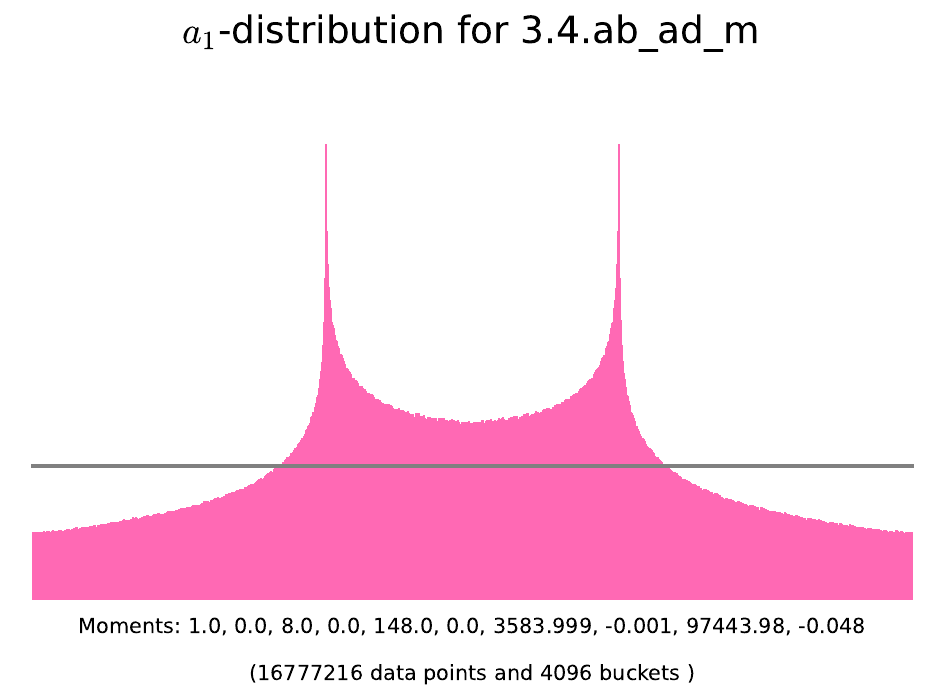}
    \caption{\href{https://www.lmfdb.org/Variety/Abelian/Fq/3/4/ab_ad_m}{\texttt{3.4.ab\_ad\_m}}}
  \end{subfigure}
  \hfill
  \begin{subfigure}{0.24\textwidth}
    \centering
    \includegraphics[width=\textwidth]{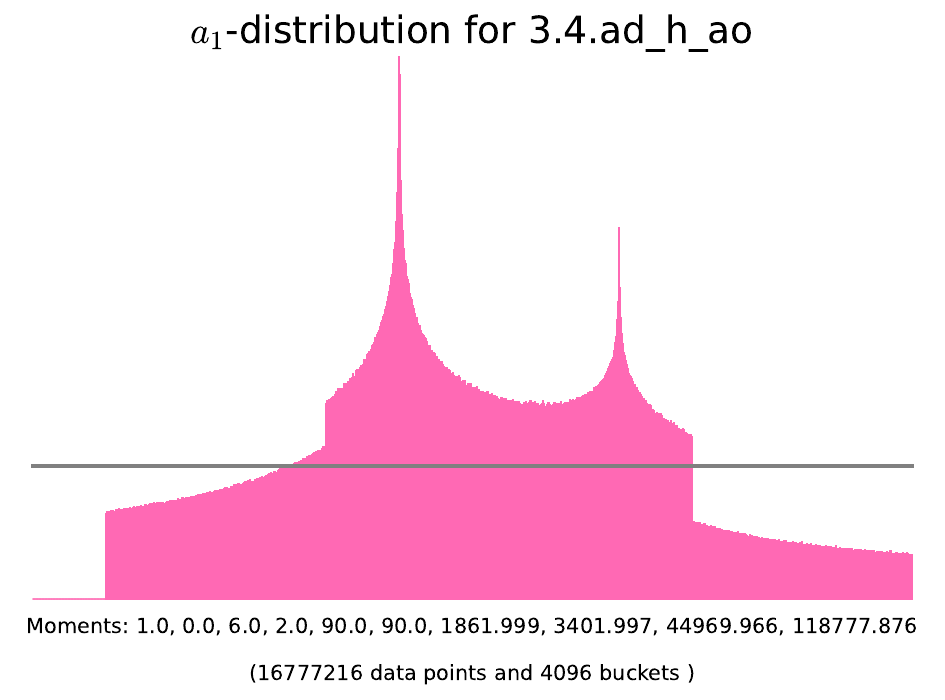}
    \caption{\href{https://www.lmfdb.org/Variety/Abelian/Fq/3/4/ad_h_ao}{\texttt{3.4.ad\_h\_ao}}}
  \end{subfigure}
  \hfill
  \begin{subfigure}{0.24\textwidth}
    \centering
    \includegraphics[width=\textwidth]{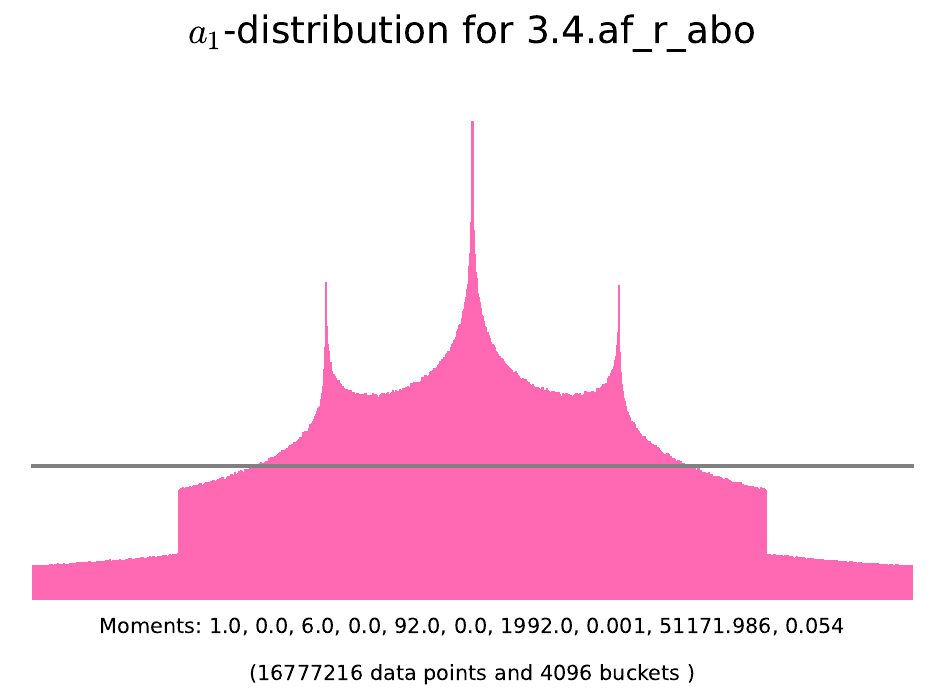}
    \caption{
      \href{https://www.lmfdb.org/Variety/Abelian/Fq/3/4/af_r_abo}{\texttt{3.4.af\_r\_abo}}}
  \end{subfigure}
  \hfill
  \begin{subfigure}{0.32\textwidth}
    \centering
    \includegraphics[width=0.8\textwidth]{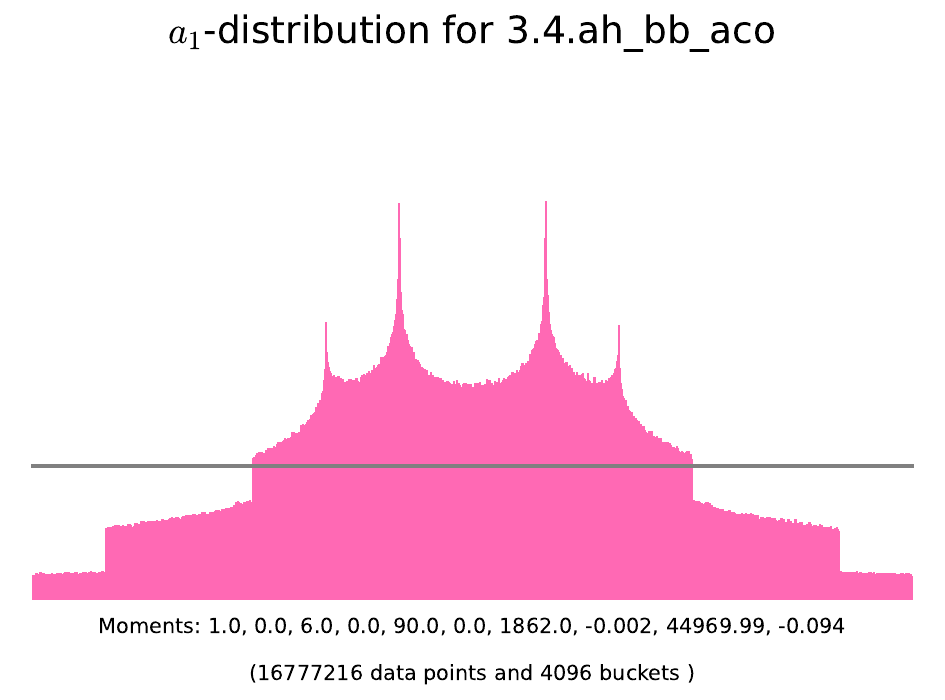}
    \caption{\href{https://www.lmfdb.org/Variety/Abelian/Fq/3/4/ah_bb_aco}{\texttt{3.4.ah\_bb\_aco}}}
  \end{subfigure}
  \hfill
  \begin{subfigure}{0.32\textwidth}
    \centering
    \includegraphics[width=0.8\textwidth]{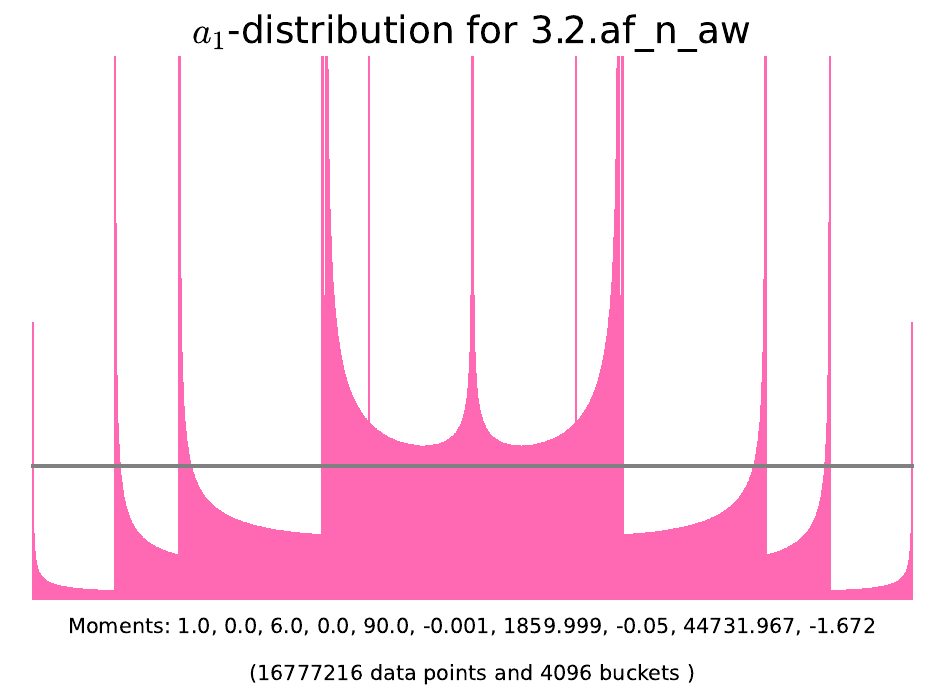}
    \caption{\href{https://www.lmfdb.org/Variety/Abelian/Fq/3/2/af_n_aw}{\texttt{3.2.af\_n\_aw}}}
  \end{subfigure}
  \hfill
  \begin{subfigure}{0.32\textwidth}
    \centering
    \includegraphics[width=0.8\textwidth]{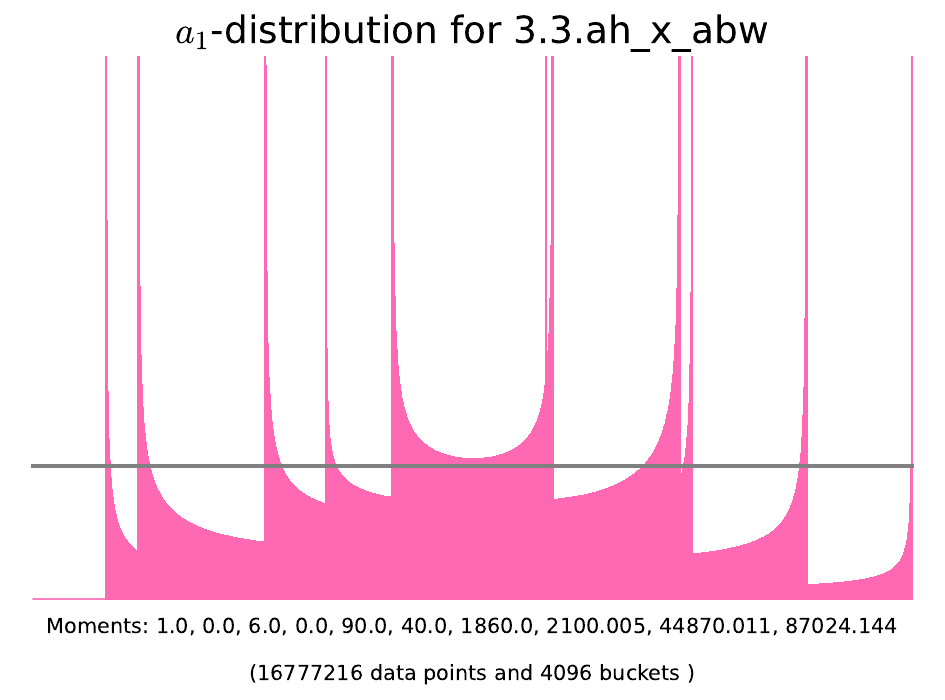}
    \caption{\href{https://www.lmfdb.org/Variety/Abelian/Fq/3/3/ah_x_abw}{\texttt{3.3.ah\_x\_abw}}}
  \end{subfigure}
  \caption{$a_1$-distributions of non-simple almost ordinary abelian
    threefolds of angle rank $2$.}
  \label{fig:non-simple-almost-ord-threefolds-rank-2}
\end{figure}

\begin{figure}[H]
  \centering
  \begin{subfigure}{0.24\textwidth}
    \centering
    \includegraphics[width=\textwidth]{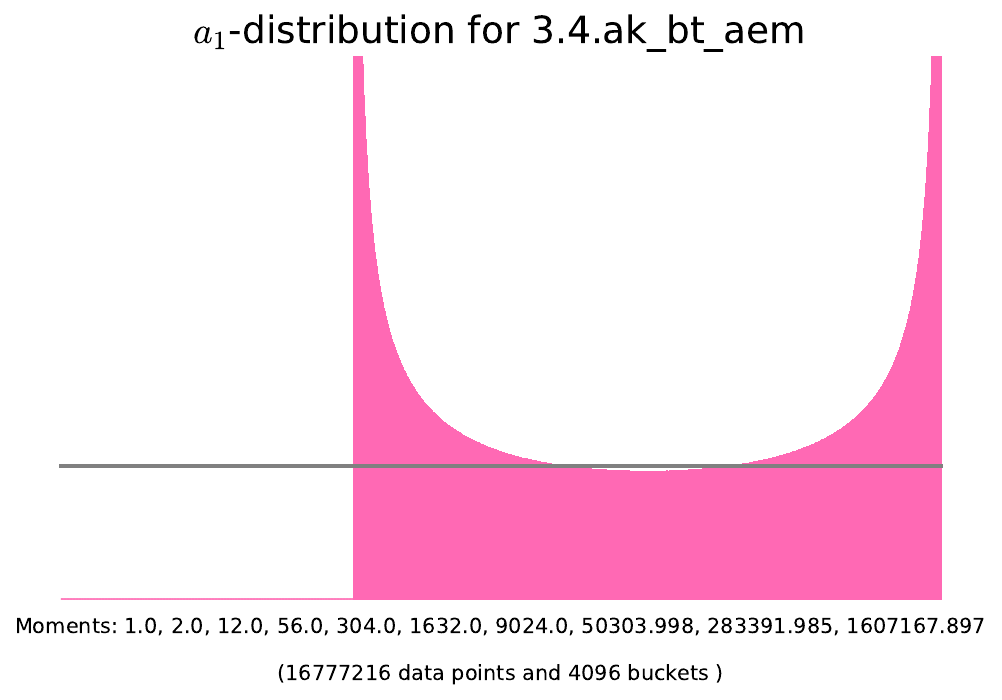}
    \caption{\href{https://www.lmfdb.org/Variety/Abelian/Fq/3/4/ak_bt_aem}{\texttt{3.4.ak\_bt\_aem}}}
  \end{subfigure}
  \hfill
  \begin{subfigure}{0.24\textwidth}
    \centering
    \includegraphics[width=\textwidth]{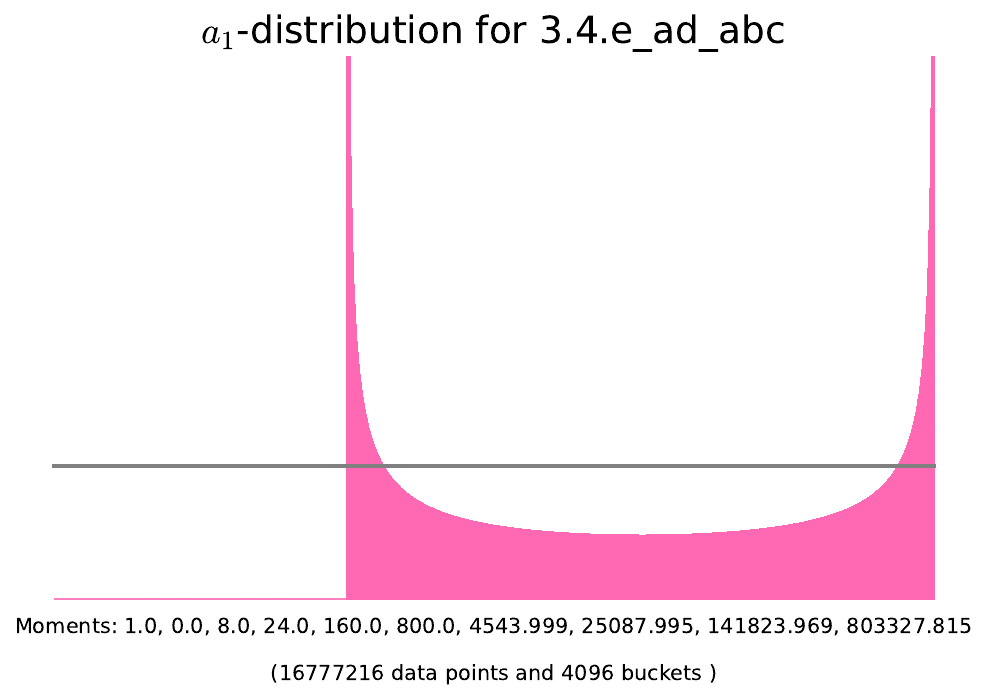}
    \caption{\href{https://www.lmfdb.org/Variety/Abelian/Fq/3/4/e_ad_abc}{\texttt{3.4.e\_ad\_abc}}}
  \end{subfigure}
  \hfill
  \begin{subfigure}{0.24\textwidth}
    \centering
    \includegraphics[width=\textwidth]{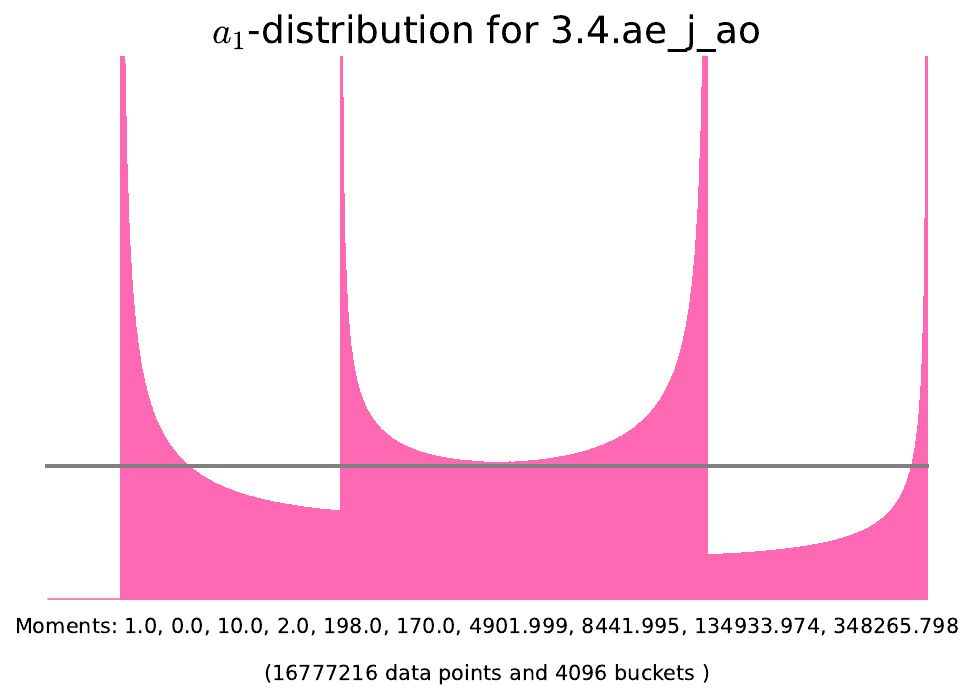}
    \caption{\href{https://www.lmfdb.org/Variety/Abelian/Fq/3/4/ae_j_ao}{\texttt{3.4.ae\_j\_ao}}}
  \end{subfigure}
  \hfill
  \begin{subfigure}{0.24\textwidth}
    \centering
    \includegraphics[width=\textwidth]{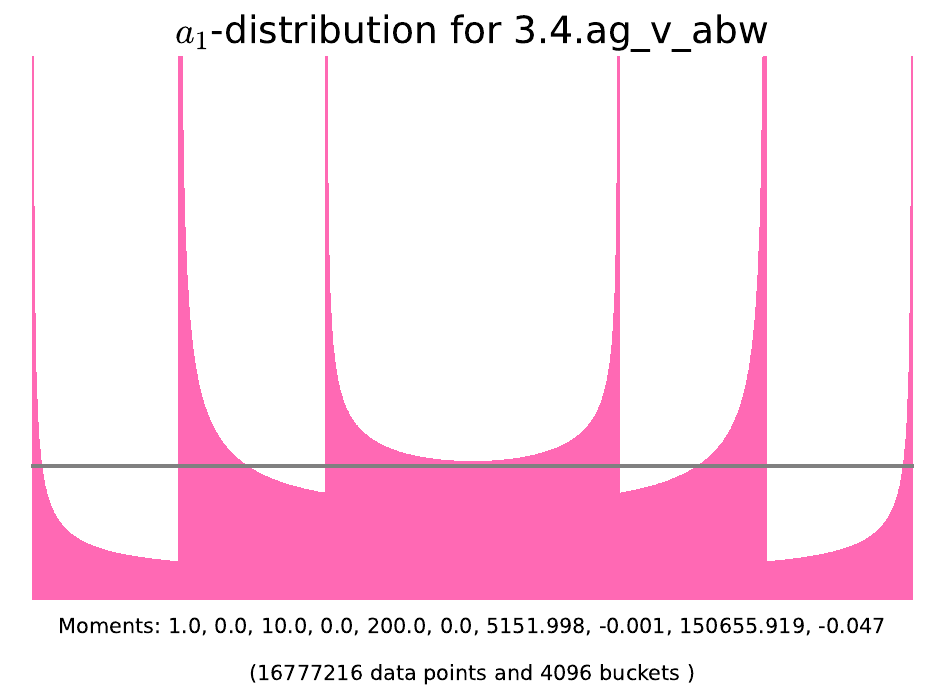}
    \caption{\href{https://www.lmfdb.org/Variety/Abelian/Fq/3/4/ag_v_abw}{\texttt{3.4.ag\_v\_abw}}}
  \end{subfigure}
  \hfill
  \begin{subfigure}{0.32\textwidth}
    \centering
    \includegraphics[width=0.8\textwidth]{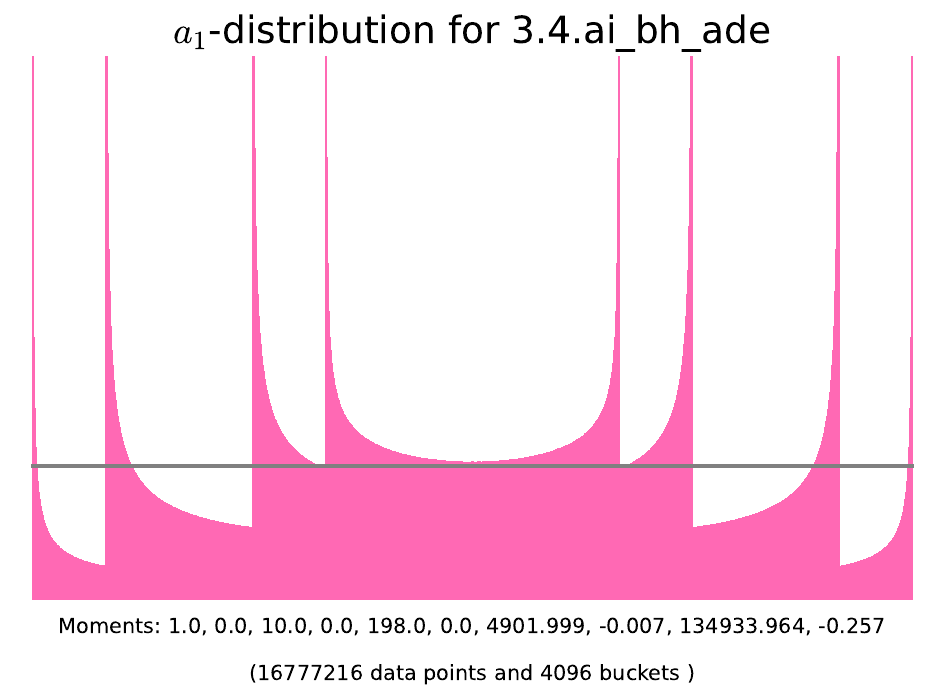}
    \caption{\href{https://www.lmfdb.org/Variety/Abelian/Fq/3/4/ai_bh_ade}{\texttt{3.4.ai\_bh\_ade}}}
  \end{subfigure}
  \hfill
  \begin{subfigure}{0.32\textwidth}
    \centering
    \includegraphics[width=0.8\textwidth]{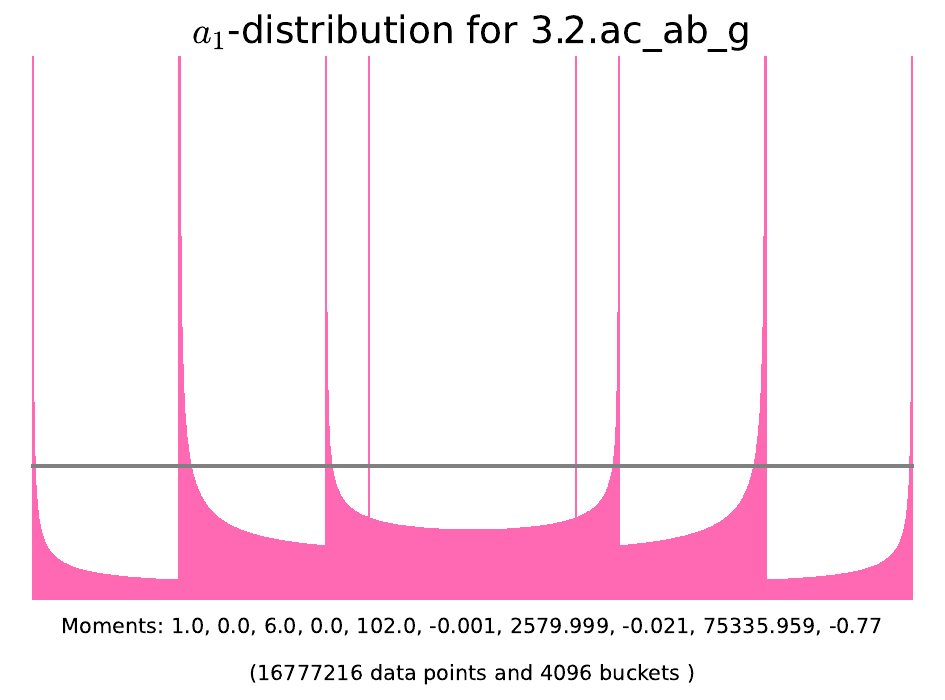}
    \caption{\href{https://www.lmfdb.org/Variety/Abelian/Fq/3/2/ac_ab_g}{\texttt{3.2.ac\_ab\_g}}}
  \end{subfigure}
  \hfill
  \begin{subfigure}{0.32\textwidth}
    \centering
    \includegraphics[width=0.8\textwidth]{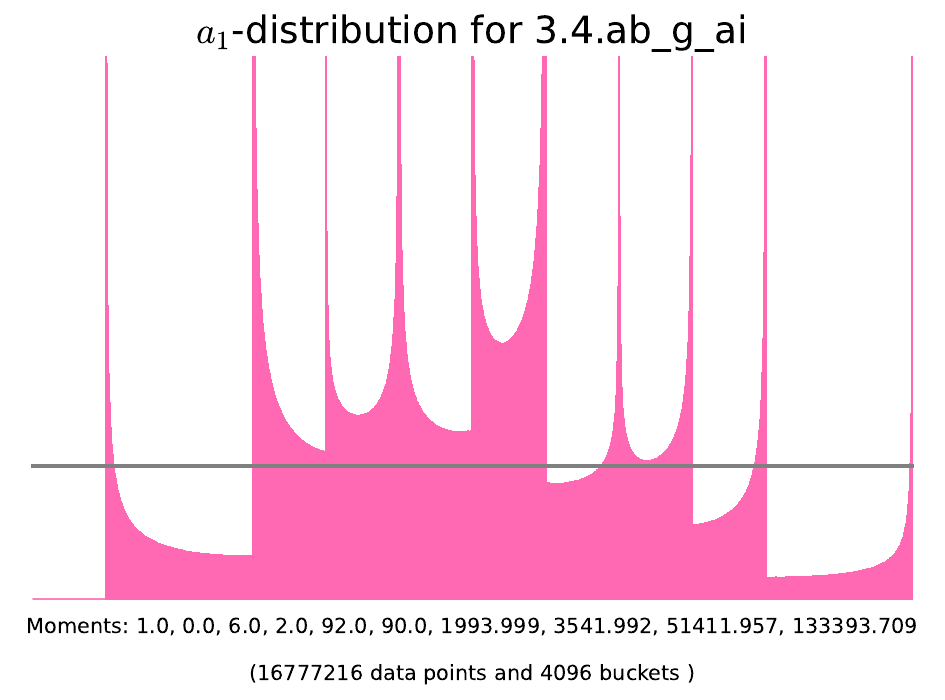}
    \caption{\href{https://www.lmfdb.org/Variety/Abelian/Fq/3/4/ab_g_ai}{\texttt{3.4.ab\_g\_ai}}}
  \end{subfigure}
  \hspace{2cm}
  \begin{subfigure}{0.32\textwidth}
    \centering
    \includegraphics[width=0.8\textwidth]{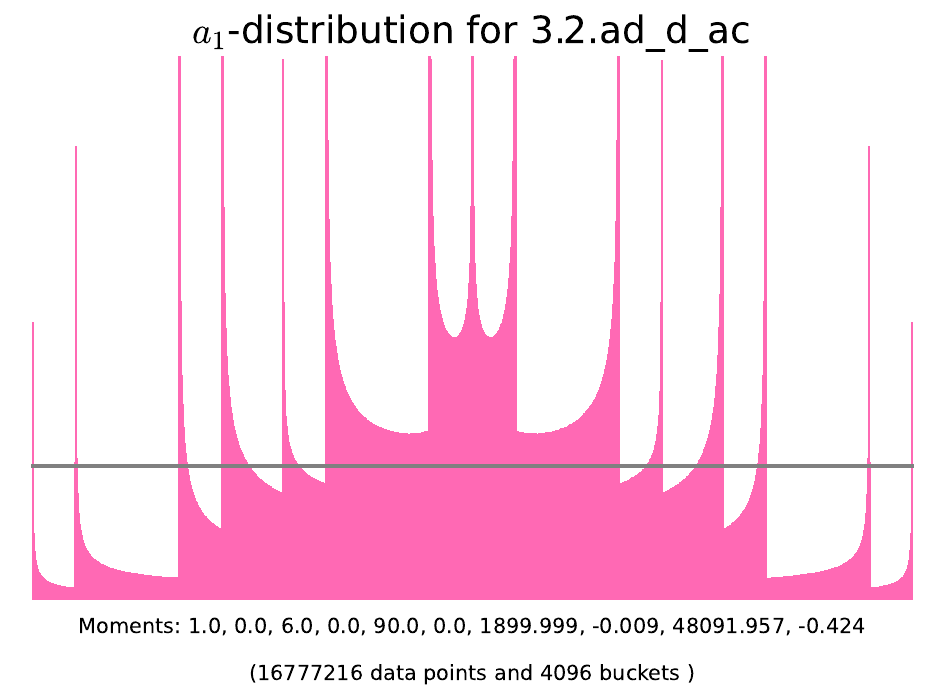}
    \caption{\href{https://www.lmfdb.org/Variety/Abelian/Fq/3/2/ad_d_ac}{\texttt{3.2.ad\_d\_ac}}}
  \end{subfigure}
  \caption{$a_1$-distributions of non-simple almost ordinary abelian
    threefolds of angle rank $1$.}
  \label{fig:non-simple-almost-ord-threefolds-rank-1}
\end{figure}

\subsection{Abelian threefolds of K3-type}
\label{sec:K3-type-X}
In this section $X$ will be an abelian threefold defined over $\FF_q$
of $p$-rank $1$. The $q$-Newton polygon of such a variety has slopes
$(0,\tfrac12,\tfrac12,\tfrac12,\tfrac12,1)$. This is the
three-dimensional instance of abelian varieties of K3 type, which were
studied by Zarhin in \cite{Zarhin1990K3} and \cite{Zarhin1991K3}.

\begin{defn}
  An abelian variety $A$ defined over $\FF_q$ is said to be of
  \cdef{K3-type} if the set of Newton slopes is either $\brk{0,1}$ or
  $\brk{0,1/2,1}$, and the segments of slope $0$ and $1$ have length
  one.
\end{defn}

By \cite[Theorem 5.9]{Zarhin1991K3}, simple abelian varieties of
K3-type have maximal angle rank. As a corollary, we have another piece
of the classification.

\begin{lemma}[Node X-F in \Cref{fig:proof-3}]
  Let $X$ be a simple abelian threefold over $\FF_q$ of $p$-rank
  $1$. Then $X$ has maximal angle rank and $\SF(X) \cong \UU(1)^3$.
\end{lemma}

There are several examples of such \(X\), one of them being
\href{https://www.lmfdb.org/Variety/Abelian/Fq/3/2/ab_a_a}{\texttt{3.2.ab\_a\_a}}. Now
assume that $X$ is not simple, so that $X \sim S \times E$ for some
surface $S$ and elliptic curve $E$.

\begin{lemma}[Node X-G in \Cref{fig:proof-3}]
  Let $X \sim S\times E$ be a non-simple abelian threefold over
  $\FF_q$ of $p$-rank $1$. The Serre--Frobenius group of $X$ is given
  by \Cref{table:p-rank-1-threefolds}.
\end{lemma}

We consider three cases:

\begin{enumerate}[label=(\thecounter-\alph*), leftmargin = 2cm]
\item \label{case:S-simple-ao} $S$ is simple and almost ordinary, and
  $E$ is supersingular.
\item \label{case:S-nonsimple-ao} $S$ is non-simple and almost
  ordinary, and $E$ is supersingular.
\item \label{case:E-ordinary} $S$ is supersingular and $E$ is
  ordinary.
\end{enumerate}

\begin{proof}
  As in Section \ref{sec:non-sim-ord-threefolds}, we let
  $\pi_1, \overline{\pi}_1, \pi_2, \overline{\pi}_2$ and
  $\alpha, \overline{\alpha}$ be the Frobenius eigenvalues of $S$ and
  $E$ respectively. Denote the normalized eigenvalues by
  $u_1 \colonequals \pi_1/\sqrt{q}, u_2 \colonequals \pi_2/\sqrt{q}$
  and $u \colonequals \alpha/\sqrt{q}$.

  Suppose first that $X$ is of type \ref{case:S-simple-ao}. By
  \Cref{lemma:simple-ao-surfaces}, the set $\brk{u_1,u_2}$ is
  multiplicatively independent. Since $u$ is a root of unity,
  $U_X = \langle u_1, u_2, u\rangle = U_S\oplus U_E \cong \ZZ^2\oplus
  C_m$ for $m \in M = \brk{1,2,3,4,6,8,12}$ the set of possible
  torsion orders for supersingular elliptic curves. Thus in this case,
  $\SF(X) \cong \UU(1)^2 \times C_m$ and is generated by
  $(u_1,u_2,\zeta_m)$.

  If $X$ is of type \ref{case:S-nonsimple-ao}, then
  $S\sim E_1\times E_2$ with $E_1$ ordinary and $E_2$
  supersingular. By \Cref{lemma:A1xB},
  $\SF(X) \cong \UU(1) \times C_m$, with $m$ in the set of possible
  torsion orders of non-simple supersingular surfaces.

  If $X$ is of type \ref{case:E-ordinary}, we have
  $U_X = U_E \oplus U_S \cong \ZZ\oplus C_m$ for $m$ in the set
  $M = \brk{1, 2,3,4,5,6,8,10,12,24}$ of possible torsion orders of
  supersingular surfaces from \Cref{lemma:simple-ss-surfaces} and
  \Cref{lemma:non-simple-SS-surfaces}.
\end{proof}

\begin{table}[H]
  \caption{Serre--Frobenius groups of abelian threefolds of $p$-rank
    1.}
  \setlength{\arrayrulewidth}{0.3mm} \setlength{\tabcolsep}{5pt}
  \renewcommand{\arraystretch}{1.2}
  \begin{longtable}{|c|c|c|c|c|}
    \hline
    \rowcolor{header_color} 
    Splitting type& $\cong$ class& Generator & $m \in M$ & Example\\ \hline
    Absolutely simple & $\UU(1)^3$ & $(u_1, u_2, u_3)$ & $\brk{1}$ & \href{https://www.lmfdb.org/Variety/Abelian/Fq/3/2/ab_a_a}{\texttt{3.2.ab\_a\_a}}  \\ \hline
    \ref{case:S-simple-ao} & $\UU(1)^2\times C_m$ & $(u_1, u_2, \zeta_{m_E})$ & $m = m_E \in \brk{1,2,3,4,6,8,12}$ & - \\ \hline
    \ref{case:S-nonsimple-ao} & $\UU(1)\times C_m$ & $(u_1, \zeta_{m_1}, \zeta_{m_2})$ & $m = \lcm(m_1,m_2)$ in \Cref{table:non-simple-ss-S-torsion} & - \\  \hline
    \ref{case:E-ordinary} & $\UU(1)\times C_m$ & $(u_1, \zeta_{m_1}, \zeta_{m_2})$ & $m = \lcm(m_1,m_2) \in \brk{1,2,3,4,5,6,8,10,12,24}$ & \Cref{fig:prank1-X-rank-1} \\ \hline
  \end{longtable}
  \addtocounter{table}{-1}
  \label{table:p-rank-1-threefolds}
\end{table}

The following examples are all of splitting type
\ref{case:E-ordinary}, since this splitting type contains all the new
Serre--Frobenius groups appearing in
\Cref{table:p-rank-1-threefolds}. The histograms corresponding to
these examples are presented in \Cref{fig:prank1-X-rank-1}.

\begin{figure}[H]
  \centering
  \begin{subfigure}{0.24\textwidth}
    \centering
    \includegraphics[width=1\textwidth]{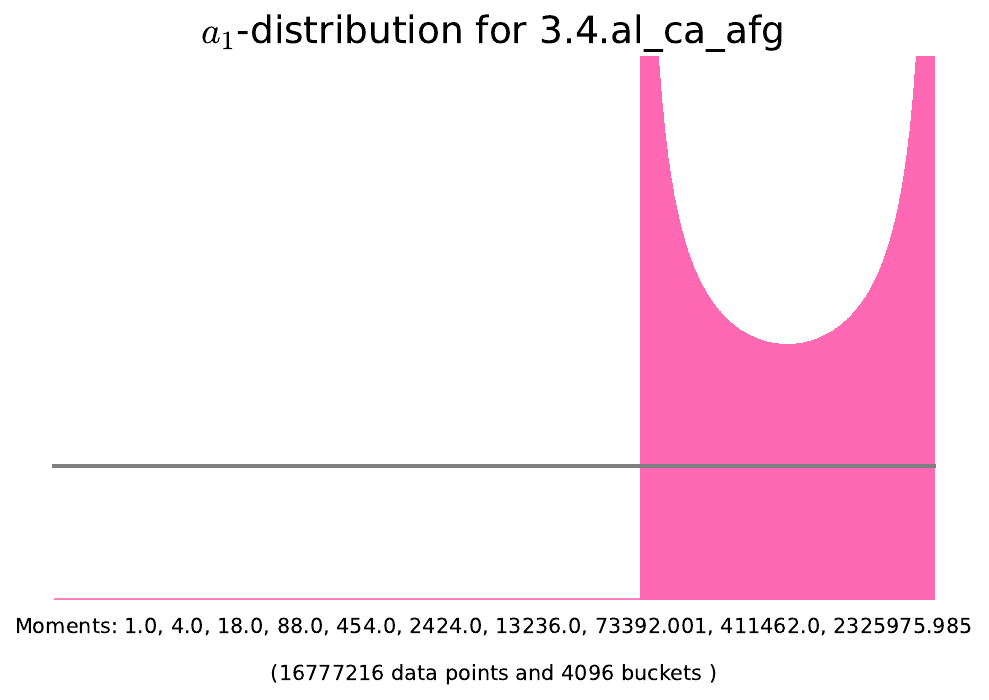}
    \caption{\href{https://www.lmfdb.org/Variety/Abelian/Fq/3/4/al_ca_afg}{\texttt{3.4.al\_ca\_afg}}}
  \end{subfigure}
  \hfill
  \begin{subfigure}{0.24\textwidth}
    \centering
    \includegraphics[width=\textwidth]{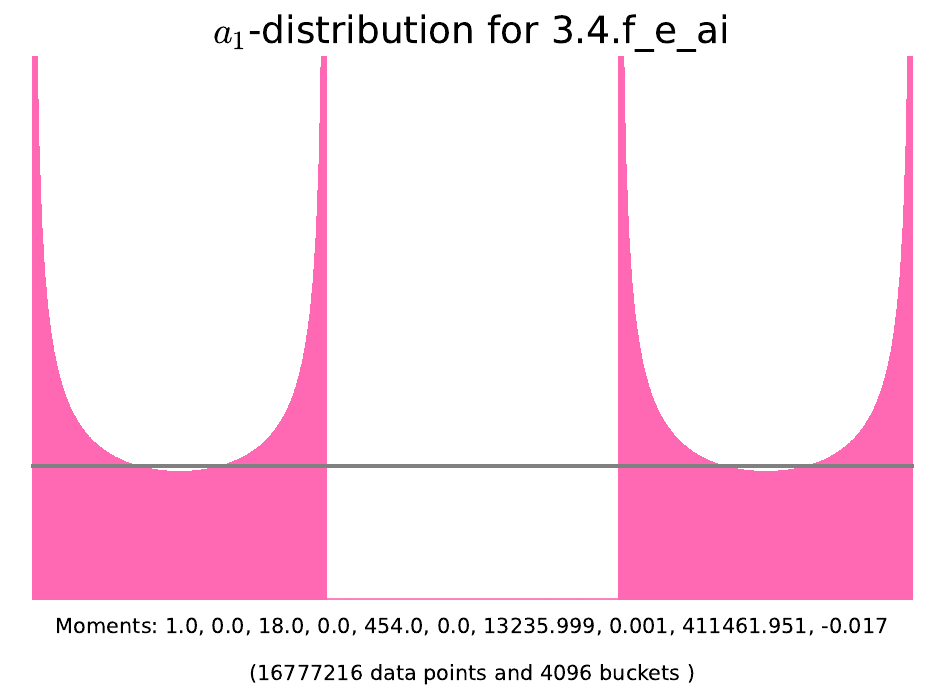}
    \caption{\href{https://www.lmfdb.org/Variety/Abelian/Fq/3/4/f_e_ai}{\texttt{3.4.f\_e\_ai}}}
  \end{subfigure}
  \hfill
  \begin{subfigure}{0.24\textwidth}
    \centering
    \includegraphics[width=\textwidth]{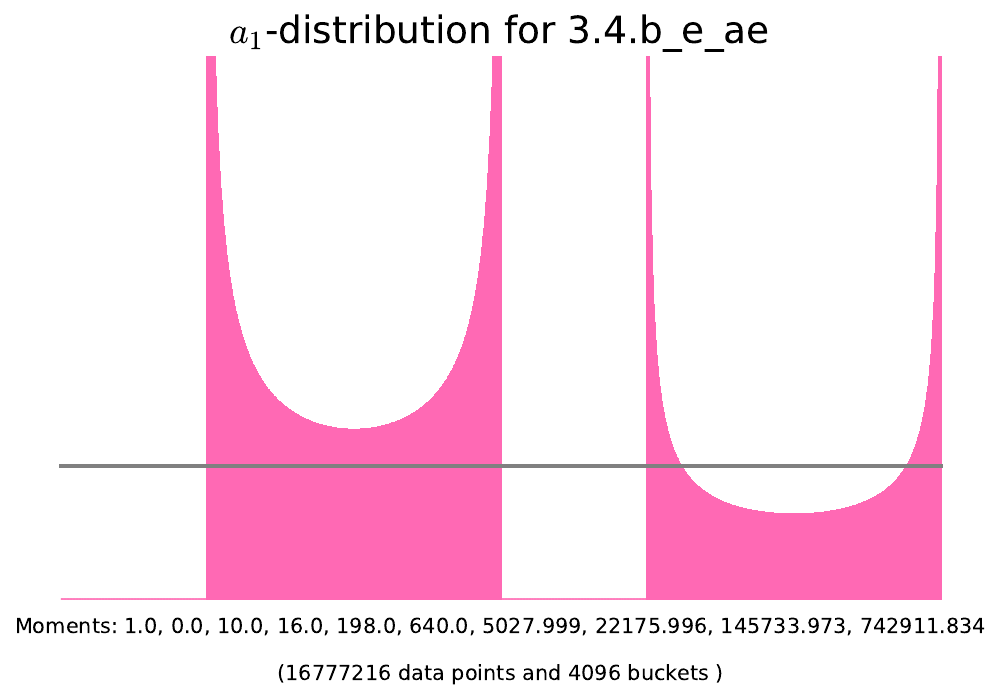}
    \caption{\href{https://www.lmfdb.org/Variety/Abelian/Fq/3/4/b_e_ae}{\texttt{3.4.b\_e\_ae}}}
  \end{subfigure}
  \hfill
  \begin{subfigure}{0.24\textwidth}
    \centering
    \includegraphics[width=\textwidth]{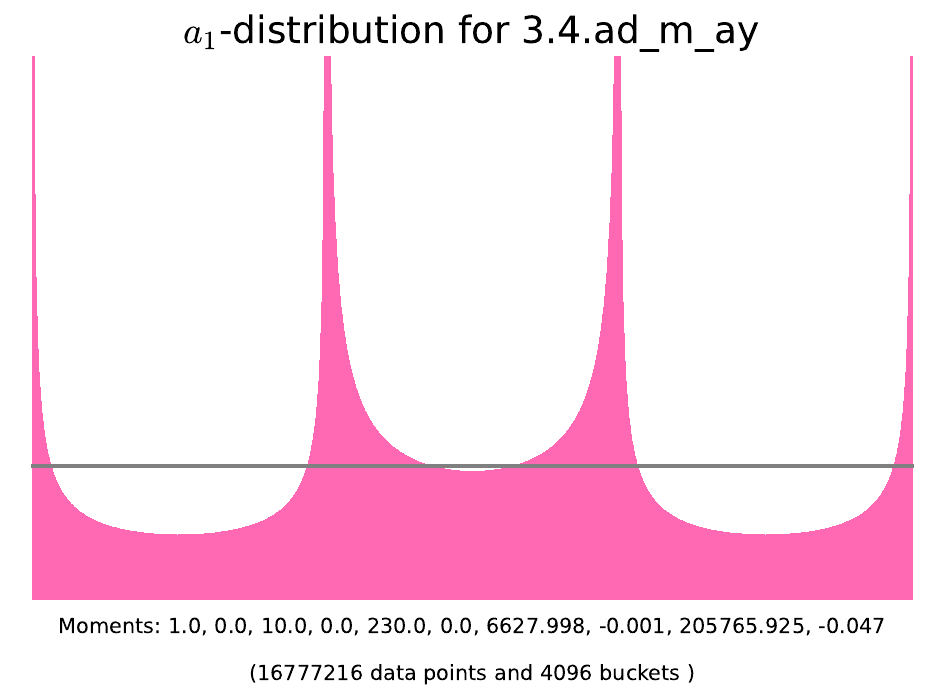}
    \caption{\href{https://www.lmfdb.org/Variety/Abelian/Fq/3/4/ad_m_ay}{\texttt{3.4.ad\_m\_ay}}}
  \end{subfigure}
\end{figure}
\begin{figure}[H]\ContinuedFloat
  \begin{subfigure}{0.24\textwidth}
    \centering
    \includegraphics[width=\textwidth]{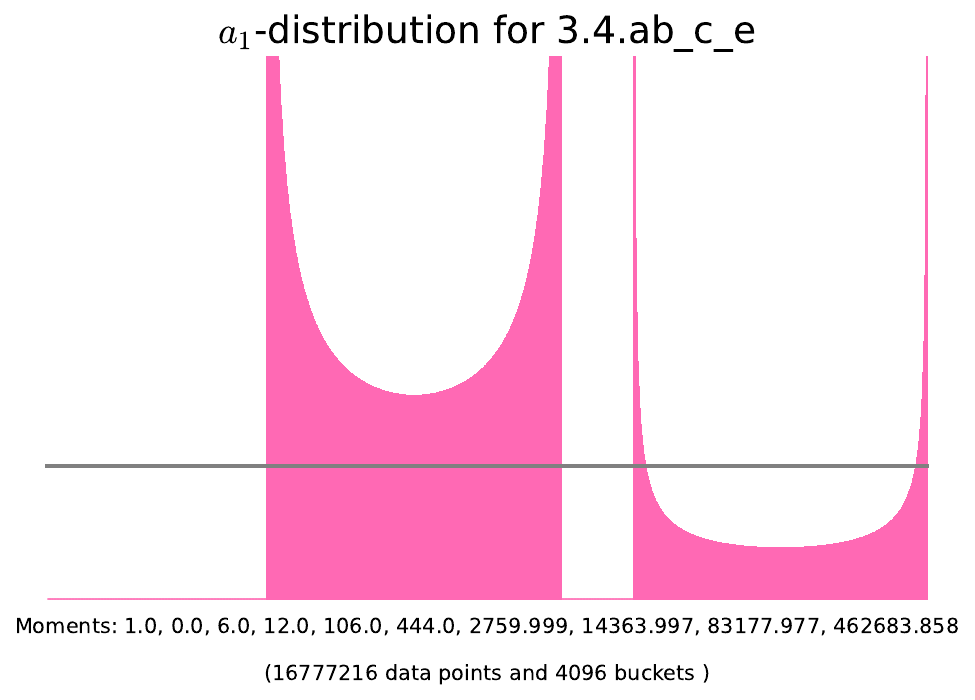}
    \caption{\href{https://www.lmfdb.org/Variety/Abelian/Fq/3/4/ab_c_e}{\texttt{3.4.ab\_c\_e}}}
  \end{subfigure}
  \hfill
  \begin{subfigure}{0.24\textwidth}
    \centering
    \includegraphics[width=\textwidth]{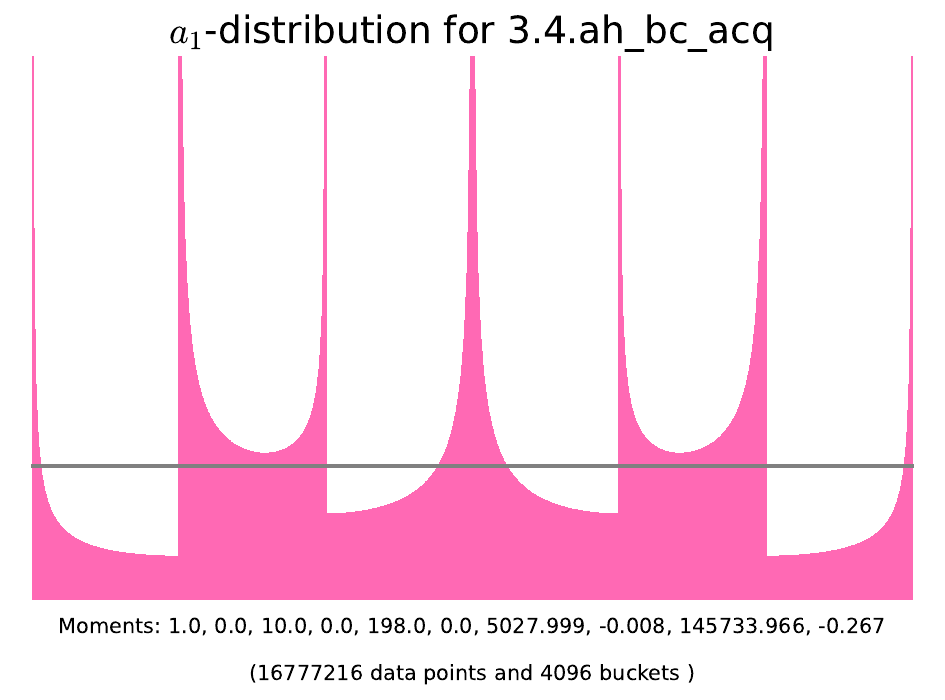}
    \caption{\href{https://www.lmfdb.org/Variety/Abelian/Fq/3/4/ah_bc_acq}{\texttt{3.4.ah\_bc\_acq}}}
  \end{subfigure}
  \hfill
  \begin{subfigure}{0.24\textwidth}
    \centering
    \includegraphics[width=\textwidth]{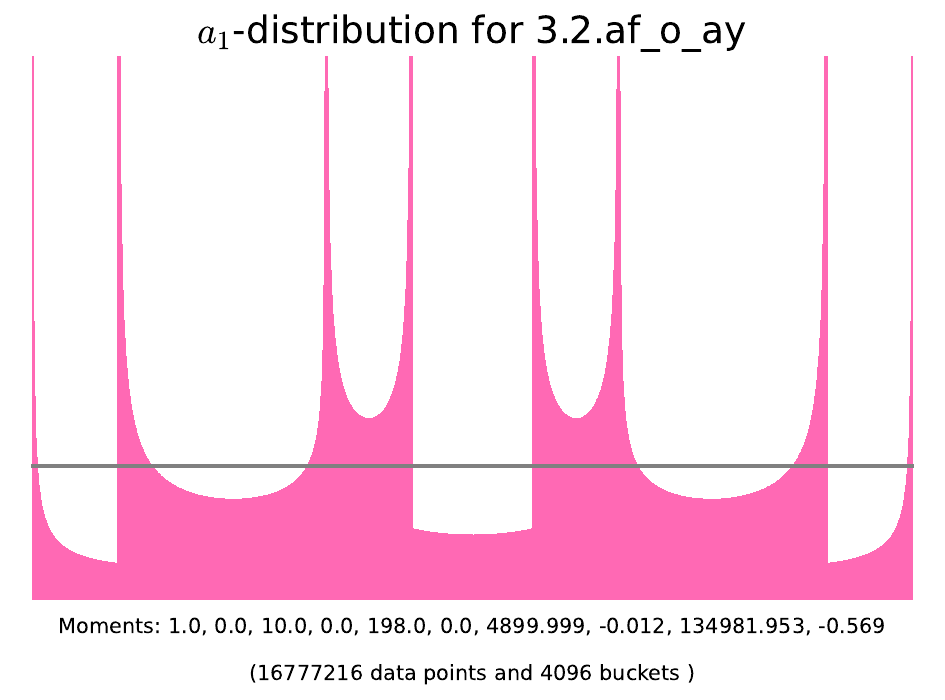}
    \caption{\href{https://www.lmfdb.org/Variety/Abelian/Fq/3/2/af_o_ay}{\texttt{3.2.af\_o\_ay}}}
  \end{subfigure}
  \hfill
  \begin{subfigure}{0.24\textwidth}
    \centering
    \includegraphics[width=\textwidth]{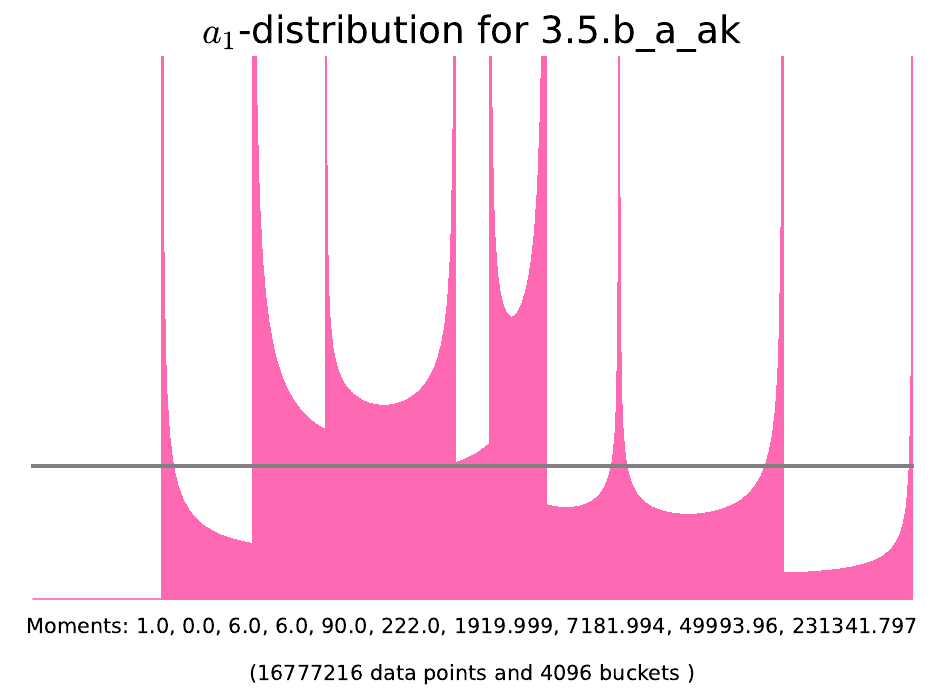}
    \caption{\href{https://www.lmfdb.org/Variety/Abelian/Fq/3/5/b_a_ak}{\texttt{3.5.b\_a\_ak}}}
  \end{subfigure}
  \hfill
  \begin{subfigure}{0.32\textwidth}
    \centering
    \includegraphics[width=0.8\textwidth]{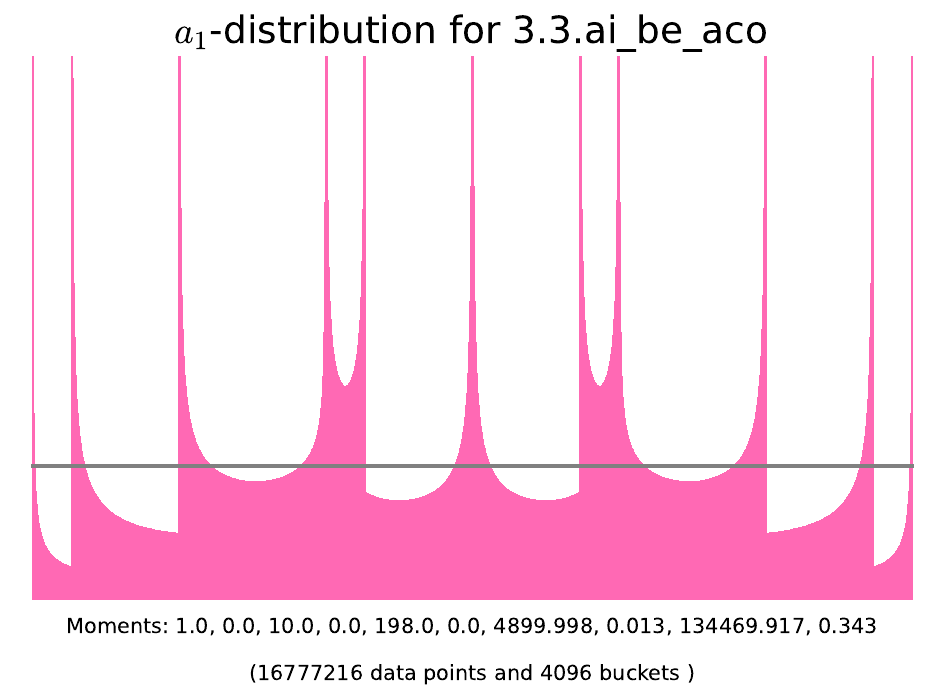}
    \caption{\href{https://www.lmfdb.org/Variety/Abelian/Fq/3/3/ai_be_aco}{\texttt{3.3.ai\_be\_aco}}}
  \end{subfigure}
  \hspace{2cm}
  \begin{subfigure}{0.32\textwidth}
    \centering
    \includegraphics[width=0.8\textwidth]{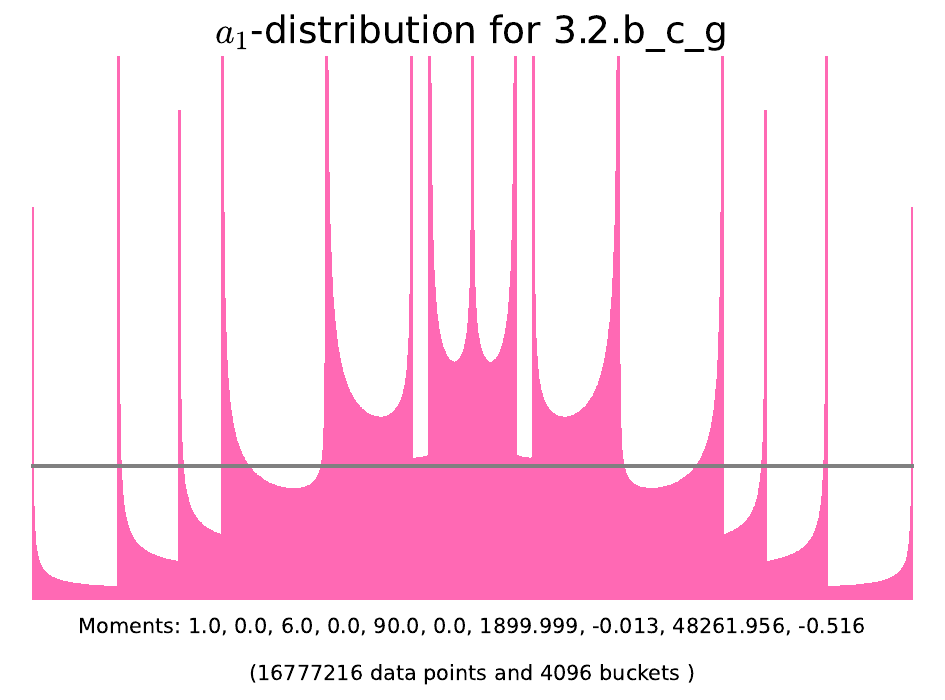}
    \caption{\href{https://www.lmfdb.org/Variety/Abelian/Fq/3/2/b_c_g}{\texttt{3.2.b\_c\_g}}}
  \end{subfigure}
  \caption{$a_1$-distributions of $p$-rank $1$ threefolds with angle
    rank 1}
  \label{fig:prank1-X-rank-1}
\end{figure}

\subsection{Absolutely simple p-rank 0 threefolds}\label{sec:absolutely-simple-prank0}
In this section, $X$ will be a non-supersingular $p$-rank $0$ abelian
threefold over $\FF_q$. Since the $q$-Newton polygon of the Frobenius
polynomial $P(T) = P_X(T)$ has slopes $\tfrac13$ and $\tfrac23$, each
with multiplicity three, it follows that $X$ is absolutely simple,
since the slope $1/3$ does not occur for abelian varieties of smaller
dimension. Let $e_r^2$ denote the dimension of
$\mathrm{End}^0(X_{(r)})$ over its center. We consider two cases:
\begin{enumerate}[label=(\thesubsection-\alph*), leftmargin = 2cm]
\item \label{case:xing} There exists $r\geq 1$ such that $e_r = 3$. In
  this case we have $P_{(r)}(T) = h_{(r)} (T)^3$ and $h_{(r)}(T)$ is
  as in \Cref{thm:xing}, so that $3$ divides $r\cdot\log_p(q)$.
\item \label{case:not-xing} $e_r = 1$ for every positive integer $r$.
\end{enumerate}

\begin{lemma}[Node X-H in \Cref{fig:proof-3}]
  Let $X$ be an absolutely simple abelian threefold of $p$-rank $0$
  defined over $\FF_q$. Then, the Serre--Frobenius group of $X$ is
  classified according to \Cref{table:X-prank0}. Furthermore, $X$ is
  of type \ref{case:xing}, $m_X$ is the smallest positive integer $r$
  such that $e_r = 3$.
\end{lemma}
\begin{table}[H]
  \setlength{\arrayrulewidth}{0.3mm} \setlength{\tabcolsep}{5pt}
  \renewcommand{\arraystretch}{1.2}
  \caption{Serre--Frobenius groups of absolutely simple abelian
    threefolds of $p$-rank $0$.}
  \begin{longtable}{|c|c|c|c|c|}
    \hline
    \rowcolor{header_color} 
    Case & $q = p^d$  & $\cong$ class & Generator & $m \in M$ \\ \hline
    \ref{case:xing}  & $3 \mid m_X\cdot d$ & $\UU(1) \times C_m $ & $(u_1, \zeta_mu_1, \zeta_m^\nu u_1)$ & $\brk{1,3,7}$ \\ \hline
    \ref{case:not-xing} & - & $\UU(1)^3$ & $(u_1,u_2,u_3)$ & $\brk{1}$   \\ \hline
  \end{longtable}
  \addtocounter{table}{-1}
  \label{table:X-prank0}
\end{table}

\begin{remark}
  The techniques for proving the Generalized Lenstra--Zarhin result in
  \cite[Theorem 1.5]{dupuy2022angle}, cannot be applied to this
  case. Thus, even the angle rank analysis in this case is
  particularly interesting.
\end{remark}

\begin{proof}
  Suppose first that $X$ is of type \ref{case:xing}, and let $m$ be
  the minimal positive integer such that $e_m = 3$. Maintaining
  previous notation, $P_{(m)}(T) = h_{(m)}(T)^3$ implies that
  $\alpha_2 = \zeta\cdot \alpha_1$ and $\alpha_3 = \xi\cdot \alpha_1$
  for $m$-th roots of unity $\zeta$ and $\xi$, whose orders have lcm
  $m$.  By \Cref{lemma:equiv-char-of-SFgroup}, this implies that
  $\SF(X) \cong \UU(1)\times C_m$. We conclude that $\delta_X = 1$ and
  $m = m_X$. To calculate the set $M$ of possible torsion orders,
  assume that $m_X = m > 1$. Then $\QQ(\alpha_1^m)$ is a quadratic
  imaginary subextension of $\QQ(\alpha_1)\supset \QQ$, and we can
  argue as in the proof of \Cref{thm:simple-ordinary-prime-splitting}
  (with $\ell =3$) to conclude that $m \in \brk{3,7}$.

  Assume now that $X$ is of type \ref{case:not-xing}. This implies
  that $\QQ(\alpha_1^r)$ is a degree $6$ CM-field for every positive
  integer $r$. If $m := m_X$, the base extension $X_{(m)}$ is neat and
  the Frobenius eigenvalues are distinct and not supersingular. By
  Remark \ref{rmk:neat2} we have that $\delta_X = 3$ and $m = 1$.
\end{proof}

\begin{example}[$a_1$-distribution for $p$-rank 0 non-supersingular
  threefolds of splitting type \ref{case:xing}] The histograms
  corresponding to these examples are presented in
  \Cref{fig:xing-type-threefolds}. Note that the first one already
  showed up in \Cref{fig:non-simple-ord-threefolds-i}, while the other
  ones appeared in \Cref{fig:hist-simple-ord-X}.
  \begin{enumerate}[label=(\Alph*)]
  \item ($m=1$) The isogeny class
    \href{https://www.lmfdb.org/Variety/Abelian/Fq/3/8/ag_bk_aea}{\texttt{3.8.ag\_bk\_aea}}
    satisfies $m_X=1$. Note that $3$ divides $m_X\cdot \log_2(8)$.
    
  \item ($m=3$) The isogeny class
    \href{https://www.lmfdb.org/Variety/Abelian/Fq/3/2/a_a_ac}{\texttt{3.2.a\_a\_ac}}
    has angle rank $1$ and irreducible Frobenius polynomial
    $P(T) = T^6 - 2T^3 + 8$. The cubic base extension gives the
    isogeny class
    \href{https://www.lmfdb.org/Variety/Abelian/Fq/3/8/ag_bk_aea}{\texttt{3.8.ag\_bk\_aea}}
    with reducible Frobenius polynomial
    $P_{(3)}(T) = (T^6 - 2T^3 + 8)^3$. Note that $3$ divides
    $m_X\cdot \log_2(2)$.
    
  \item ($m = 7$) The isogeny class
    \href{https://www.lmfdb.org/Variety/Abelian/Fq/3/8/ai_bk_aeq}{\texttt{3.8.ai\_bk\_aeq}}
    has angle rank $1$ and irreducible Frobenius polynomial
    $P(T) = T^6 - 8T^5 + 36T^4 - 120T^3 + 288T^2 - 512T + 512$. Its
    base change over a degree $m_X = 7$ extension is the isogeny class
    \texttt{3.2097152.ahka\_bfyoxc\_adesazpwa} with Frobenius
    polynomial
    \begin{equation*}
      P_{(7)}(T) = (T^2 - 1664T + 2097152)^3.
    \end{equation*}
    In this example, $q = 8$, so that $3$ divides
    $m_X\cdot \log_2(8)$.
  \end{enumerate}
\end{example}

\begin{figure}[h]
  \centering
  \begin{subfigure}{0.32\textwidth}
    \centering
    \includegraphics[width=0.8\textwidth]{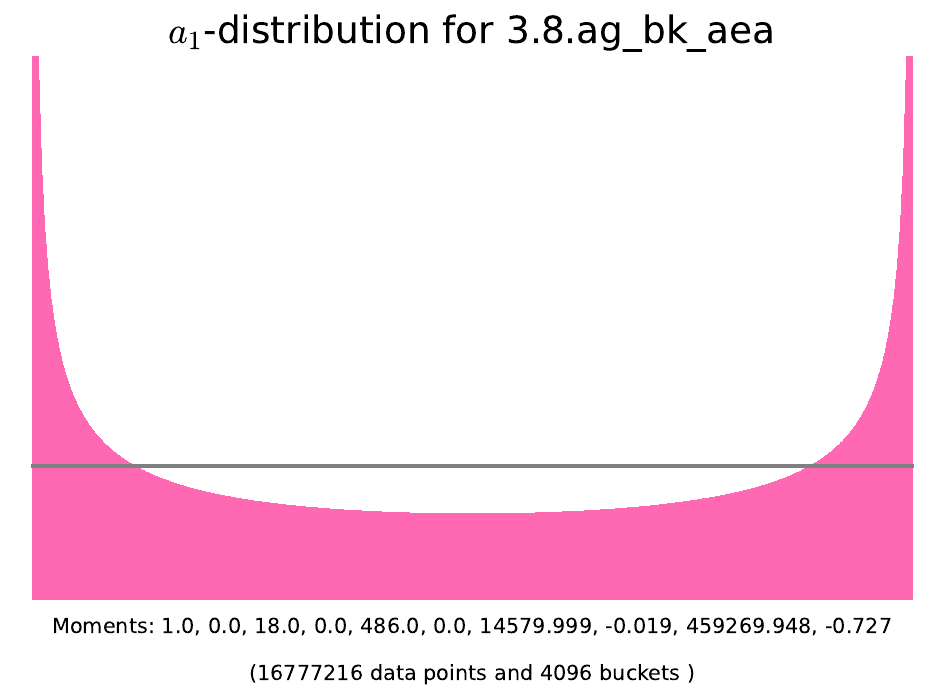}
    \caption{
      \href{https://www.lmfdb.org/Variety/Abelian/Fq/3/8/ag_bk_aea}{\texttt{3.8.ag\_bk\_aea}}}
  \end{subfigure}
  \hfill
  \begin{subfigure}{0.32\textwidth}
    \centering
    \includegraphics[width=0.8\textwidth]{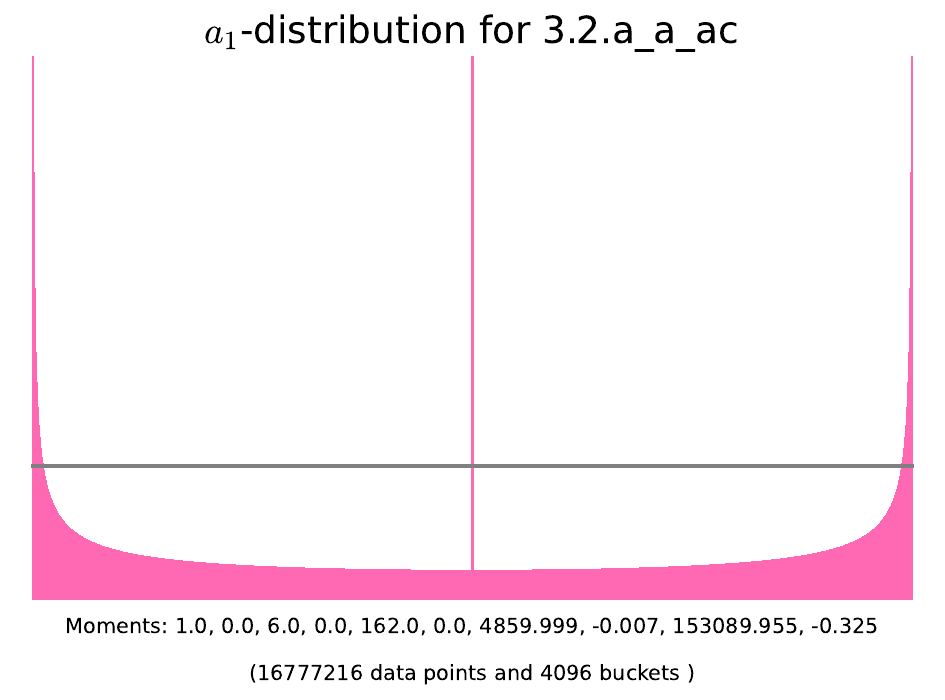}
    \caption{
      \href{https://www.lmfdb.org/Variety/Abelian/Fq/3/2/a_a_ac}{\texttt{3.2.a\_a\_ac}}}
  \end{subfigure}
  \hfill
  \begin{subfigure}{0.32\textwidth}
    \centering
    \includegraphics[width=0.8\textwidth]{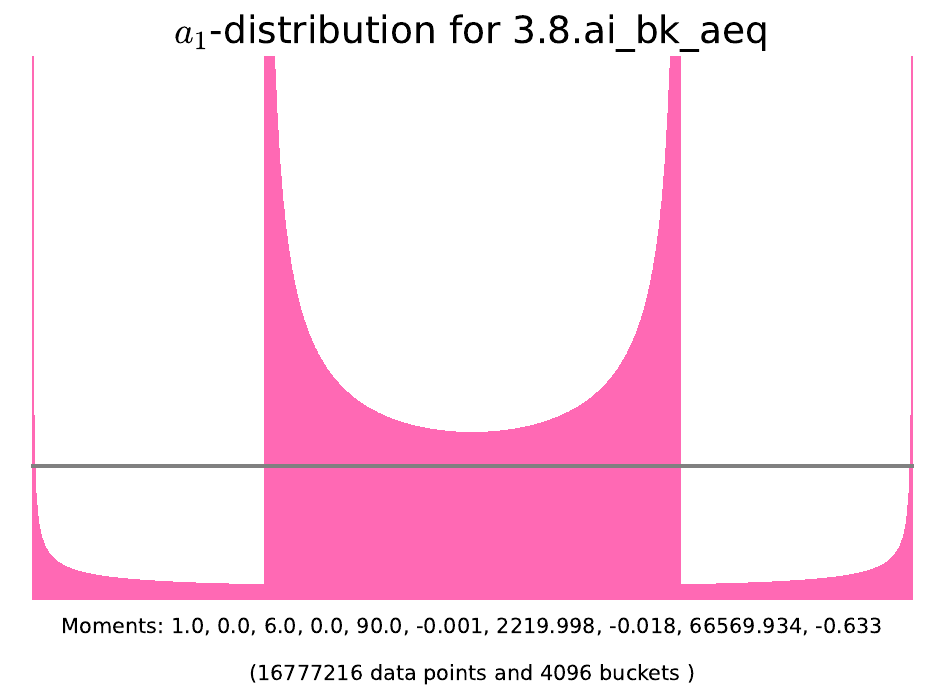}
    \caption{
      \href{https://www.lmfdb.org/Variety/Abelian/Fq/3/8/ai_bk_aeq}{\texttt{3.8.ai\_bk\_aeq}}}
  \end{subfigure}
  \caption{$a_1$-distribution for $p$-rank 0 non-supersingular
    threefolds of splitting type \ref{case:xing}.}
  \label{fig:xing-type-threefolds}
\end{figure}

\subsection{Simple supersingular threefolds}
Nart and Ritzenthaler \cite{nart&ritz08} showed that the only degree
$6$ supersingular $q$-Weil numbers are the conjugates of
\begin{align*}
  \pm\sqrt{q}\zeta_7, \pm\sqrt{q}\zeta_9, \quad \text{ when } q \text{ is a square, and} \\
  7^{d/2}\zeta_{28}, 3^{d/2}\zeta_{36}, \quad \text{ when } q \text{ is not a square.}
\end{align*}

Building on their work, Haloui \cite[Proposition 1.5]{haloui}
completed the classification of simple supersingular threefolds. This
classification is also discussed in \cite{Singh&McGuire&Zaytsev2014};
and we adapt their notation for the polynomials of Z-3 type. Denoting
by $(a_1, a_2, a_3)$ the isogeny class of abelian threefolds over
$\FF_q$ with Frobenius polynomial
$P_X(T) = T^6 + a_1T^5+ a_2T^4 + a_3T^3+ qa_2T^2 + q^2a_1T + q^3$, the
following lemma gives the classification of Serre--Frobenius groups of
simple supersingular threefolds, which is a corollary of Haloui's
result.

\begin{lemma}[Node X-I in \Cref{fig:proof-3}]
  Let $X$ be a simple supersingular abelian threefold defined over
  $\FF_q$. The Serre--Frobenius group of $X$ is classified according
  to \Cref{table:ss-simple-threefolds}.
\end{lemma}

\begin{table}[H]
  \setlength{\arrayrulewidth}{0.3mm} \setlength{\tabcolsep}{5pt}
  \renewcommand{\arraystretch}{1.3}
  \caption{Serre--Frobenius groups of simple supersingular
    threefolds.}
  \begin{longtable}{|c|c|c|c|c|c|c|}
    \hline
    \rowcolor{header_color} 
    $(a_1, a_2, a_3)$       & $p$  & $d$  & Type & $\tilde{h}(T)$ & $\SF(X)$ & Example \\ \hline
    $(\sqrt{q}, q, q\sqrt{q})$ & $7 \nmid (p^3-1) $   & even & Z-1  & $\Phi_7(T)$                  & $C_7$  &  \href{https://www.lmfdb.org/Variety/Abelian/Fq/3/9/d_j_bb}{\texttt{3.9.d\_j\_bb}}  \\ \hline
    $ (-\sqrt{q}, q, -q\sqrt{q}) $  & $7 \nmid (p^3-1)$               & even   & Z-1  & $\Phi_{14}(T)$                  & $C_{14}$ & \href{https://www.lmfdb.org/Variety/Abelian/Fq/3/9/ad_j_abb}{\texttt{3.9.ad\_j\_abb}}   \\ \hline
    $(0,0,q\sqrt{q})$    & $\not\equiv 1 \md 3$                     & even  & Z-1  & $\Phi_9(T)$                & $C_9$ & \href{https://www.lmfdb.org/Variety/Abelian/Fq/3/4/a_a_i}{\texttt{3.4.a\_a\_i}}   \\ \hline
    $(0,0,-q\sqrt{q})$    & $\not\equiv 1 \md 3$ & even & Z-1  & $\Phi_{18}(T)$               & $C_{18}$ & \href{https://www.lmfdb.org/Variety/Abelian/Fq/3/4/a_a_ai}{\texttt{3.4.a\_a\_ai}}\\ \hline
    $(\sqrt{7q}, 3q, q\sqrt{7q})$           & $=7$               & odd & Z-3  & $h_{7,1}(T)$ & $C_{28}$ & \href{https://www.lmfdb.org/Variety/Abelian/Fq/3/7/h_v_bx}{\texttt{3.7.h\_v\_bx}} \\ \hline
    $(-\sqrt{7q}, 3q, -q\sqrt{7q})$     & $=7$   & odd & Z-3  & $h_{7,1}(-T)$                  & $C_{28}$ &  \href{https://www.lmfdb.org/Variety/Abelian/Fq/3/7/ah_v_abx}{\texttt{3.7.ah\_v\_abx}}   \\ \hline
    $(0,0,q\sqrt{3q})$    & $=3$   & odd & Z-3  & $h_{3,3}(T)$  & $C_{36}$ & \href{https://www.lmfdb.org/Variety/Abelian/Fq/3/3/a_a_j}{\texttt{3.3.a\_a\_j}}\\ \hline
    $(0,0,-q\sqrt{3q})$    & $=3$   & odd & Z-3  & $h_{3,3}(-T)$  & $C_{36}$ & \href{https://www.lmfdb.org/Variety/Abelian/Fq/3/3/a_a_aj}{\texttt{3.3.a\_a\_aj}} \\ \hline
  \end{longtable}
  \addtocounter{table}{-1}
  \label{table:ss-simple-threefolds}
\end{table}

\begin{proof}
  By \Cref{thm:xing} and the discussion following it, we know that the
  Frobenius polynomial of every supersingular threefold $P_X(T)$,
  coincides with the minimal polynomial $h_X(T)$ with $e=1$ in some
  row of the table.  The first four rows of
  \Cref{table:ss-simple-threefolds} correspond to isogeny classes of
  type (Z-1). By the discussion in \Cref{sec:ss-qST-groups}, the
  minimal polynomials are of the form\footnote{Recall that
    $f^{[a]}(T) \colonequals a^{\deg f} f(T/a)$.}
  $\Phi_m^{[\sqrt{q}]}(T)$ and the normalized polynomials are just the
  usual cyclotomic polynomials $\Phi_m(T)$.

  The last four rows of \Cref{table:ss-simple-threefolds} correspond
  to isogeny classes of type (Z-3). The normalized Frobenius
  polynomials are
  $h_{7,1}(\pm T) = \, T^6 \pm \sqrt{7}T^5 + 3T^4 \pm \sqrt{7}T^3
  +3T^2 \pm \sqrt{7}T + 1$, and
  $h_{3,3}(\pm T) = \, T^6 \pm \sqrt{3}T^3 + 1.$ Noting that
  $h_{7,1}(T)h_{7,1}(-T) = \Phi_{28}(T)$ and
  $h_{3,3}(T)h_{3,3}(-T) = \Phi_{36}(T)$ we conclude that the unit
  groups $U_X$ are generated by $\zeta_{28}$ and $\zeta_{36}$
  respectively.
\end{proof}

\subsection{Non-simple supersingular threefolds}
\label{subsec:non-simple-ss-3fold}

If $X$ is a non-simple supersingular abelian threefold over $\FF_q$,
then there are two cases:
\begin{enumerate}[label=(\thecounter-\alph*), leftmargin = 2cm]
\item \label{case:X-SS-SxE} $X \sim S \times E$, with $S$ a simple
  supersingular surface over $\FF_q$ and $E$ a supersingular elliptic
  curve.
\item \label{case:X-SS-ExExE} $X \sim E_1 \times E_2 \times E_3$,
  where each $E_i$ is a supersingular elliptic curve.
\end{enumerate}

The classification of the Serre--Frobenius group in these cases can be
summarized in the following lemma.

\begin{lemma}[Node X-J in \Cref{fig:proof-3}]
  If $X$ is a non-simple supersingular abelian threefold as in Case
  \ref{case:X-SS-SxE}, then $\SF(X) \cong C_m$, for $m \in M(p,d)$,
  where
  \begin{itemize}
  \item if $d$ is even,
    $M(p,d) = \{3, 4, 5, 6, 8, 10, 12, 15, 20, 24, 30 \}$, and
  \item if $d$ is odd, $M(p,d) = \{{4, 8, 12, 20, 24} \}$.
  \end{itemize}
\end{lemma}

\begin{proof}
  In this case, $m = \lcm(m_S, m_E)$, since this is the degree of the
  smallest extension over which the Serre--Frobenius group becomes
  connected. The list of values for $m_E$ and $m_S$ come from
  \Cref{table:elliptic-curves} and \Cref{table:ss-simple-surfaces}.
\end{proof}

\begin{lemma}[Node X-J in \Cref{fig:proof-3}]
  If $X$ is a non-simple supersingular abelian threefold as in Case
  \ref{case:X-SS-ExExE}, then $\SF(X) \cong C_m$, for $m \in M(p,d)$,
  where
  \begin{itemize}
  \item if $d$ is even, $M(p,d) = \{1, 2, 3, 4, 6, 12\}$, and
  \item if $d$ is odd, $M(p,d) = \{4, 8, 12\}$.
  \end{itemize}
\end{lemma}

\begin{proof}
  We know that $m$ is the degree of the extension over which all the
  elliptic curve factors $E_i$ become isogenous.  This is precisely
  the least common multiple of the $m_{E_i}$'s. From
  \Cref{table:elliptic-curves}, we can calculate the various
  possibilities for the $\lcm$'s depending on the parity of $d$.
\end{proof}

\subsection{Acknowledgements}

We would like to thank David Zureick-Brown, Kiran Kedlaya, Francesc
Fit\'{e}, Brandon Alberts, Edgar Costa, and Andrew Sutherland for
useful conversations about this paper. We thank Yuri Zarhin for
providing us with useful references, and Hendrik Lenstra for pointing
out one of the missing cases in \Cref{sec:elliptic-curves}. We would
also like to thank Everett Howe for helping us with a missing piece of
the puzzle in \Cref{thm:simple-ordinary-prime-splitting}. This project
started as part of the Rethinking Number Theory workshop in 2021. We
thank the organizers of the workshop for giving us the opportunity and
space to collaborate, and the funding sources for the workshop: AIM,
the Number Theory Foundation, and the University of Wisconsin-Eau
Claire Department of Mathematics. We are also grateful to Rachel Pries
for her guidance at the beginning of the workshop, which helped launch
this project. Finally, we thank the anonymous referee for the
elucidating and pertinent suggestions that improved the exposition and
results in the paper.

\bibliography{refs.bib}{} \bibliographystyle{chicago}
\end{document}